\documentclass[10pt,a4paper]{article}

\usepackage{array}
\usepackage[dvips]{graphicx}
\usepackage{amssymb}
\usepackage{amsmath}
\usepackage{amsthm}
\usepackage{color}
\usepackage{bm}
\usepackage{bbm}
\usepackage{sublabel}
\usepackage{rotating}

\newcommand\E{{\rm E}}
\newcommand\Pp{{\rm P}}
\newcommand\cov{{\rm cov}}
\newcommand\var{{\rm var}}
\newcommand\re{{\rm e}}
\newcommand\Exp{{\rm Exp}}
\newcommand\MBin{{\rm MixBin}}
\newcommand\dmax{d_{\max}}
\newcommand{\ind}[1]{\mathbbm{1}_{\{#1\}}}

\newcommand\SigmaMRGC{\sigma^2_{\rm MRGC}}
\newcommand\SigmaNSWGC{\sigma^2_{\rm NSWGC}}
\newcommand\SigmaMR{\sigma^2_{\rm MR}}
\newcommand\SigmaNSW{\sigma^2_{\rm NSW}}
\newcommand\SigmaMRE{\sigma^2_{\rm MRND}}
\newcommand\SigmaNSWE{\sigma^2_{\rm NSWND}}

\newcommand\TnMR{T^N_{\rm MR}}
\newcommand\TnNSW{T^N_{\rm NSW}}
\newcommand\TnMRND{T^N_{\rm MRND}}
\newcommand\TnNSWND{T^N_{\rm NSWND}}

\newcommand\RnMR{R^N_{\rm MR}}
\newcommand\RnNSW{R^N_{\rm NSW}}

\newcommand\bzero{\boldsymbol{0}}
\newcommand\bone{\boldsymbol{1}}

\newcommand\bc{\boldsymbol{c}}
\newcommand\bl{\boldsymbol{l}}
\newcommand\bn{\boldsymbol{n}}
\newcommand\bp{\boldsymbol{p}}
\newcommand\bptwo{\boldsymbol{p}_{[2]}}
\newcommand\bv{\boldsymbol{v}}
\newcommand\bw{\boldsymbol{w}}
\newcommand\bx{\boldsymbol{x}}
\newcommand\by{\boldsymbol{y}}

\newcommand\bwt{\tilde{\boldsymbol{w}}}
\newcommand\bxt{\tilde{\boldsymbol{x}}}
\newcommand\byt{\tilde{\boldsymbol{y}}}
\newcommand\zEt{\tilde{z}_E}
\newcommand\xt{\tilde{x}}
\newcommand\yt{\tilde{y}}
\newcommand\zt{\tilde{z}}
\newcommand\btz{\tilde{b}(z)}

\newcommand\Gt{\tilde{G}}
\newcommand\Bd{B_{\delta}}

\newcommand\fd{f_D}
\newcommand\fde{f_{D_{\epsilon}}}
\newcommand\mud{\mu_D}

\newcommand\betat{\tilde{\beta}}
\newcommand\taut{\tilde{\tau}}
\newcommand\tautd{\tilde{\tau}_{\delta}}
\newcommand\zd{z_{\delta}}
\newcommand\tautdN{\tilde{\tau}_{\delta}^N}
\newcommand\tautN{\tilde{\tau}^N}
\newcommand\Phit{\tilde{\Phi}}
\newcommand\Sigmat{\tilde{\Sigma}}

\newcommand\psit{\widetilde{\psi}}

\newcommand\SigmatMR{\tilde{\Sigma}_{\rm MR}}
\newcommand\SigmatNSW{\tilde{\Sigma}_{\rm NSW}}

\newcommand\xEt{\tilde{x}_E}
\newcommand\yEt{\tilde{y}_E}
\newcommand\rhoEt{\tilde{\rho}_E}
\newcommand\etaEt{\tilde{\eta}_E}

\newcommand\epsE{\epsilon_E}

\newcommand\bU{\boldsymbol{U}}
\newcommand\bV{\boldsymbol{V}}
\newcommand\bVt{\tilde{\boldsymbol{V}}}

\newcommand\DN{D^{(N)}}

\newcommand\bWN{\boldsymbol{W}^N}
\newcommand\bXN{\boldsymbol{X}^N}
\newcommand\bYN{\boldsymbol{Y}^N}

\newcommand\bWNt{\tilde{\boldsymbol{W}}^N}
\newcommand\bXNt{\tilde{\boldsymbol{X}}^N}
\newcommand\bYNt{\tilde{\boldsymbol{Y}}^N}
\newcommand\ZNEt{\tilde{Z}^N_E}
\newcommand\YNEt{\tilde{Y}^N_E}

\newcommand\besi{\boldsymbol{e}^{\rm S}_i}
\newcommand\beii{\boldsymbol{e}^{\rm I}_i}
\newcommand\ber{\boldsymbol{e}^{\rm R}}
\newcommand\bes{\boldsymbol{e}^{\rm S}}
\newcommand\bei{\boldsymbol{e}^{\rm I}}

\newcommand\convD{\stackrel{{\rm D}}{\longrightarrow}}
\newcommand\convas{\stackrel{{\rm a.s.}}{\longrightarrow}}
\newcommand\convp{\stackrel{{\rm p}}{\longrightarrow}}

\newtheorem{thm}{Theorem}
\newtheorem{conj}{Conjecture}

\newtheorem{prop}{Proposition}
\newtheorem{lem}{Lemma}
\newtheorem{remark}{Remark}

\numberwithin{equation}{section}
\numberwithin{thm}{section}
\numberwithin{conj}{section}
\numberwithin{cor}{section}
\numberwithin{prop}{section}
\numberwithin{lem}{section}
\numberwithin{remark}{section}

\newcommand{\figwidth}{10cm}
\newcommand{\hfigwidth}{5.8cm}
\newcommand{\tfigwidth}{4cm}

\begin{document}
\title{A stochastic SIR network epidemic model with preventive dropping of edges\footnote{Accepted for publication in Journal of Mathematical Biology}}
\author{Frank Ball$^1$,  Tom Britton$^2$, KaYin Leung$^2$, David Sirl$^1$}
\date{$^1$ University of Nottingham, UK; \ $^2$ Stockholm University, Sweden \\[1mm] 21st January 2019}

\maketitle

\begin{abstract}
A Markovian SIR (Susceptible $\to$ Infectious $\to$ Recovered) model is considered for the spread of an epidemic on a configuration model network, in which susceptible individuals may take preventive measures by dropping edges to infectious neighbours.  An
effective degree formulation of the model is used in conjunction with the theory of density dependent population processes to obtain a law of large numbers and a functional central limit theorem for the epidemic
as the population size $N \to \infty$, assuming that the degrees of individuals are bounded.
%when a strictly positive fraction of the population is initially infectious.
A central limit theorem is conjectured for the final size of the epidemic.
%, and for the final size of a major outbreak 
%for epidemics started by a finite number of infectives as $N \to \infty$.  
The results are obtained for both the
Molloy--Reed (in which the degrees of individuals are deterministic) and Newman--Strogatz--Watts (in which
the degrees of individuals are independent and identically distributed) versions of the configuration model.
%The latter requires a new extension of the functional central limit theorem for density dependent population processes
%to allow for random initial conditions.
The two versions yield the same limiting deterministic model but the asymptotic variances in the central limit theorems
are greater in the Newman--Strogatz--Watts version.  The basic reproduction number $R_0$ and the process of susceptible individuals in the limiting deterministic model,
for the model with dropping of edges, are the same as for a corresponding SIR model without dropping of edges  but an increased recovery rate,
though, when $R_0>1$, the probability of a major outbreak is greater in the model with dropping of edges.
The results are specialised to the model without dropping of edges to yield  conjectured central limit theorems
for the final size of  Markovian SIR epidemics on configuration-model networks, and for the giant components of those networks.  The theory is illustrated by numerical studies, which demonstrate that the asymptotic 
approximations are good, even for moderate $N$.
%and explore the impact of dropping of edges on epidemic properties.

\end{abstract}

\section{Introduction}
\label{sec:intro}

In understanding the transmission dynamics in a population, one of the most important modelling components is the contact process. In this work we consider a form of self-initiated social distancing in response to an epidemic while at the same time taking into account the underlying contact network structure of the population. The resulting network is sometimes referred to as an adaptive network, e.g.~Gross et al.~\cite{Gross2006}, Shaw and Schwarz~\cite{Shaw2008}, Zanette and Rissau-Gusm\'an~\cite{Zanette2008} and Tunc and Shaw~\cite{Tunc2014}. Behavioural dynamics in infectious disease models can come in many different forms. Much of the literature that combines behavioural changes with network models uses agent-based simulations, as in the works cited above, although analytical advances have also been made (e.g.\ Britton et al.~\cite{Britton2016} and Jacobsen et al.~\cite{Jacobsen2016}).  Our work takes the model introduced in Britton et al.~\cite{Britton2016} as its starting point.  Britton et al.~\cite{Britton2016} consider a broader class of models but restrict the analysis to the initial phase of the epidemic. In the current paper we analyse the time evolution and the final size of the epidemic. We model an SIR (Susceptible $\to$ Infectious $\to$ Recovered) infection on a configuration network that is static in the absence of infection. A susceptible individual breaks off its connection to an infectious neighbour upon learning of that neighbour's infectious status. This occurs at a constant rate, independently per neighbour. One can think of this mechanism as being governed by infectious individuals informing their neighbours. Whereas infectious and recovered neighbours do not take any action upon being informed, susceptible neighbours want to avoid becoming infected and therefore cease contact with the infectious individual. We use the term `preventive dropping of edges' to indicate this type of behaviour. Details of the model formulation are presented in Section~\ref{sec:model}.

To some extent, from the point of view of a susceptible neighbour of an infectious individual, it does not matter whether the infectious individual recovers or informs and dissolves the connection. Either way, it means that the susceptible neighbour can no longer acquire infection from this individual. In Section~\ref{sec:meantemp} we see that this is true when dealing with the asymptotic mean (deterministic) process, in that the number of susceptibles in the deterministic process for the model with dropping of edges coincides with that for the model without dropping of edges but with an increased recovery rate. In Section~\ref{sec:relatedmodel} we also see that this is not true for the stochastic process, in particular, the probability of a major outbreak differs (Theorem~\ref{prop:pmajor}). Indeed, we cannot expect the two stochastic processes to coincide since informing neighbours happens independently of one another, while recovery affects all neighbours simultaneously.

In Section~\ref{sec:ED} we analyse the preventive dropping model throughout the epidemic outbreak, by using a so-called effective degree construction (cf.\ Ball and Neal~\cite{Ball:2008}). Using such a construction, conditional on a major outbreak, by using techniques from Ethier and Kurtz~\cite{Ethier:1986}, we show under the assumption of bounded degrees that, as the population size $N$ tends to infinity, the fractions of the population that are susceptible, infective and recovered satisfy a law of large numbers (LLN) over any finite time interval (more specifically that they converge almost surely to a limiting deterministic process), together with an associated functional central limit theorem (CLT) which describes fluctuations of the stochastic epidemic process about the limiting deterministic epidemic.

The population consists of $N$ individuals that make up a network, which is formed using the configuaration model. The configuration model was  introduced by Bollob\'as~\cite{Bollobas1980}, see Bollob\'as~\cite{Bollobas2001} for further references, and comes in two versions: either (i) the degrees of individuals are given by a deterministic sequences of degrees,  the Molloy--Reed (MR) random graph~\cite{Molloy:1995}, or (ii) the degrees of individuals are assumed to be independent and identically distributed, the  Newman--Strogatz--Watts (NSW) random graph~\cite{Newman:2001}. We treat both the MR and the NSW versions. If the limiting properties of the degree sequence of the MR construction agrees with that of the degree distribution of the NSW, the two versions give the same LLN, as we show in Theorem~\ref{thm:as}. However, the two versions differ regarding the variance in the CLTs, since (for finite $N$) there is greater variability in the degrees of the individuals in the NSW model than in the MR model. The functional CLT for the epidemic on an MR random graph is given in Theorem~\ref{thm:MRclt}. By making a random time transformation, in Section~\ref{sec:final}, we conjecture a CLT for the final outcome of the epidemic on an MR random graph; see Conjecture~\ref{conj:MRcltfin}.  Corresponding results for the epidemic on an NSW random graph are discussed in Section~\ref{sec:iiddegrees}; see Theorem~\ref{thm:NSWtemporalCLT} and Conjecture~\ref{conj:nswCLT}. To prove the latter results we require a version of the functional CLT in Ethier and Kurtz~\cite{Ethier:1986} which allows for asymptotically random initial conditions; see Theorem~\ref{KurtzFCLTrandinit}. 

The asymptotic variance-covariance matrix in the CLT in Proposition~\ref{prop:MRcltfin} is far from explicit.  In order to obtain a nearly-explicit expression for the limiting variance of the final size, it is necessary to solve (partially) a time-transformed limiting deterministic process, which is more amenable to analysis than the
corresponding deterministic process in real time.  This is done in Section~\ref{sec:dettemp} and linked to the solution of the real-time process in Section~\ref{sec:detrealtime}.  These results are used in  Sections~\ref{sec:varfinal} and~\ref{sec:iiddegrees} to obtain almost fully explicit expressions for the asymptotic variance of the final size of epidemics on MR and NSW random graphs, respectively, see Proposition~\ref{prop:mrVar} and Conjecture~\ref{conj:nswCLT}.  In Section~\ref{sec:otherapproaches}, we connect our analysis of the deterministic effective degree model to results derived using other deterministic approaches (cf.\ Volz~\cite{Volz:2008}, Leung and Diekmann~\cite{Leung:2016} for related models), leading to a simple proof that
the process of susceptible individuals in the limiting deterministic model for the epidemic with preventive dropping of edges is identical to that in the corresponding deterministic model without dropping of edges but with an increased recovery rate (see Remark~\ref{rmk:binding final1}). 
%The final size of the deterministic model with preventive dropping of edges is studied briefly in Section~\ref{sec:finalsize}.

Note that in the absence of behaviour change, we are in the setting of a Markov SIR epidemic on a configuration model network, which we consider in Section~\ref{sec:nodropping}. This model has been analysed in several papers, e.g.~Newman~\cite{Newman:2002}, Kenah and Robins~\cite{Kenah:2007}, Lindquist et al.~\cite{Lindquist2011} and Miller~\cite{Miller:2011}. Our results further improve understanding of this well-studied model, particularly in terms of the asymptotic variance of the final size in Conjecture~\ref{conj:nodroppingCLT}. Moreover, our work yields conjectured CLTs for the size of the giant component in MR and NSW configuration model random graphs; see Conjecture~\ref{prop:giantCLT}.

In Section~\ref{sec:numerical}, we illustrate our results with some numerical studies. In particular, we demonstrate that the asymptotic results generally give a good approximation for moderate population sizes, investigate the impact of the dropping of edges on properties of epidemics and do some comparison of the behaviour of the epidemic on MR and NSW type random graphs.
Some brief concluding comments are given in Section~\ref{sec:conc}.

Finally, we would like to make a note on the structure of the paper. Clearly, this paper does not readily lend itself to a quick superficial read, owing to its length and some of the technicalities and details involved in obtaining our results. However, we have tried to help the reader by formulating our main results in terms of propositions, theorems and well-motivated conjectures.  The more technical aspects can be found in the appendices for the interested reader, which consequently constitute a significant part of the paper.

\section{The stochastic SIR network epidemic model with preventive dropping}
\label{sec:model}
In this section we define the \emph{stochastic SIR network epidemic model with preventive dropping}. This model is a special case of the network epidemic model with preventive rewiring defined in Britton et al.~\cite{Britton2016}, namely where there is no latency period and where the fraction of dropped edges that are replaced by new edges is set to zero.

The population consists of $N$ individuals, labelled $1,2,\ldots,N$, that make up a network. The network is formed using the configuration model, which, as described in Section~\ref{sec:intro}, comes in two versions, namely MR and NSW random graphs.
Let $D$ be a random variable which describes the degree of a typical individual and let $p_k=\Pp(D=k), k=0,1,\ldots$ Let $\mud$ and $\sigma^2_D$ denote the mean and variance of $D$, respectively, both of which are assumed to be finite.
\begin{itemize}
\item[(i)] In the MR model, the degrees are prescribed.  More specifically, for $N=1,2,\ldots$,  let $d_1^N,d_2^N,\ldots, d_N^N$ denote the degrees of the individuals when the population size is $N$.  Note that these are deterministic.
Let $p_k^N=N^{-1}\sum_{i=1}^N \delta_{k,d_i^N}, k=0,1,\ldots$ be the empirical distribution of $d_1^N,d_2^N,\ldots, d_N^N$, where the Kronecker delta $\delta_{k,j}$ is $1$ if $k=j$ and $0$ otherwise. It is assumed that $\lim_{N \to \infty} p_k^N =p_k, k=0,1,\ldots$.

\item[(ii)] In the NSW model, the degrees $D_1,D_2,\ldots,D_N$ of the $N$ individuals are independent and identically distributed copies of $D$.  A sequence of networks, indexed by $N$, may be constructed from a sequence $D_1,D_2,\ldots$ of independent and identically distributed copies of $D$ by using the first $N$ random variables for the network on $N$ individuals.
\end{itemize}

In both models the network is formed by attaching a number of stubs (i.e.~half-edges) to each individual, according to its degree (so, for example, in the NSW model, $D_i$ stubs are attached to individual $i$, for $i=1,2,\ldots,N$), and then pairing up these stubs uniformly at random to form the network. If $D_1+D_2+\ldots+D_N$ is odd, there is a left-over stub, which is ignored.  The network may have some `defects', specifically self-loops and multiple edges between pairs of individuals, but provided $\sigma^2_D<\infty$, which we assume, such defects become sparse in the network as $N \to \infty$; see  Durrett~\cite{Durrett:2007}, Theorem 3.1.2.

A Markovian SIR epidemic is defined on the network of $N$ individuals as follows. Each individual is at any point in time either susceptible, infective or recovered (and immune to further infection). An infective individual infects each of its susceptible neighbours at the points of independent Poisson processes, each having rate $\beta$. An infectious individual recovers and becomes immune at rate $\gamma$ (implying that the duration of the infectious period follows an exponential distribution having mean $1/\gamma$). Finally, susceptible individuals that have infectious neighbours drop such connections, independently, at rate $\omega$ (an equivalent description to be used later is that the infective `warns' its neighbours \emph{independently} at rate $\omega$, and warned susceptible individuals drop the corresponding edge). All infectious periods, infecting processes and edge-dropping processes are mutually independent. The epidemic is initiated at time $t=0$ by one or more individuals being infectious and all other individuals being susceptible. More precise initial conditions are given when they are required.  The epidemic continues until there is no infectious individual. Then the epidemic stops and the result is that some of the individuals have been infected (and later recovered) and the rest of the population remains susceptible and hence have not been infected during the outbreak.

The parameters of the model are the degree distribution $\{p_k\}$, including its mean $\mu_D$ and variance $\sigma_D^2$, the infection rate $\beta$, the recovery rate $\gamma$ and the dropping rate $\omega$.  
%Note that degree-$0$ individuals have no effect on the epidemic, so to ease the presentation of results we assume that $p_0=0$.  Extension of our results to the case when $p_0>0$ is straightforward.

It was shown in Britton et al.~\cite{Britton2016} that the basic reproduction number for the model is given by
\begin{equation}
\label{R_0}
R_0=\frac{\beta}{\beta+\gamma+\omega}\left(\mu_D+\frac{\sigma_D^2}{\mu_D}-1\right),
\end{equation}
see also Section~\ref{sec:relatedmodel}. Note that the first factor in~\eqref{R_0} is the probability that an infective infects a given susceptible neighbour before either the infective recovers or the neighbour drops its edge to that infective. The second factor is the expected number of susceptible neighbours for infected individuals during the early stages of an outbreak initiated by few infectives. Owing to the way the network is constructed, the degree $\tilde{D}$ of a typical neighbour of a typical individual has the size-biased distribution $\Pp\left(\tilde{D}=k\right)=\mud^{-1}k p_k$,  $k=1,2,\ldots$, and hence mean $\mud^{-1}\E[D^2]=\mud^{-1}\left(\mud^2+\sigma_D^2\right)$. In the early stages of an outbreak, a typical infective has all susceptible neighbours except for one, namely its infector.

Note that $R_0$ for the dropping model is the same as for a Markovian SIR epidemic on a configuration model network without dropping of edges but with an increased recovery rate $\gamma+\omega$; see also Remark~\ref{rmk:binding final1} and Section~\ref{sec:relatedmodel}, where we discuss this modified model with increased recovery rate and its relation to the dropping model. Furthermore, from~\eqref{R_0} we find that $R_0$ is a monotonically decreasing function of $\omega$, i.e.\ dropping edges always decreases the epidemic threshold parameter $R_0$; see also Figure~\ref{fig:dropping} in Section~\ref{sec:droppingEffect}. For epidemics initiated by few infectives, this paper is concerned mainly with the case where $R_0>1$, since only then is there a possibility for a major outbreak to take place.

\section{Effective degree formulation}
\label{sec:ED}
In this section we analyse the stochastic SIR network epidemic model with preventive dropping that is described in Section~\ref{sec:model}. We do so by extending the `effective degree' construction of an SIR epidemic on a configuration model network, introduced in Ball and Neal~\cite{Ball:2008}, to incorporate dropping of edges. This allows us to prove a LLN and a functional CLT for the epidemic process (Theorems~\ref{thm:as} and~\ref{thm:MRclt}). Our proofs rely on the results of Ethier and Kurtz~\cite{Ethier:1986} (see also Kurtz~\cite{Kurtz:1970,Kurtz:1971}), and we adopt mostly the notation used in their work for ease of reference.

In the effective degree formulation the network is constructed as the epidemic progresses.  The process starts with some individuals infective and the remaining individuals susceptible, but with none of the stubs paired up.  For $i=1,2,\ldots,N$, the effective degree of individual $i$ is initially $d_i^N$ in the MR graph and $D_i$ in the NSW graph.  Infected individuals behave in the following fashion.  An infective, $i$ say, transmits infection down its unpaired stubs at points of independent Poisson processes, each having rate $\beta$.  When $i$ transmits infection down a stub, that stub is paired with a stub (attached to individual $j$, say) chosen uniformly at random from all other unpaired stubs to form an edge.  If $i \ne j$ then the effective degrees of both $i$ and $j$ are reduced by $1$, otherwise the effective degree of $i$ is reduced by $2$. If individual $j$ is susceptible then it becomes infective.  If individual $j$ is infective or recovered then nothing happens, apart from the edge being formed. The infective $i$ also independently sends warning messages down its unpaired stubs at points of independent Poisson processes, each having rate $\omega$. When $i$ sends a warning message down a stub, that stub is paired with a stub (attached to individual $j$, say) chosen uniformly at random from all other unpaired stubs.  If individual $j$ is susceptible then the stub from individual $i$ and the stub from individual $j$ are deleted, corresponding to dropping of an edge in the original model.
If individual $j$ is infective or recovered then the two stubs are paired to form an edge.  In all three cases, the effective degrees of $i$ and $j$ are reduced as above. Individual $i$ recovers independently at rate $\gamma$, keeping all, if any, of its
unpaired stubs. Note that in the formulation in Ball and Neal~\cite{Ball:2008}, when an infective recovers, its unpaired stubs, if any, are paired immediately but that is not necessary and indeed complicates analysis of the model.

Note also that we now use the equivalent formulation of the process for dropping edges of Section~\ref {sec:model}, where dropping is driven by infectives rather than by susceptibles, although it is clear that the two formulations are probabilistically equivalent.  The change is required for the effective degree formulation to model dropping of edges correctly.

Before proceeding we introduce some notation. For $i=0,1,\ldots$ and $t \ge 0$, let $X_i^N(t)$ and $Y_i^N(t)$ be respectively the numbers of susceptibles and infectives having effective degree $i$ at time $t$.  We refer to such individuals as type-$i$ susceptibles and type-$i$ infectives.  For $t \ge 0$, let $Z_E^N(t)$ be the number of unpaired stubs attached to recovered individuals at time $t$.  (Note that it is not necessary to keep track of the effective degrees of recovered individuals since only the total number of unpaired stubs attached to recovered individuals, and not the effective degrees of the individuals involved, is required in the above effective degree formulation.) 
Let $\bXN(t)=(X_0^N(t), X_1^N(t), \ldots)$, $\bYN(t)=(Y_0^N(t), Y_1^N(t), \ldots)$ and $\bWN(t)=(\bXN(t), \bYN(t), Z_E^N(t))$. (Unless stated to the contrary, vectors are row vectors in this paper.) Let $H=\mathbb{Z}_+^{\infty} \times \mathbb{Z}_+^{\infty}\times \mathbb{Z}_+$ denote the state space of $\{\bWN(t)\}=\{\bWN(t): t \ge 0\}$. Define unit vectors $\besi, \beii$ $(i=0,1,\ldots)$ and $\ber$ on $H$, where, for example, $\besi$ has a one in the $i$th `susceptible component' and zeros elsewhere, and $\ber$ has a one in the `recovered component' and zeros elsewhere. Let $\bn=(n_0^X,n_1^X,\ldots, n_0^Y,n_1^Y, \ldots, n_E^Z)$ denote a typical element of $H$, and let $n_E^X=\sum_{i=1}^{\infty}i n_i^X$ and $n_E^Y=\sum_{i=1}^{\infty} i n_i^Y$.  Thus $n_E^X, n_E^Y$ and $n_E^Z$ are the total number of stubs attached to susceptibles, infectives and recovered individuals, respectively, when $\bWN(t)=\bn$.

The process $\{\bWN(t)\}$ is a continuous-time Markov chain with the following transition intensities, where an intensity is zero if $n_E^X+n_E^Y+n_E^Z=1$, since then there is only one stub remaining.

For $i,j=1,2,\ldots$,
\begin{enumerate}
\item[(i)]
type-$i$ infective infects a type-$j$ susceptible
\[
q^N(\bn, \bn-\bei_i+\bei_{i-1}-\bes_j+\bei_{j-1})=\beta i n_i^Y \frac{j n_j^X}{n_E^X+n_E^Y+n_E^Z-1};
\]
\item[(ii)]
type-$i$ infective `infects' a type-$j$ infective, so an edge is formed
\[
q^N(\bn, \bn-\bei_i+\bei_{i-1}-\bei_j+\bei_{j-1})=\beta i n_i^Y \frac{j n_j^Y}{n_E^X+n_E^Y+n_E^Z-1};
\]
\item[(iii)]
type-$i$ infective warns a type-$j$ susceptible, so an edge is dropped
\[
q^N(\bn, \bn-\bei_i+\bei_{i-1}-\bes_j+\bes_{j-1})=\omega i n_i^Y \frac{j n_j^X}{n_E^X+n_E^Y+n_E^Z-1};
\]
\item[(iv)]
type-$i$ infective `warns' a type-$j$ infective, so an edge is formed
\[
q^N(\bn, \bn-\bei_i+\bei_{i-1}-\bei_j+\bei_{j-1})=\omega i n_i^Y \frac{j n_j^Y}{n_E^X+n_E^Y+n_E^Z-1}.
\]
\end{enumerate}
For $i=1,2,\ldots$,
\begin{enumerate}
\item[(v)]
type-$i$ infective `infects' a recovered individual, so an edge is formed
\[
q^N(\bn, \bn-\bei_i+\bei_{i-1}-\ber)=\beta i n_i^Y \frac{n_E^Z}{n_E^X+n_E^Y+n_E^Z-1};
\]
\item[(vi)]
type-$i$ infective `warns' a recovered individual, so an edge is formed
\[
q^N(\bn, \bn-\bei_i+\bei_{i-1}-\ber)=\omega i n_i^Y \frac{n_E^Z}{n_E^X+n_E^Y+n_E^Z-1}.
\]
\end{enumerate}
For $i=0,1,\ldots$,
\begin{enumerate}
\item[(vii)] type-$i$ infective recovers
\[
q^N(\bn, \bn-\bei_i+i\ber)=\gamma n_i^Y.
\]
\end{enumerate}

\begin{remark}[Comments on the intensities]
\label{rk:intensities}
Note that although the above intensities are all independent of $N$, we index them by $N$ since that is required so that $\{\bWN(t)\}$ is a density dependent population process, see~\eqref{DDPPcond} and~\eqref{intensityfun} below.
Note also that the intensities in (ii) and (iv) above need to be modified slightly if $i=j$ to include the possibility
that an infective `infects' or `warns' itself.  For example, the intensity for a type-$i$ infective `infecting' itself
is given by $q^N(\bn, \bn-\bei_i+\bei_{i-2})=\beta i (i-1)n_i^Y/(n_E^X+n_E^Y+n_E^Z-1)$, so this should be subtracted from
the intensity in (ii) when $j=i$ and included instead in a new transition, (ii') say. It is easily verified that 
that $q^N(\bn, \bn-\bei_i+\bei_{i-2})=O(1)$ as $N \to \infty$, so the modifications may be absorbed into the
$O(1/N)$ term in~\eqref{DDPPcond} below and ignoring such transitions does not affect the LLNs and CLTs in the paper.
\end{remark}

We now introduce notation for the jumps of $\{\bWN(t)\}$.  Note that the transitions
in (ii) and (iv) above are identical, as are the transitions in (v) and (vi), so there are five types of
jumps. For $i,j=1,2,\ldots$, let
\begin{eqnarray}
\bl_{ij}^{(1)}&=&-\bei_i+\bei_{i-1}-\bes_j+\bei_{j-1},\label{lij1}\\
\bl_{ij}^{(2)}&=&-\bei_i+\bei_{i-1}-\bei_j+\bei_{j-1},\label{lij2}\\
\bl_{ij}^{(3)}&=&-\bei_i+\bei_{i-1}-\bes_j+\bes_{j-1}\label{lij3},
\end{eqnarray}
for $i=1,2,\ldots$, let
\begin{eqnarray}
\bl_{i}^{(4)}&=&-\bei_i+\bei_{i-1}-\ber,\label{li1}
\end{eqnarray}
and, for $i=0,1,\ldots$, let
\begin{eqnarray}
\bl_{i}^{(5)}&=&-\bei_i+i\ber. \label{li2}
\end{eqnarray}
Then, excluding self-infection and self-warning (see Remark~\ref{rk:intensities}), the set of possible jumps of $\{\bWN(t)\}$ from a typical state $\bn\in H$ is $\Delta=\cup_{k=1}^5 \Delta_k$, where
\begin{align*}
\Delta_k=\left\{\bl_{ij}^{(k)}:i,j=1,2,\ldots\right\}&\quad(k=1,2,3), \quad \Delta_4=\left\{\bl_{i}^{(4)}:i=1,2,\ldots\right\}\\
\mbox{and}\quad \Delta_5=\left\{\bl_{i}^{(5)}:i=0,1,\ldots\right\}.
\end{align*}

Let $\bx=(x_0,x_1,\ldots)$ and $\by=(y_0,y_1,\ldots)\in \mathbb{R}_+^{\infty}$, $z_E \in \mathbb{R}_+$ and $\bw=(\bx,\by,z_E)$. Further, let $x_E=\sum_{i=1}^{\infty} i x_i$, $y_E=\sum_{i=1}^{\infty} i y_i$ and $\eta_E=x_E+y_E+z_E$. For $\epsilon>0$, let $H_{\epsilon}^N=\{\bn \in H:\sum_{i=1}^{\infty}i n_i^X \ge \epsilon N\}$. For any $\epsilon>0$, the intensities of the jumps of $\{\bWN(t)\}$ admit the form
\begin{equation}
\label{DDPPcond}
q^N(\bn,\bn+\bl)=N\left[\beta_{\bl}(N^{-1}\bn)+O(1/N)\right]\qquad(\bn \in H_{\epsilon}^N, \bl \in \Delta),
\end{equation}
with the functions $\beta_{\bl}$ $(\bl \in \Delta)$ given by
\begin{equation}
\label{intensityfun}
\beta_{\bl}(\bw)=\beta_{\bl}(\bx,\by,z_E) = \begin{cases}
	      \beta_{ij}^{(1)}(\bx,\by,z_E)=\frac{\beta i y_i j x_j}{\eta_E}& \text{ for } \bl=\bl_{ij}^{(1)} \in \Delta_1, \\
	      \beta_{ij}^{(2)}(\bx,\by,z_E)=\frac{(\beta+\omega)i y_i j y_j}{\eta_E}& \text{ for } \bl=\bl_{ij}^{(2)} \in \Delta_2,\\
	      \beta_{ij}^{(3)}(\bx,\by,z_E)=\frac{\omega i y_i j x_j}{\eta_E}& \text{ for } \bl=\bl_{ij}^{(3)} \in \Delta_3,\\
	      \beta_{i}^{(4)}(\bx,\by,z_E)=\frac{(\beta+\omega)i y_i z_E}{\eta_E}& \text{ for } \bl=\bl_{i}^{(4)} \in \Delta_4,\\
	      \beta_{i}^{(5)}(\bx,\by,z_E)=\gamma y_i & \text{ for } \bl=\bl_{i}^{(5)} \in \Delta_5.
\end{cases}
\end{equation}

\begin{remark}[Applying the theory of Ethier and Kurtz]\label{rk:context}
The theory of density dependent population processes in Ethier and Kurtz~\cite{Ethier:1986}, Chapter 11, is for a class of continuous-time Markov chains whose state space is a subset of $\mathbb{Z}^d$ for some $d \in \mathbb{N}$.  Thus to use this theory we need to assume that there is a maximum degree, i.e.~that $\dmax<\infty$, where
$\dmax=\sup\{k \ge 0 :p_k>0\}$.  Then, for any $\epsilon >0$, provided the sample paths of $\{\bWN(t)\}$ remain within $ H_{\epsilon}^N$, $\{\bWN(t)\}$ is a density dependent population process; see Appendix~\ref{app:kurtz} for details. 
We conjecture that our results continue to hold when the condition $\dmax<\infty$ is relaxed, provided suitable conditions are
imposed on (i) the distribution of $D$ and (ii), for epidemics on MR random graphs, the convergence of the empirical distribution of prescribed degrees.

The key theorems in  Ethier and Kurtz~\cite{Ethier:1986}, Chapter 11, have their origin in Kurtz~\cite{Kurtz:1970,Kurtz:1971}.
However, the proofs in Ethier and Kurtz~\cite{Ethier:1986} are different from those in the earlier papers and the LLN is
stronger in that it concerns almost sure convergence rather than convergence in probability.
In Ethier and Kurtz~\cite{Ethier:1986}, the processes corresponding to $\{\bWN(t)\}$ $(N=1,2,\dots)$ are defined on the same probability space by
using a single set of independent unit-rate Poisson processes indexed by the possible jumps $\bl$. 

A LLN and a functional CLT for density dependent population processes having countable state space are proved in Barbour and Luczak~\cite{Barbour:2012b,Barbour:2012a}. They do not apply immediately to $\{\bWN(t)\}$ as the jumps of $\{Z_E^N(t)\}$ are unbounded, though that can be overcome by replacing $\{Z_E^N(t)\}$ by $\{(Z_0^N(t), Z_1^N(t), \ldots)\}$, where $Z_i^N(t)$ is the number of recovered individuals having effective degree $i$ at time $t$.  We do not consider here sufficient conditions for the theorems in Barbour and Luczak~\cite{Barbour:2012b,Barbour:2012a} to be satisfied in the present setting, since $\dmax< \infty$ is satisfied for real-life epidemics.  We note that  LLNs for the Markov SIR epidemic ($\omega =0$) on an MR random graph with unbounded degree are given in Decreusefond et al.~\cite{Decreusefond:2012} and Janson et al.~\cite{Janson:2014}, and a functional CLT for the Markov SI epidemic ($\omega=\gamma=0$) on an MR random graph with unbounded degree is given in KhudaBukhsh et al.~\cite{KhudaBukhsh:2017}.  It seems likely that similar techniques used in the first two of those papers will apply to the present model.   LLNs for the Markov SIR epidemic ($\omega =0$) on an MR random graph with bounded degree are given in Bohman and Picollelli~\cite{Bohman:2012} and Barbour and Reinert~\cite{Barbour:2013}, the latter for epidemics started by a trace of infection. Indeed our model (assuming bounded degrees) fits into the framework of Barbour and Reinert~\cite{Barbour:2013}, Sec 3.2.
%if in their notation we take $\Phi_i(t) = 1-e^{-\gamma t}$ and $G_{ij}(t) = \frac{\beta}{\beta + \omega} (1-e^{-\beta t})$.

\end{remark}

Following Ethier and Kurtz~\cite{Ethier:1986}, define the drift function $F(\bw)=F(\bx,\by,z_E)$ by
\[
F(\bx,\by,z_E)=\sum_{\bl \in\Delta}\bl \beta_{\bl}(\bx,\by,z_E).
\]
Substituting from~\eqref{intensityfun} yields (see Appendix~\ref{app:drift} for details)
\begin{align}
\label{driftF}
F(\bx,&\by,z_E)=\sum_{i=0}^{\infty}\left[-\beta i x_i +\omega(-ix_i + (i+1) x_{i+1})\right]\frac{y_E}{\eta_E} \bes_i \nonumber \\
&+\sum_{i=0}^{\infty}\left[(\beta +\omega)(-iy_i + (i+1) y_{i+1})\left(1+\frac{y_E}{\eta_E}\right)
+\beta (i+1)x_{i+1}\frac{y_E}{\eta_E}-\gamma y_i \right] \bei_i \nonumber \\
&+\left[\gamma y_E-(\beta+\omega)\frac{y_E z_E}{\eta_E}\right]\ber.
\end{align}

Consider a sequence of epidemics indexed by $N$, each having $Z_E^N(0)=0$.  
Suppose that $N^{-1}Y_i^N(0) \convas \epsilon_i$
and $N^{-1}X_i^N(0) \convas p_i-\epsilon_i$ as $N \to \infty$, where $\epsilon_E=\sum_{i=1}^{\infty}i \epsilon_i>0$ and $\convas$ denotes almost sure convergence.  Note that for epidemics on NSW random graphs $\bXN(0)$ is random and, depending on how the initial infectives are chosen, $\bYN(0)$ may also be random.  The above almost sure convergence is reasonable for such epidemics since in an NSW random graph, the fraction of vertices of any given degree satisfies the strong law of large numbers.  For epidemics on
MR random graphs it is often more natural for  $(\bXN(0), \bYN(0))$ to be non-random, in which case $N^{-1}Y_i^N(0) \to \epsilon_i$ and $N^{-1}X_i^N(0) \to p_i-\epsilon_i$ as $N \to \infty$.
Let $\bx(0)=(p_0-\epsilon_0, p_1-\epsilon_1,\ldots)$ and $\by(0)=(\epsilon_0,\epsilon_1,\ldots)$.  The following result holds for  epidemics on both MR and NSW random graphs.

\begin{thm}[LLN for epidemic on network with dropping]
\label{thm:as}
\mbox{}\\ Suppose that $\dmax<\infty$ and $\epsilon_E>0$.  Then, for any $T>0$,
\[
\lim_{N \to \infty} \sup_{0 \le t \le T} |N^{-1}\bWN(t)-\bw(t)|=0\qquad\mbox{almost surely},
\]
where $\bw(t)=(\bx(t),\by(t),z_E(t))$ is given by the solution of the following system of ordinary differential equations (ODEs) with initial condition $\bw(0)=(\bx(0), \by(0),0)$:
\begin{align}
\dfrac{dx_i}{dt}&=-\beta \rho_E(t) i x_i+\omega \rho_E(t)(-i x_i +(i+1)x_{i+1})\quad (i=0,1,\ldots), \label{diff1} \\
\dfrac{dy_i}{dt}&=(\beta+\omega)((i+1)y_{i+1}-i y_i)-\gamma y_i+(\beta +\omega)\rho_E(t)[(i+1)y_{i+1}-i y_i]\nonumber\\
&\qquad+\beta\rho_E(t)(i+1)x_{i+1}\quad (i=0,1,\ldots),\label{diff2}\\
\dfrac{dz_E}{dt}&= \gamma y_E(t)-(\beta+\omega)\rho_E(t) z_E,\label{diff3}
\end{align}
where 
\begin{equation}\label{eq:rhoE}
\rho_E(t)=y_E(t)/\eta_E(t)
\end{equation}
and $\eta_E(t)=x_E(t)+y_E(t)+z_E(t)$.
\end{thm}
\begin{proof}
See Appendix~\ref{app:kurtz}. 
$\Box$	
\end{proof}

\begin{remark}[Solving the ODEs~\eqref{diff1}-\eqref{diff3}] 
\label{rmk:detedsol}
The solution of the system of ODEs~\eqref{diff1}-\eqref{diff3} is considered in Section~\ref{sec:meantemp}.  Note that 
under the conditions of Theorem~\ref{thm:as} the system of ODEs~\eqref{diff1}-\eqref{diff3} is finite, so existence and uniqueness
of a solution follow from standard results.  We do not consider existence and uniqueness
of solutions to ODEs~\eqref{diff1}-\eqref{diff3} when the degrees are unbounded  but acknowledge that further justification and some regularity conditions will be required.
A similar comment applies to the time-transformed  
system of ODEs~\eqref{difft1}-\eqref{difft3} in Section~\ref{sec:final}.

\end{remark}

For the epidemic on an MR random graph, a functional CLT for the fluctuations of $\{\bWN(t)\}$ about its deterministic limit $\{\bw(t)\}$ is also available using Ethier and Kurtz~\cite{Ethier:1986}, Theorem 11.2.3, as we formulate in Theorem~\ref{thm:MRclt}. See Section~\ref{sec:iiddegrees} for discussion of a corresponding CLT for the epidemic on an NSW random graph.

Write $\bw$ as $(w_1,w_2,\dots)$ and let $\partial F(\bw)=[\partial_j F_i(\bw)]$ denote the
matrix of first partial derivatives of $F(\bw)$.  For $0 \le u \le t <\infty$,
let $\Phi(t,u)$ be the solution of the matrix ODE
\begin{equation}
\dfrac{\partial}{\partial t}\Phi(t,u)=\partial F(\bw(t))\Phi(t,u),
\quad \Phi(u,u)=I,
\label{Phi}
\end{equation}
where $I$ denotes the identity matrix of appropriate dimension.  Let
\[
G(\bw)=\sum_{\bl \in \Delta} \bl \bl^{\top} \beta_{\bl}(\bw),
\]
where $\top$ denotes transpose.  Note that $\partial F(\bw(t))$ is the coefficient matrix of the time-inhomogeneous
linear drift of the limiting Gaussian process $\{\bV(t)\}$ in Theorem~\ref{thm:MRclt} below and $\Phi(t,u)$ 
enables a representation of $\{\bV(t)\}$ in terms of an It\^o integral with respect to a time-inhomogeneous
Brownian motion; see~\eqref{ItoV} in Section~\ref{sec:iiddegrees}.

\begin{thm}[Functional CLT for epidemic on MR graph with dropping]\label{thm:MRclt}
\mbox{}\\Suppose that $\dmax<\infty, \epsilon_E>0$ and, for $i=0,1,\dots,\dmax$,
\begin{equation}\label{initialcond}
\lim_{N \to \infty} \sqrt{N}\left(N^{-1}Y_i^N(0) - \epsilon_i\right)=v_i^{Y} \quad\mbox{and}\quad
\lim_{N \to \infty} \sqrt{N}\left(N^{-1}X_i^N(0) - p_i-\epsilon_i\right)=v_i^{X},
\end{equation}
where $\bv=(v_0^{X}, v_1^{X}, \dots, v_{\dmax}^{X},v_0^{Y}, v_1^{Y}, \dots, v_{\dmax}^{Y},0)$ is constant.
Then
\begin{equation}\label{FCLT}
\sqrt{N}\left(\{N^{-1}\bWN(t)\}-\{\bw(t)\}\right) \Rightarrow \{\bV(t)\} \quad\mbox{as } N \to \infty,
\end{equation}
where $\Rightarrow$ denotes weak convergence and $\{\bV(t)\}=\{\bV(t):t \ge 0\}$ is a zero-mean
Gaussian process with $\bV(0)=\bv$ and covariance function given by
\begin{equation*}
\cov\left(\bV(t_1), \bV(t_2)\right)=\int_0^{\min(t_1,t_2)}\Phi(t_1,u)
G(\bw(u))\Phi(t_2,u) ^\top \,{\rm d}u \qquad(t_1,t_2 \ge 0).
\end{equation*}
\end{thm}

\begin{proof}
See Appendix~\ref{app:kurtz}, where a complete definition of $\Rightarrow$ is given.
$\Box$	
\end{proof}

\begin{remark}[Computing the asymptotic variance]\label{rmk:cltvart}
Theorem~\ref{thm:MRclt} yields immediately that
\begin{equation}\label{Sigma}
\Sigma(t)=\var\left(\bV(t)\right)=\int_0^t \Phi(t,u)
G(\bw(u))\Phi(t,u) ^\top \,{\rm d}u.
\end{equation}
It follows from~\eqref{Phi} and~\eqref{Sigma} that $\Sigma(t)$ satisfies the ODE
\begin{equation}\label{diffSigma}
\dfrac{d\Sigma}{dt}=G(\bw)+\partial F(\bw) \Sigma +\Sigma [\partial F(\bw)]^{\top},
\end{equation}
with initial condition $\Sigma(0)=0$.  Thus, provided $\dmax<\infty$, $\Sigma(t)$ can be computed by numerically solving the ODEs~\eqref{diff1}-\eqref{diff3} and~\eqref{diffSigma} simultaneously.
\end{remark}

\section{Final outcome of epidemic on MR random graph}
\label{sec:final}
We conjecture a CLT for the final outcome of the epidemic with preventive dropping on an MR random graph (see Conjecture~\ref{prop:MRcltfin}). In order to do so, we consider a random time-transformation of the real-time process.
%The time transformation is revisited in Section~\ref{sec:detrealtime}.
%We present the result for the final size in Proposition~\ref{prop:MRcltfin}.

For $t \ge 0$, let $X_E^N(t)=\sum_{i=1}^\infty i X_i^N(t)$ and $Y_E^N(t)=\sum_{i=1}^\infty i Y_i^N(t)$ be respectively the number of susceptible and infectious stubs at time $t$. Let $\tau^N=\inf\{t \ge 0: Y_E^N(t)=0\}$, so the final number of susceptibles of different types is given by $\bXN(\tau^N)$.  For $\delta \ge 0$, let $\tau^N_{\delta}= \inf\{t \ge 0: N^{-1}Y_E^N(t)\le \delta\}$, so $\tau^N=\tau^N_0$.  Recall the definition of $\epsilon_E$ following~\eqref{driftF}. For $\delta \in (0, \epsilon_E)$, we derive a CLT for $\bWN(\tau^N_{\delta})$; see Proposition~\ref{prop:MRcltfin}. Assuming that Proposition~\ref{prop:MRcltfin} holds also when $\delta=0$ leads immediately to a CLT (Conjecture~\ref{conj:MRcltfin}) for $X^N(\tau^N)=\sum_{i=0}^\infty X_i^N(\tau^N)$, and hence for the total number of individuals that are ultimately infected by the epidemic, since the latter is given by $N-\sum_{i=0}^\infty X_i^N(\tau^N)$. A key step in deriving these CLTs is to consider the following random time-scale transformation of $\{\bWN(t)\}$; cf.~Ethier and Kurtz~\cite{Ethier:1986}, page 467, and Janson et al.~\cite{Janson:2014}, Section 3, where similar transformations are used to derive a CLT for the final size of the so-called general stochastic epidemic and a LLN for the Markovian SIR epidemic on an MR random graph, respectively. 

For $t \in [0, \tau^N]$, let
\[
A^N(t)=\int_0^t \frac{Y_E^N(u)}{X_E^N(u)+Y_E^N(u)+Z_E^N(u)} \,{\rm d}u,
\]
and let $\tautN=A^N(\tau^N)$.  For $0 \le t \le \tautN$, let $U^N(t)=\inf\{u \ge 0:A^N(u)=t\}$ and
\[
\bWNt(t)=(\bXNt(t), \bYNt(t), \ZNEt(t))=\bWN\left(U^N(t)\right).
\]
Then $\{\bWNt(t)\}=\{\bWNt(t): 0 \le t \le \tautN\}$ is also a density dependent population process, having the same set $\Delta$ of jumps as $\{\bWN(t)\}$, and intensity functions $\tilde{\beta}_{\bl}$ $(\bl \in \Delta)$ given by
\begin{equation}
\label{intensityfun1}
\tilde{\beta}_{\bl}(\bw)=\tilde{\beta}_{\bl}(\bx,\by,z_E) = \begin{cases}
	      \tilde\beta_{ij}^{(1)}(\bx,\by,z_E)=\frac{\beta i y_i j x_j}{y_E}& \text{ for } \bl=\bl_{ij}^{(1)} \in \Delta_1, \\
	      \tilde\beta_{ij}^{(2)}(\bx,\by,z_E)=\frac{(\beta+\omega)i y_i j y_j}{y_E}& \text{ for } \bl=\bl_{ij}^{(2)} \in \Delta_2,\\
	      \tilde\beta_{ij}^{(3)}(\bx,\by,z_E)=\frac{\omega i y_i j x_j}{y_E}& \text{ for } \bl=\bl_{ij}^{(3)} \in \Delta_3,\\
	      \tilde\beta_{i}^{(4)}(\bx,\by,z_E)=\frac{(\beta+\omega)i y_i z_E}{y_E}& \text{ for } \bl=\bl_{i}^{(4)} \in \Delta_4,\\
	      \tilde\beta_{i}^{(5)}(\bx,\by,z_E)=\gamma y_i \frac{\eta_E}{y_E} & \text{ for } \bl=\bl_{i}^{(5)} \in \Delta_5.
\end{cases}
\end{equation}
Note that when $\{\bWN(t)\}$ is in state $\bn=(n_0^X,n_1^X,\ldots, n_0^Y,n_1^Y, \ldots, n_E^Z)$, the clock in $\{\bWNt(t)\}$ runs at rate $(n_E^X+n_E^Y+n_E^Z)/n_E^Y$ times faster than the clock in $\{\bWN(t)\}$, so the intensities in~\eqref{intensityfun1} are obtained by multiplying the corresponding intensities in~\eqref{intensityfun} by $\eta_E/y_E$. The drift function associated with $\{\bWNt(t)\}$ is (cf.~\eqref{driftF})
\begin{align}
\label{driftF1}
\tilde{F}(\bx,&\by,z_E)=\sum_{i=0}^{\infty}\left[-\beta i x_i +\omega(-ix_i + (i+1) x_{i+1})\right] \bes_i \nonumber \\
&+\sum_{i=0}^{\infty}\left[(\beta +\omega)(-iy_i + (i+1) y_{i+1})\left(1+\frac{\eta_E}{y_E}\right)
+\beta (i+1)x_{i+1}-\gamma y_i\frac{\eta_E}{y_E} \right] \bei_i \nonumber \\
&+\left[\gamma \eta_E-(\beta+\omega)z_E\right]\ber.
\end{align}

Let $\{\bwt(t): t \ge 0\}=\{(\bxt(t),\byt(t),\zEt(t)): t \ge 0\}$ be the solution of the following system of ODEs, with initial condition $\bwt(0)=(\bx(0), \by(0),0)$:
\begin{align}
\dfrac{d\xt_i}{dt}&=-\beta  i \xt_i+\omega [-i \xt_i +(i+1)\xt_{i+1}], \label{difft1} \\
\dfrac{d\yt_i}{dt}&=\left\{(\beta+\omega)[(i+1)\yt_{i+1}-i \yt_i]-\gamma \yt_i\right\}\frac{1}{\rhoEt(t)}\nonumber\\
&\qquad+(\beta +\omega)[(i+1)\yt_{i+1}-i \yt_i]+\beta (i+1)\xt_{i+1},\label{difft2}\\
\dfrac{d\zEt}{dt}&= \gamma \etaEt(t)-(\beta+\omega) \zEt,\label{difft3}
\end{align}
where $i=0,1,\ldots$ and $\rhoEt(t)=\yEt(t)/\etaEt(t)$, $\etaEt(t)=\xEt(t)+\yEt(t)+\zEt(t)$ with $\xEt(t)= \sum_{i=1}^{\infty} i \xt_i(t)$ and $\yEt(t)= \sum_{i=1}^{\infty} i \yt_i(t)$. The solution of this system is considered in Section~\ref{sec:dettimechanged}. Let $\taut=\inf\{t \ge 0: \yEt(t)=0\}$.  It is shown in
Appendix~\ref{app:tautdelta} that $\taut< \infty$, i.e.\ the duration of the limiting time-changed deterministic epidemic is finite,
unless $\gamma=\omega=p_1-\epsilon_1=0$.

We consider the same sequence of epidemics as for Proposition~\ref{thm:as} in Section~\ref{sec:ED}. Again, using  Ethier and Kurtz~\cite{Ethier:1986}, Theorem 11.2.1, as $N \to \infty$, $\{N^{-1}\bWNt(t)\}$ converges almost surely over any finite time interval $[0, t_0]$, with $t_0 < \taut$, to $\{\bwt(t)\}=\{\bwt(t):0 \le t \le \taut\}$ (see Appendix~\ref{app:kurtz} for further details of this and of the functional CLT given at~\eqref{FCLTt}). Suppose further that the initial conditions satisfy~\eqref{initialcond} and $\dmax < \infty$. Then it follows using Ethier and Kurtz~\cite{Ethier:1986}, Theorem 11.2.3, that, for any $t_0 \in [0,\taut)$,
\begin{equation}
\label{FCLTt}
\sqrt{N}\left(\{N^{-1}\bWNt(t):0 \le t \le t_0\}-\{\bwt(t):0 \le t \le t_0\}\right) \Rightarrow \{\bVt(t)\} \quad\mbox{as } N \to \infty,
\end{equation}
where $\{\bVt(t):0 \le t \le t_0\}$ is a zero-mean Gaussian process with $\bVt(0)=\bzero$ and variance given by
\begin{equation}
\label{Sigma1}
\SigmatMR(t)=\var\left(\bVt(t)\right)=\int_0^t \Phit(t,s)
\Gt(\bwt(u))\Phit(t,s) ^\top \,{\rm d}s,
\end{equation}
where
\begin{equation}
\label{Gtilde}
\Gt(\bwt(u))=\sum_{\bl \in \Delta} \bl \bl^{\top} \betat_{\bl}(\bwt(u))
\end{equation}
and, for $0 \le s \le t <\infty$, $\Phit(t,s)$ is the solution of the matrix ODE
\begin{equation}
\dfrac{\partial}{\partial t}\Phit(t,u)=\partial \tilde{F}(\bwt(t))\Phit(t,u),
\quad \Phit(u,u)=I.
\label{Phit}
\end{equation}

For $t \ge 0$, let $\YNEt(t)=\sum_{i=1}^{\infty} i \tilde{Y}_i^N(t)$.  Further, for $\delta \ge 0$, let
\begin{equation}
\label{taudelta}
\tautdN= \inf\{t \ge 0: N^{-1}\YNEt(t)\le \delta\} \quad\mbox{ and }\quad \tautd=\inf\{t \ge 0: \yEt(t)=\delta\},
\end{equation}
so both $\tautdN$ and $\tautd$ are decreasing with $\delta$, $\tautN_0=\tautN$ and $\taut_0=\taut$.  
We show in Appendix~\ref{app:tautdelta} that $\tautd<\infty$; it is clearly finite if $\taut < \infty$.
%Thus $\tautd<\infty$, since $\taut < \infty$.  
Let $\varphi(\bwt)=
\varphi(\bxt,\byt,\zEt)=\sum_{i=1}^{\infty} i \yt_i$ $(=\yEt)$, so
\[
\tautdN= \inf\left\{t \ge 0: \varphi\left(N^{-1}\bWNt(t)\right)\le \delta\right\} \quad\mbox{ and }\quad \tautd=\inf\{t \ge 0: \varphi(\bwt(t))=\delta\}.
\]

For fixed $\delta \in (0,y_E(0))$, application of Ethier and Kurtz~\cite{Ethier:1986}, Theorem 11.4.2, yields
\begin{align}
\label{CLT1}
\sqrt{N}\left(N^{-1} \bWNt(\tautdN)-\bwt(\tautd)\right) \convD
\bVt(\tautd)-&\frac{\nabla \varphi(\bwt(\tautd))\cdot \bVt(\tautd)}{\nabla \varphi(\bwt(\tautd)) \cdot \tilde{F}(\bwt(\tautd))} \tilde{F}(\bwt(\tautd))\nonumber\\
&\quad\mbox{as } N \to \infty,
\end{align}
where $\cdot$ denotes inner vector product and $\convD$ denotes convergence in distribution.  This result requires that
\begin{equation}
\label{gradphidotF}
\nabla \varphi(\bwt(\tautd)) \cdot \tilde{F}(\bwt(\tautd))<0,
\end{equation}
which we show in Appendix~\ref{app:tautdelta}. Condition~\eqref{gradphidotF} ensures that $\tautd$ is a proper crossing time. 
Note that
\[
\bVt(\tautd)-\frac{\nabla \varphi(\bwt(\tautd))\cdot \bVt(\tautd)}{\nabla \varphi(\bwt(\tautd)) \cdot \tilde{F}(\bwt(\tautd))} \tilde{F}(\bwt(\tautd))
=\bVt(\tautd)\Bd^{\top},
\]
where
\begin{equation}\label{eq:Bdelta}
\Bd=I-\frac{\tilde{F}(\bwt(\tautd))\bigotimes\nabla \varphi(\bwt(\tautd))}{\nabla \varphi(\bwt(\tautd)) \cdot \tilde{F}(\bwt(\tautd))}
\end{equation}
and $\bigotimes$ denotes outer vector product.

The following proposition follows immediately from~\eqref{CLT1} on noting that $\bWN(\tau^N_{\delta})=\bWNt(\tautdN)$ and $\bw(\tau_{\delta})=
\bwt(\tautd)$, where $\tau_{\delta}=\inf\{t \ge 0: y_E(t)=\delta\}$.

\begin{prop}[CLT for `final' outcome of epidemic on MR graph with dropping]
\label{prop:MRcltfin}
Suppose that $\dmax<\infty, \epsilon_E>0, \delta \in (0,y_E(0))$ and~\eqref{initialcond} is satisfied.  Then
\begin{equation}
\label{CLT2}
\sqrt{N}\left(N^{-1} \bWN(\tau^N_{\delta})-\bw(\tau_{\delta})\right) \convD N\left(\bzero, \Sigma_{{\rm MR},\delta} \right)
\quad\mbox{as } N \to \infty,
\end{equation}
where
\begin{equation*}
\Sigma_{{\rm MR},\delta} =\Bd \SigmatMR(\tautd) \Bd^{\top}
\end{equation*}
and $N\left(\bzero, \Sigma_{{\rm MR},\delta} \right)$ denotes a multivariate normal distribution (of appropriate dimension) with mean vector $\bzero$ and
variance-covariance matrix $\Sigma_{{\rm MR},\delta}$.
\end{prop}

\begin{remark}[Extending Proposition~\ref{prop:MRcltfin} to $\delta=0$]
\label{rmk:MRcltfin1}
We are primarily interested in the case when $\delta=0$.  The difficulty in extending Proposition~\ref{prop:MRcltfin} to
include $\delta=0$ is that to apply Ethier and Kurtz~\cite{Ethier:1986}, Theorem 11.4.2, we need the weak convergence at~\eqref{FCLTt}
to hold for some $t_0>\taut$.  Thus we need to extend the process $\{\bWNt(t)\}$ so that it is defined beyond time $\tautN$.
Now $\yEt(t)<0$ for $t>\taut$ (see~\eqref{yet} in Section~\ref{sec:dettimechanged}), so we need to extend the state space of
$\{\bWNt(t)\}$ so that $\tilde{Y}_i^N(t)$ $(i=0,1,\dots,\dmax)$ can be negative.  However, this cannot be done so that the
conditions of the LLN and CLT theorems in Ethier and Kurtz~\cite{Ethier:1986} are satisified.  In particular, in any neighbourhood of $\{\bw:y_E=0\}$, the intensity functions
$\tilde{\beta}_{\bl}$ $(\bl \in \Delta)$ are not bounded and the drift function $\tilde{F}$ is not Lipschitz continuous.

In work done while this paper was under review, the first author has found a way of overcoming this problem; 
see Ball~\cite{Ball:2018} which is in the setting of an SIR epidemic (without dropping of edges) with an arbitrary but specified
infectious period distribution on configuration model networks.  The theorems proved in Ball~\cite{Ball:2018}  provide further (very strong) support for Conjecture~\ref{conj:MRcltfin} below, which assumes that Proposition~\ref{prop:MRcltfin} extends in the obvious way to include $\delta=0$, and for subsequent conjectures which are contingent on Conjecture~\ref{conj:MRcltfin}.   Note that the final outcome of the epidemic is given by $\bWNt(\tautN)$ and the corresponding determinsitic outcome is $\bwt(\taut)$.
\end{remark}

We use the term final outcome to refer to that of the effective degree formulation, in which the degrees of susceptibles can change owing to dropping of edges.  This is sufficient to determine the final size of an epidemic.  If the final numbers of susceptibles of various original degrees are required, the effective degree
formulation can be extended to keep track of both the original and effective degrees of suceptibles.

\begin{conj}[CLT for final outcome of epidemic on MR graph with dropping]
\label{conj:MRcltfin}
Suppose that $\dmax<\infty, \epsilon_E>0$ and~\eqref{initialcond} is satisfied.  Then
\begin{equation}
\label{CLT2a}
\sqrt{N}\left(N^{-1} \bWNt(\tautN)-\bwt(\taut)\right) \convD N\left(\bzero, \Sigma_{{\rm MR}} \right)
\quad\mbox{as } N \to \infty,
\end{equation}
where
\begin{equation*}
\Sigma_{{\rm MR}} =B \SigmatMR(\taut) B^{\top}
\end{equation*}
with $B$ given by~\eqref{eq:Bdelta} with $\delta=0$.
\end{conj}

\begin{remark}[LLN for final outcome of SIR epidemic with preventive dropping]
\label{rmk:finalwlln}
Conjecture~\ref{prop:MRcltfin} implies that $\bXN(\tau^N) \convp \bx(\infty)$ as $N \to \infty$, where $\convp$ denotes convergence in probability, i.e. the final outcome of the epidemic on an MR random graph obeys a weak LLN.  The same conjecture holds also for the epidemic on an NSW random graph, using the theory in Section~\ref{sec:iiddegrees}. Note that $\bx(\infty)=\bxt(\taut)$ and an expression for $\bxt(\taut)$ is given in equation~\eqref{finalsus} in Section~\ref{sec:finalsize}.
\end{remark}

\begin{remark}[Explicit expression for asymptotic variance of final size]
\label{rmk:MRcltfin2}
Note that $\SigmatMR(t)$, and hence $\Sigma_{{\rm MR},\delta}$, can be computed numerically as described for $\Sigma(t)$ in Remark~\ref{rmk:cltvart}.  However, as detailed in Section~\ref{sec:varfinal} for the case $\delta=0$, it is possible to derive an almost fully explicit expression, as a function of $\tautd$, for the asymptotic variance of the `final' number of susceptibles.  Moreover, the expression is fully explicit when $\omega=0$, i.e.~when there
is no dropping of edges, so the model reduces to a standard Markov SIR epidemic on an MR configuration model network.
\end{remark}

\section{Deterministic temporal behaviour and final size}
\label{sec:meantemp}
In Section~\ref{sec:dettemp} we study the deterministic temporal behaviour of the effective degree model, described by the system of ODEs~\eqref{diff1}-\eqref{diff3} given in Theorem~\ref{thm:as}, by considering first the corresponding time-transformed system~\eqref{difft1}-\eqref{difft3}. The resulting (partial) solution of this system is required to calculate the asymptotic variance of the final size in Sections~\ref{sec:varfinal} and~\ref{sec:iiddegrees}. Furthermore, the results of this section
are used in Appendix~\ref{app:tautdelta} to prove that the conditions $\tautd<\infty$
and~\eqref{gradphidotF}, required for the application of Ethier and Kurtz~\cite{Ethier:1986}, Theorem 11.4.2,
are satisfied.  In Section~\ref{sec:otherapproaches}, we connect the analysis of~\eqref{difft1}-\eqref{difft3} to other approaches taken in literature for the deterministic analysis of epidemics on configuration model
networks. Finally, in Section~\ref{sec:finalsize}, we give a characterization of the deterministic final size of the epidemic and consider the final size of epidemics initiated by a trace of infection in Proposition~\ref{prop:final}.
We do not consider existence and uniqueness of solutions of the determinstic model when $\dmax=\infty$ (see 
Remark~\ref{rmk:detedsol}) but indicate where further justification is required for a proof.
\subsection{Temporal behaviour}
\label{sec:dettemp}
\subsubsection{Time-transformed process}
\label{sec:dettimechanged}
Consider the system of ODEs given by~\eqref{difft1}-\eqref{difft3}, with initial condition
$\bxt(0)=(p_0-\epsilon_0, p_1-\epsilon_1,\ldots)$, $\byt(0)=(\epsilon_0,\epsilon_1,\ldots)$ and $\zEt(0)=0$. In this section we obtain explicit expressions for $\bxt(t)$, $\tilde x_E(t)$, $\tilde y_E(t)$ and other variables pertaining to the fraction of susceptible, infectious, and recovered individuals in the population in the time-transformed process, while in Section~\ref{sec:detrealtime} we connect these to corresponding variables in the real-time process.

Observe that the evolution of $\{\bxt(t)\}$ is decoupled from the rest of the system.  To solve~\eqref{difft1}, let $\{X(t)\}=\{X(t):t \ge 0\}$ denote a transient continuous-time Markov chain describing the evolution of a single susceptible individual, whose stubs are independently dropped at rate $\omega$ and independently infected at rate $\beta$.  For $t \ge 0$, let $X(t)$ be the number of stubs attached to the individual at time $t$, if it is still susceptible, otherwise let $X(t)=-1$.  For $i,j=0,1,\ldots$ and $t \ge 0$, let $p_{ji}(t)=\Pp(X(t)=i|X(0)=j)$.  By deriving the forward equation for $\{X(t)\}$ it is easily seen that, for $i=0,1,\ldots$, $\xt_i(t)=\sum_{j=i}^{\infty} \xt_j(0)p_{ji}(t)$ $(t \ge 0$).

It is straightforward to calculate $p_{ji}(t)$, since stubs disappear (by dropping or infection) independently, the probability that a given initial stub has disappeared by time $t$ is $1-\re^{-(\beta+\omega)t}$ and, given that a stub has disappeared, the probability its disappearance was caused by dropping is $p_{\omega}=\frac{\omega}{\beta+\omega}$. Thus,
\begin{equation}
\label{pij}
p_{ji}(t)=\begin{cases} \binom{j}{i} \re^{-(\beta+\omega)it}\left(1-\re^{-(\beta+\omega)t}\right)^{j-i} p_{\omega}^{j-i} & \text{ for } j \ge i,\\
0& \text{ for } j < i,
\end{cases}
\end{equation}
whence, for $i=0,1,\ldots$,
\begin{eqnarray}
\label{xit}
\xt_i(t)&=&\sum_{j=i}^{\infty}\xt_j(0)p_{ji}(t)\nonumber\\
&=&\sum_{j=i}^{\infty}(p_j -\epsilon_j)\binom{j}{i} \re^{-(\beta+\omega)it}\left(1-\re^{-(\beta+\omega)t}\right)^{j-i} p_{\omega}^{j-i}\nonumber\\
&=&\frac{ \re^{-(\beta+\omega)it}}{i!}\sum_{j=i}^{\infty}(p_j -\epsilon_j) \frac{j!}{(j-i)!}\left[p_{\omega}\left(1-\re^{-(\beta+\omega)t}\right)\right]^{j-i}\nonumber\\
&=&\frac{ \re^{-(\beta+\omega)it}}{i!} \fde^{(i)}\left(p_{\omega}\left(1-\re^{-(\beta+\omega)t}\right)\right),
\end{eqnarray}
where
\begin{equation}
\label{fdes}
\fde(s)=\sum_{k=0}^{\infty} (p_k-\epsilon_k)s^k\qquad (0 \le s \le 1),
\end{equation}
and $\fde^{(i)}$ denotes the $i$th derivative of $\fde$. It then follows that
\begin{eqnarray}
\label{xet}
\xEt(t)&=&\sum_{i=0}^{\infty} \frac{i \re^{-(\beta+\omega)it}}{i!} \fde^{(i)}\left(p_{\omega}\left(1-\re^{-(\beta+\omega)t}\right)\right)\nonumber\\
&=&\re^{-(\beta+\omega)t}\sum_{i=1}^{\infty} \frac{ \re^{-(\beta+\omega)(i-1)t}}{(i-1)!} \fde^{(i)}\left(p_{\omega}\left(1-\re^{-(\beta+\omega)t}\right)\right)\nonumber\\
&=&\re^{-(\beta+\omega)t}\fde'\left(p_{\omega}\left[1-\re^{-(\beta+\omega)t}\right]+\re^{-(\beta+\omega)t}\right)\nonumber\\
&=&\re^{-(\beta+\omega)t}\fde'\left(\psi(t)\right),
\end{eqnarray}
where
\begin{equation}\label{eq:psi}
\psi(t)= p_{\omega}+(1-p_{\omega})\re^{-(\beta+\omega)t}.
\end{equation}
Differentiating~\eqref{xet} yields
\begin{equation}
\label{difft4}
\dfrac{d\xEt}{dt}=-(\beta+\omega)\xEt-\beta \re^{-2(\beta+\omega)t}\fde''\left(\psi(t)\right).
\end{equation}
Note that $\sum_{i=1}^{\infty} i[(i+1)\yt_{i+1}-i \yt_i]=-\yEt$ and, using a similar argument to the derivation of~\eqref{xet},
\begin{equation}
\label{sumiip1x}
\sum_{i=1}^{\infty} i (i+1)\xt_{i+1}(t)=\re^{-2(\beta+\omega)t}\fde''\left(\psi(t)\right).
\end{equation}
Multiplying~\eqref{difft2} by $i$ and summing over $i=1,2,\ldots$ yields
\begin{equation}
\label{difft5}
\dfrac{d\yEt}{dt}=-(\beta+\omega+\gamma)\etaEt-(\beta+\omega)\yEt+\beta \re^{-2(\beta+\omega)t}\fde''\left(\psi(t)\right).
\end{equation}
(This requires justifying and further conditions if $\dmax=\infty$.  A similar comment applies to equations contingent
on~\eqref{difft5}, such as~\eqref{yet}.)
Adding~\eqref{difft4},~\eqref{difft5} and~\eqref{difft3} gives
\begin{equation*}
\dfrac{d\etaEt}{dt}=-2(\beta+\omega)\etaEt,
\end{equation*}
which, together with the initial condition $\etaEt(0)=\mud$, yields
\begin{equation}
\label{etaEt}
\etaEt(t)=\mud \re^{-2(\beta+\omega)t}.
\end{equation}
Substituting~\eqref{etaEt} into~\eqref{difft3} yields
\begin{equation*}
\dfrac{d\zEt}{dt}= \gamma\mud \re^{-2(\beta+\omega)t} -(\beta+\omega) \zEt, \qquad \zEt(0)=0,
\end{equation*}
whence
\begin{equation}
\label{zet}
\zEt(t)=\frac{\gamma}{\beta+\omega}\mud\re^{-(\beta+\omega)t}\left(1-\re^{-(\beta+\omega)t}\right).
\end{equation}
Thus
\begin{align}
\label{yet}
\yEt(t)&=\etaEt(t)-\xEt(t)-\zEt(t)\nonumber\\
&= \re^{-(\beta+\omega)t}\left(\frac{\beta+\omega+\gamma}{\beta+\omega}\mud \re^{-(\beta+\omega)t}-\frac{\gamma}{\beta+\omega}\mud-\fde'\left(\psi(t)\right)\right).
\end{align}

\begin{remark}[Fractions of susceptible, infectious, and recovered individuals]
Although the above results are useful for analysing the final outcome of the epidemic, of greater practical interest is the evolution of the fractions of the population that are susceptible, infective and recovered individuals, which in the time-transformed process are given by $\xt(t)=\sum_{i=0}^{\infty} \tilde x_i(t), \yt(t)=\sum_{i=0}^{\infty} \tilde y_i(t)$ and $\zt(t)=\sum_{i=0}^{\infty} \tilde z_i(t)$, respectively. Summing~\eqref{xit} over $i=0,1,\ldots$ and using a similar argument to the derivation of~\eqref{xet} yields
\begin{equation}
\label{xt}
\xt(t)=\fde\left(\psi(t)\right).
\end{equation}
Turning to $\yt(t)$, summing~\eqref{difft2} over $i=1,2,\ldots$ and using~\eqref{xet} yields
\begin{equation}
\label{ytdiff}
\dfrac{d\yt}{dt}=-\frac{\gamma}{\rhoEt(t)}\yt+\beta \re^{-(\beta+\omega)t}\fde\left(\psi(t)\right).
\end{equation}
Let $\epsilon=\sum_{i=0}^{\infty} \epsilon_i=\yt(0)$ and
\begin{equation*}
c(t)=\int_0^t \frac1{\tilde \rho_E(u)}\,{\rm d}u.
\end{equation*}
Then~\eqref{ytdiff} has solution
\begin{equation}
\label{yt}
\yt(t)=\re^{-\gamma c(t)}\epsilon+\beta \int_0^t \re^{-[(\beta+\omega)u+\gamma(c(t)-c(u))]}\fde\left(\psi(u)\right)\,{\rm d}u.
\end{equation}
We do not have a closed-form expression for the integral in~\eqref{yt}, though it is straightforward to calculate $\yt(t)$ numerically using the ODE~\eqref{ytdiff}.  Finally, note that $\zt(t)=1-\xt(t)-\yt(t)$.
\end{remark}

\subsubsection{Real-time process}
\label{sec:detrealtime}
Turning to the system of ODEs~\eqref{diff1}-\eqref{diff3}, which describe the limiting evolution of the epidemic as the population size $N \to \infty$, let
\begin{equation}
\label{eq:xi}
\xi(t)= \int_0^t \rho_E(u)\,{\rm d}u,
\end{equation}
where $\rho_E$ is given by~\eqref{eq:rhoE}. Then $\xi'(t)=\rho_E(t)$ and it follows that, for $t \ge 0$,
\begin{equation}\label{eq:timetransform}
\bw(t)=\bwt(\xi(t)),
\end{equation}
connecting the original process to the time-transformed process. Hence, $\xi'(t)=\rhoEt(\xi(t))$, so~\eqref{yet} and~\eqref{etaEt}
imply that $\xi(t)$ is determined by
\begin{equation}
\label{xiODE}
\dfrac{d\xi}{dt}=1+\frac{\gamma}{\beta+\omega}\left(1-\re^{(\beta+\omega)\xi}\right)-\re^{(\beta+\omega)\xi}\frac{\fde'\left(\psi(\xi)\right)}{\mud},
\end{equation}
together with $\xi(0)=0$. The ODE~\eqref{xiODE} does not seem to admit an explicit solution, although it is straightforward to solve numerically.

\subsection{Connection to other approaches}
\label{sec:otherapproaches}
In this section we consider other deterministic formulations of the preventive dropping model and make the connection to the effective degree approach (ODE system~\eqref{diff1}-\eqref{diff3}). Our focus is on the deterministic variable $\theta(t)$ that is defined as follows:
\begin{equation}\label{eq:RE}
\theta(t)=\mathcal F(t)-\int_{0}^t \frac{\fde'(\theta(u))}{\mu_D}\mathcal F'(t-u)\,{\rm d}u.
\end{equation}
Here, $\mathcal F(t)$ is the probability that an individual escapes infection from a given neighbour, up to at least $t$ units of time after the neighbour became infected. In the Markovian SIR case with dropping of edges, this probability equals
\begin{equation}\label{eq:escape}
\mathcal F(t) = \frac{\gamma+\omega}{\beta+\gamma+\omega}+\frac{\beta}{\beta+\gamma+\omega}\re^{-(\beta+\gamma+\omega)t}.
\end{equation}
Indeed, there are three competing events: transmission, ending of the infectious period, and informing the susceptible neighbour, that occur at rates $\beta$, $\gamma$, and $\omega$, respectively. 
%One advantage of the binding site formulation is that 
We see immediately from the renewal equation for $\theta$, obtained by substituting~\eqref{eq:escape} into~\eqref{eq:RE}, that one can also interpret dropping of edges as an increased recovery rate for the deterministic mean temporal behaviour since $\omega$ only appears as part of the sum $\gamma+\omega$ (see Remark~\ref{rmk:binding final1} in Section~\ref{sec:finalsize}). This aspect of the mean temporal behaviour may not be immediately clear from the system~\eqref{diff1}-\eqref{diff3}.

The variable $\theta$ can be interpreted as the probability that along a randomly chosen edge between two individuals, $i$ and $j$ say, there is no transmission from $j$ to $i$ before time $t$, given that no transmission occurred from individual $i$ to $j$. The variable $\theta$ formed the basis for the edge-based compartmental models of Volz, Miller and co-workers (see e.g.~Kiss et al.~\cite{Kiss2017} and references therein). Closely related to edge-based compartmental models is the binding site formulation presented in Leung and Diekmann~\cite{Leung:2016}, where the relation to edge-based compartmental models is also explained. We use the binding site formulation in this section to state the renewal equation for the variable $\theta$, restricting ourselves to the Markovian SIR epidemic (in~\cite{Leung:2016} $\bar x$ is used instead of $\theta$). In principle, the renewal equation~\eqref{eq:RE} is far more general and allows for randomness in infectiousness beyond the Markovian setting, see~\cite{Leung:2016} for details. Note that in the above works, the derivation of the equations describing the evolution of $\theta(t)$ is heuristic. Those equations are proved for the Markov SIR epidemic on a configuration model network, in the sense of a large population limit, in Decreusefond et al.~\cite{Decreusefond:2012} and Janson et al.~\cite{Janson:2014}; see also Barbour and Reinert~\cite{Barbour:2013}.

The variable $\theta$ relates to the effective degree formulation as follows:
\begin{equation}\label{eq:relation}
\theta(t)=p_{\omega}+(1-p_{\omega})\re^{-(\beta+\omega)\xi(t)}=\psi(\xi(t)),
\end{equation}
where the functions $\psi$ and $\xi$ from the effective degree formulation are defined at~\eqref{eq:psi} and~\eqref{eq:xi}, respectively. Indeed, equation~\eqref{eq:relation} is expected from the interpretation of $\theta$: $p_{\omega}$ is the probability that the susceptible individual is informed by the infection status of a given neighbour before being infected by that neighbour, so the stub disappears through dropping, while $(1-p_{\omega})\re^{-(\beta+\omega)\xi(t)}$ is the probability that there is no dropping and the given stub has not disappeared at time $\xi(t)$ (where $\xi(t)$ accounts for the time-transformation, see~\eqref{eq:timetransform}). One can check that~\eqref{eq:relation} holds true by first transforming the renewal equation~\eqref{eq:RE} into an ODE for $\theta$ by differentiating (and using~\eqref{eq:escape}):
\begin{equation}\label{eq:ODE_theta}
\dfrac{d\theta}{dt}=\beta\frac{\fde'(\theta)}{\mu_D}-(\beta+\gamma+\omega)\theta+\gamma+\omega,
\end{equation}
with initial condition $\theta(0)=1$. Next, differentiating the right-hand-side of~\eqref{eq:relation}, and using~\eqref{xiODE}, we find that $\psi(\xi)$ satisfies the ODE~\eqref{eq:ODE_theta}. Furthermore, the initial condition $\xi(0)=0$ implies that $\psi(\xi(0))=1$.

Finally, the Malthusian parameter $r$, the basic reproduction number $R_0$ and the final size of the epidemic are easily derived from the single renewal equation~\eqref{eq:RE}. Here we only state the expressions and refer to Leung and Diekmann~\cite{Leung:2016}, Section 2.5, for details. In the limit of $\epsilon \downarrow 0$ the Euler-Lotka characteristic equation is
\begin{align}
1&=-\frac{f''_{D}(1)}{\mu_D}\int_0^\infty \re^{-\lambda t}\mathcal F'(t)\,{\rm d}t\nonumber\\
&=\left(\mu_D-1+\frac{\sigma_D^2}{\mu_D}\right)\int_0^\infty e^{-\lambda t}\beta \re^{-(\beta+\gamma+\omega)t}\,{\rm d}t.\label{eq:Malthusian}
\end{align}
The Malthusian parameter $r$ is the unique real root of~\eqref{eq:Malthusian} and a simple calculation yields
\begin{equation}
r=\beta\left(\mu_D-2+\frac{\sigma^2_D}{\mu_D}\right)-\gamma-\omega,
\label{eq:MathusianExplicit}
\end{equation}
agreeing with Britton et al.~\cite{Britton2016}, equation (3). The basic reproduction number $R_0$ is obtained from~\eqref{eq:Malthusian} by evaluating the right hand side at $\lambda=0$, yielding the same expression as~\eqref{R_0}. The final size is discussed in Remark~\ref{rmk:binding final}.

\subsection{Final size}
\label{sec:finalsize}
Recall that $\tautd$ defined at~\eqref{taudelta} satisfies $\yEt(\tautd)=\delta$. In particular, using~\eqref{yet}, $\taut=\taut_0$ satisfies
\begin{equation}
\label{finaltau}
\frac{\beta+\omega+\gamma}{\beta+\omega}\mud \re^{-(\beta+\omega)\taut}-\frac{\gamma}{\beta+\omega}\mud
-\fde'\left(\psi(\taut)\right)=0.
\end{equation}
For later use, we rewrite~\eqref{finaltau} as
\begin{equation}
\label{zdef}
\left[\frac{(\beta+\omega+\gamma)z-\gamma}{\beta+\omega}\right]\mud =\fde'\left(\psit(z)\right),
\end{equation}
where $z=\re^{-(\beta+\omega)\taut}$ and $\psit(z)=p_{\omega}+(1-p_{\omega})z$.
Further, using~\eqref{xt} yields that the final proportion of the population that remains uninfected is given by
\begin{equation}
\label{finalsus}
\xt(\taut)=\fde\left(\psi(\taut)\right).
\end{equation}
We let $\rho=1-\xt(\taut)$ denote the fraction of the population ultimately infected in the limiting deterministic epidemic.

Let $\epsE=\sum_{i=1}^{\infty} i\epsilon_i$.  Then in the limit as $\epsE \downarrow 0$, i.e.~for epidemics started by a trace of infection (or, more precisely, a trace of infected stubs), the final susceptible fraction is given by~\eqref{finalsus}, where $\taut$ satisfies
\begin{equation}
\label{finaltautrace}
\frac{\beta+\omega+\gamma}{\beta+\omega}\mud \re^{-(\beta+\omega)\taut}-\frac{\gamma}{\beta+\omega}\mud -f_D'\left(\psi(\taut)\right)=0.
\end{equation}

We can now formulate the characterization for the final size $\rho$ of the epidemic.  We illustrate the dependence of $\rho$ on the dropping rate $\omega$ in Section~\ref{sec:droppingEffect}.

\begin{prop}[Deterministic final size]
\label{prop:final}
\mbox{}  Suppose that $\dmax<\infty$.
\begin{itemize}
\item[(a)]
Suppose that $\epsE>0$. Then the fraction of the population that is ultimately infected in the deterministic epidemic is given by
\begin{equation}
\label{finalsus1}
\rho=1-\fde(s),
\end{equation}
where $s$ is the unique solution in $[0,1)$ of
\begin{equation}
\label{finals}
(\beta+\omega+\gamma)s-(\omega+\gamma)=\beta \mud^{-1} \fde'(s).
\end{equation}
\item[(b)]
Suppose $R_0>1$.  Then in the limit as $\epsE \downarrow 0$,  the fraction of the population that is
ultimately infected in the limiting deterministic epidemic is given by
\begin{equation}
\label{finalsusz}
\rho=1-f_D(s),
\end{equation}
where $s$ is the unique solution in $[0,1)$ of
\begin{equation}
\label{finalz}
(\beta+\omega+\gamma)s-(\omega+\gamma)=\beta \mud^{-1} f_D'(s).
\end{equation}
\end{itemize}
\end{prop}
\begin{proof}
(a) Suppose that $\epsE>0$.  Let $s=\psit(z)$, so $z=\frac{(\beta+\omega)s-\omega}{\beta}$.  It then follows from~\eqref{zdef} and~\eqref{finalsus} that $s$ satisfies~\eqref{finals} and $\rho$ is given by~\eqref{finalsus1}.  Let $g_1(s)=(\beta+\omega+\gamma)s-(\omega+\gamma)$ and $g_2(s)=\beta \mud^{-1} \fde'(s)$.  Then $g_1(0)\le 0 <g_2(0)$ and $g_1(1)>g_2(1)$, since $\fde'(1)=\sum_{i=1}^{\infty} i(p_i-\epsilon_i)<\sum_{i=1}^{\infty}i p_i =\mud$.  Thus~\eqref{finals} has a unique solution in $[0,1)$ as
$g_2$ is convex on $[0,1]$, since $g_2''(s)\ge 0$.

(b) Letting $\epsE \downarrow 0$ in~\eqref{finalsus1} and~\eqref{finals} shows that $\rho$ is given by~\eqref{finalsusz}, where $s$ satisfies~\eqref{finalz}.  Let $g_1$ be as in (a) and $g_2(s)=\beta \mud^{-1} f_D'(s)$.  Then $g_1(0)\le 0 <g_2(0)$ and $g_1(1)=g_2(1)$, since $\mud= f_D'(1)$.  Further $g_2$ is a convex function, so it follows that~\eqref{finalz} has a solution in $[0,1)$ if and only if $g_1'(1)<g_2'(1)$ and moreover that solution is unique. Now $g_1'(1)=\beta+\omega+\gamma$ and $g_2'(1)=\beta \mu_D^{-1} f_D''(1)$, so $g_1'(1)<g_2'(1)$ if and only if $R_0=\frac{\beta}{\beta+\omega+\gamma}\mu_D^{-1} f_D''(1)>1$.
$\Box$	
\end{proof}

\begin{remark}[Connection to the renewal equation~\eqref{eq:RE}]
\label{rmk:binding final}
Proposition~\ref{prop:final}(b) can also be derived by taking the limit $t\to\infty$ in~\eqref{eq:RE}:
\begin{align}
\theta(\infty)&=\mathcal{F}(\infty)+(1-\mathcal F(\infty))\frac{f'_{D}(\theta(\infty))}{\mu_D}\nonumber\\
&=\frac{\gamma+\omega}{\beta+\gamma+\omega}+\frac{\beta}{\beta+\gamma+\omega}\frac{f'_{D}(\theta(\infty))}{\mu_D},\label{eq:final_theta}
\end{align}
using~\eqref{eq:escape}, so $\theta(\infty)$ satisfies~\eqref{finalz}. Then, using~\eqref{xt} and~\eqref{eq:timetransform}, one obtains that the proportion $x(\infty)$ of the population that ultimately is susceptible agrees with~\eqref{finalsusz}.
%Next, write $g(\theta(\infty))$ for the right-hand side of~\eqref{eq:final_theta}. Then $x\mapsto g(x)$ is a monotonically increasing function for $0\leq x\leq0$. Furthermore $g(0)>0$, $g(1)=1$, and $g'(1)=R_0$.  Therefore, if $R_0<1$, then there is only the trivial solution $\theta(\infty)=1$. If $R_0>1$ then there is a unique non-trivial solution $0<\theta(\infty)<1$.
\end{remark}

\begin{remark}[Increased recovery rate and no dropping]
\label{rmk:binding final1}
Observe that equations~\eqref{eq:RE} for $\theta$ and~\eqref{eq:escape} for $\mathcal F$ together imply immediately that the process of susceptibles in the deterministic model with recovery rate $\gamma$ and dropping rate $\omega$ depends on $(\gamma, \omega)$ only through their sum $\gamma+\omega$, since $\omega$ only appears in~\eqref{eq:escape} through the sum $\gamma+\omega$. Furthermore,~\eqref{eq:relation} relates the variable $\theta$ of the binding site formulation to the effective degree formulation through $\psi$ and $\xi$ defined at~\eqref{eq:psi} and~\eqref{eq:xi}, respectively. Thus the LLN limit $\{\bx(t)\}$ describing the evolution of susceptibles classified by their effective degree for the model with dropping is the same as that for the model without dropping (i.e.~the standard Markov SIR epidemic on a configuration model network) but with the recovery rate $\gamma$ increased to $\gamma+\omega$.
In particular, this implies that the deterministic final size $\rho$ of the two models are the same, as is apparent immediately from Proposition~\ref{prop:final}. This invariance also holds for the basic reproduction number $R_0$ and Mathusian parameter $r$, as is clear from the formulae in equations~\eqref{R_0} and~\eqref{eq:MathusianExplicit}, respectively.
Note however that the LLN limit $\{\by(t)\}$ describing the infectives is not the same for these two models, since infectives recover more quickly in the model with increased recovery rate. Thus (as illustrated in Figure~\ref{fig:incRecoveryTimeCourse} in Section~\ref{sec:incRecovery}) at any time $t>0$ there are more infectives in the deterministic model with dropping than in the corresponding model with increased recovery rate and no dropping. We revisit the model with increased recovery rate and no dropping in Section~\ref{sec:relatedmodel}, where we focus on the
probability of a major outbreak in the stochastic model with few initial infectives.
%Observe that equations~\eqref{eq:escape} and~\eqref{eq:RE} together show immediately that the deterministic (mean) process of an SIR epidemic with recovery rate $\gamma$ and dropping rate $\omega$ depends on $(\gamma, \omega)$ only through their sum $\gamma+\omega$. Thus the LLN limit of the model with dropping is the same as that of the model without dropping (i.e.~the standard Markov SIR epidemic on a configuration model network) but with recovery rate $\gamma$ replaced by $\gamma+\omega$. In particular, this implies that the deterministic final size $\rho$ of the model with dropping is the same as that of the model with no dropping and increased recovery rate $\gamma+\omega$, as is apparent immediately from Proposition~\ref{prop:final}. We revisit the model with increased recovery rate and no dropping in Section~\ref{sec:relatedmodel}, where we focus on the probability of a major outbreak in the stochastic model with few initial infectives.
\end{remark}

\section{Asymptotic variance of final size of epidemic on an MR random graph}
\label{sec:varfinal}

Recall that $X^N(\tau^N)=\sum_{i=0}^{\infty} X_i^N(\tau^N)$ denotes the number of susceptibles remaining at the end of the epidemic on an MR random graph.  Thus $\TnMR=X^N(0)-X^N(\tau^N)$ denotes the final size of the epidemic. Note that, in an obvious notation, $X^N(\tau^N)= \tilde{X}^N(\tautN)=\sum_{i=0}^{\infty} \tilde{X}_i^N(\tautN)$. Let $\bzero=(0,0,\ldots)$ and $\bone=(1,1,\ldots)$.  
Then, assuming the truth of Conjecture~\ref{conj:MRcltfin} for $\dmax=\infty$, the asymptotic variance of $N^{-\frac{1}{2}}\TnMR$ is
given by
\begin{equation}
\label{equ:SigmaMR}
\SigmaMR(\beta,\omega,\gamma)=(\bone, \bzero,0)\Sigma_{{\rm MR},0}(\bone, \bzero,0)^{\top}.
\end{equation}
Suppose that $\epsE=\sum_{i=1}^{\infty} i\epsilon_i>0$ and let $z$ be the unique solution in $[0,1)$ of~\eqref{zdef}; cf.~Proposition~\ref{prop:final}(a).  
The following proposition gives an almost fully explicit expression for the asymptotic variance $\SigmaMR(\beta,\omega,\gamma)$.

\begin{prop}[Asymptotic variance of final size of epidemic on MR graph with dropping]
\label{prop:mrVar}
\mbox{}
Suppose that $\epsE>0$ and $z>0$. Then,
\begin{align}
\label{MRCLTvar}
\SigmaMR(\beta,\omega,\gamma)=&2\frac{(\beta+\omega+\gamma)[\gamma-\beta-\omega-(\beta+\omega+\gamma)z]}{(\beta+\omega)^2}\mud \btz^2 z^2(1-z)\nonumber\\
&\quad +\frac{\gamma}{\beta(\beta+\omega)} \mud \btz^2 z[\beta-(2\beta+\omega)z]\nonumber\\
&\quad+\frac{\gamma}{\beta[2(\beta+\omega)+\gamma]}\btz^2 z^2\left[\beta(\sigma_D^2+\mud^2)+\omega\mud\right]\nonumber\\
&\quad -\frac{\gamma[(\beta+\omega+\gamma)z-\gamma]z}{[2(\beta+\omega)+\gamma](\beta+\omega)} \mud \btz +I_A+I_B+I_C+I_D,
\end{align}
with
\begin{align}
\btz&=\frac{\beta\left[\frac{(\beta+\omega+\gamma)z-\gamma}{\beta+\omega}\right] \mud}
{z\left[\beta \fde''\left(\psit(z)\right)-(\beta+\omega+\gamma)\mud\right]},\label{btz}\\
I_A&=\frac{1}{\beta+\omega} \int_z^1 \left[\omega\left(\psit_3(z,v)-1\right)^2+\beta \psit_3(z,v)^2\right]
\fde'\left(\psit_2(z,v)\right)\,{\rm d}v,\label{IA}\\
I_B&=2\frac{\omega z \btz}{\beta+\omega} \int_z^1 \psit_1(z,v)\left(\psit_1(z,v)-1\right)\left(1-\psit_3(z,v)\right)
\fde''\left(\psit_2(z,v)\right)\,{\rm d}v,\label{IB}\\
I_C&=\frac{\beta z \btz}{\beta+\omega} \int_z^1 \psit_1(z,v)^2\left(\btz z v^{-1}-
2\psit_3(z,v)\right)\fde''\left(\psit_2(z,v)\right)\,{\rm d}v,\label{IC}\\
I_D&=\frac{z^2 \btz^2}{\beta+\omega} \int_z^1 \left[\omega \left(\psit_1(z,v)-1\right)^2+\beta \psit_1(z,v)^2 \right]
\psit_1(z,v)^2 \fde^{(3)}\left(\psit_2(z,v)\right)\,{\rm d}v,\nonumber \\&\label{ID}
\end{align}
$\psit_1(z,v)=p_{\omega}+(1-p_{\omega})zv^{-1},\psit_2(z,v)=v\psit_1(z,v)^2+p_{\omega}(1-v)$ and \newline
$\psit_3(z,v)=\psit_1(z,v)-\btz z v^{-1}$.
\end{prop}

\begin{proof}
The proof is rather long so only an outline is given here, with detailed calculations deferred to appendices. 
Let
\begin{equation}
\label{vecc}
\bc(\taut,u)=(\bone, \bzero,0)B\Phit(\taut,u),
\end{equation}
where $B$ is given by~\eqref{eq:Bdelta} with $\delta=0$. Then, using~\eqref{Sigma1} and~\eqref{Gtilde},
\begin{eqnarray}
\SigmaMR(\beta,\omega,\gamma)&=&\int_0^{\taut} \bc(\taut,u)
\Gt(\bwt(u))\bc(\taut,u)^\top \,{\rm d}u,\nonumber\\
&=&\sum_{\bl \in \Delta}\int_0^{\taut} \bc(\taut,u)\bl \bl^{\top} \bc(\taut,u)^\top \betat_{\bl}(\bwt(u)) \,{\rm d}u. \label{sigma2}
\end{eqnarray}
The rest of the proof involves showing that the right-hand side of~\eqref{sigma2} yields the expression~\eqref{MRCLTvar} for $\SigmaMR(\beta,\omega,\gamma)$.

Recall that $\Delta=\cup_{k=1}^5 \Delta_k$ and note that $\bc(\taut,u)\bl$ is a scalar.  It then follows that
\begin{equation}
\label{mrsigma2}
\SigmaMR(\beta,\omega,\gamma)=\sum_{i=1}^5 \sigma^2_i,
\end{equation}
where
\begin{equation}
\label{mrsigma2a}
\sigma^2_i=\int_0^{\taut}\sum_{\bl \in \Delta_i}(\bc(\taut,u)\bl)^2 \betat_{\bl}(\bwt(u))\,{\rm d}u.
\end{equation}
Evaluation of~\eqref{mrsigma2a} requires $\bc(\taut,u)$, which we now determine.

Let $a(\taut)=\nabla \varphi(\bwt(\taut)) \cdot \tilde{F}(\bwt(\taut))$.
Observe that $\nabla \varphi(\bwt(\taut))=(\bzero,\bp,0)$, where $\bp=(0,1,2,\ldots)$, so  using~\eqref{driftF1},
\begin{eqnarray}
a(\taut)&=&-(\beta+\omega)[\yEt(\taut)+\etaEt(\taut)]-\gamma \etaEt(\taut)+\beta \sum_{i=1}^{\infty} i (i+1)\xt_{i+1}(\taut)\label{ataut1}\\
&=&\re^{-2(\beta+\omega)\taut}\left[\beta\fde''\left(\psi(\taut)\right)
-(\beta+\omega+\gamma)\mud\right],\label{ataut}
\end{eqnarray}
using $\yEt(\taut)=0$,~\eqref{sumiip1x} and~\eqref{etaEt}.  Also, using~\eqref{driftF1},
$(\bone, \bzero,0)\tilde{F}(\bwt(\taut))=-\beta \xEt(\taut)$, so
\begin{equation}
\label{one00B}
(\bone, \bzero,0)B=(\bone, b(\taut) \bp,0),
\end{equation}
where
\begin{equation}
\label{btaut}
b(\taut)=a(\taut)^{-1}\beta \xEt(\taut).
\end{equation}

Note from~\eqref{driftF1} that $\partial \tilde{F}(\bwt(t))$ takes the partitioned form
\begin{equation}
\label{partialFtilde}
\partial \tilde{F}(\bwt(t))=
\begin{bmatrix}
\partial\tilde{F}_{XX}(\bwt(t)) & 0 & \bzero^\top \\
\partial\tilde{F}_{YX}(\bwt(t)) &\partial\tilde{F}_{YY}(\bwt(t)) &\partial\tilde{F}_{YZ}(\bwt(t)) \\
\partial\tilde{F}_{ZX}(\bwt(t)) &\partial\tilde{F}_{ZY}(\bwt(t)) &\partial\tilde{F}_{ZZ}(\bwt(t))
\end{bmatrix}.
\end{equation}
It follows from~\eqref{Phit} that $\Phit(t,u)$ has the partitioned form
\begin{equation*}
\Phit(t,u)=
\begin{bmatrix}
\Phit_{XX}(t,u) & 0 & \bzero^\top \\
\Phit_{YX}(t,u) & \Phit_{YY}(t,u) & \Phit_{YZ}(t,u) \\
\Phit_{ZX}(t,u) & \Phit_{ZY}(t,u) & \Phit_{ZZ}(t,u)
\end{bmatrix}.
\end{equation*}
Thus, using~\eqref{vecc} and~\eqref{one00B}, we have
\[
\bc(\taut,u)=\left(\bone \Phit_{XX}(\taut,u)+b(\taut)\bp \Phit_{YX}(\taut,u), b(\taut)\bp\Phit_{YY}(\taut,u), b(\taut)\bp\Phit_{YZ}(\taut,u)\right).
\]
We show in Appendix~\ref{app:phitcalculations} that
\begin{align*}
\left(\bone \Phit_{XX}(\taut,u)\right)_j&=\psi(\taut-u)^j \qquad (j=0,1,\ldots),\\
\left(\bp \,\Phit_{YX}(\taut,u)\right)_j&=\re^{-(\beta+\omega)(\taut-u)}j\left[\frac{(\beta+\omega+\gamma)
\re^{-(\beta+\omega)(\taut-u)}-\gamma}{\beta+\omega}\right]\\
&\,-\re^{-(\beta+\omega)(\taut-u)}
j \psi(\taut-u)^{j-1}\qquad (j=0,1,\ldots),\\
\bp \,\Phit_{YY}(\taut,u)&=\left(\frac{\beta+\omega+\gamma}{\beta+\omega}\re^{-2(\beta+\omega)(\taut-u)}-
\frac{\gamma}{\beta+\omega}\re^{-(\beta+\omega)(\taut-u)}\right) \bp,\\
\bp \,\Phit_{YZ}(\taut,u)&=-\frac{\beta+\omega+\gamma}{\beta+\omega}\re^{-(\beta+\omega)(\taut-u)}\left(1-\re^{-(\beta+\omega)(\taut-u)}\right),
\end{align*}
see~\eqref{bonePhitxxj},~\eqref{bonePhityxj},~\eqref{pPhityy} and~\eqref{pPhityz}, respectively.
Hence,
\begin{equation}
\label{boldc}
\bc(\taut,u)=\left(\bc_S(\taut,u), h_I(\taut,u) \bp, h_R(\taut,u)\right),
\end{equation}
where
\begin{align}
h_I(\taut,u)&=-\frac{b(\taut)}{\beta+\omega}\re^{-(\beta+\omega)(\taut-u)}\left[\gamma-(\beta+\omega+\gamma)\re^{-(\beta+\omega)(\taut-u)}\right],\label{hItauu}\\
h_R(\taut,u)&=h_I(\taut,u)-b(\taut)\re^{-(\beta+\omega)(\taut-u)},\label{hRtauu}\\
\bc_S(\taut,u)&=(\tilde{c}_0(\taut,u), \tilde{c}_1(\taut,u), \ldots)+h_I(\taut,u)\bp,\label{cStauu}
\end{align}
with
\[
\tilde{c}_j(\taut,u)=\psi(\taut-u)^j-b(\taut)\re^{-(\beta+\omega)(\taut-u)}j\psi(\taut-u)^{j-1} \qquad (j=0,1,\ldots).
\]
%\begin{equation}
%\label{psitdef}
%\psi(t)=p_{\omega}+(1-p_{\omega})\re^{-(\beta+\omega)t}.
%\end{equation}

We can now calculate $\sigma^2_i$ $(i=1,2,\ldots,5)$ using~\eqref{mrsigma2a}, \eqref{boldc} and~\eqref{intensityfun1},
and hence obtain $\SigmaMR(\beta,\omega,\gamma)$ using~\eqref{mrsigma2}.  The details are lengthy and are given in Appendix~\ref{app:mrvariance}.
$\Box$	
\end{proof}

Recall from Section~\ref{sec:finalsize} that if $\epsE>0$ then $\rho=1-\fde\left(\psit(z)\right)$, where $z$ is the unique solution in $[0,1)$ of~\eqref{zdef}, and if $\epsE=0$ and $R_0>1$ then $\rho=1-\fd\left(\psit(z)\right)$, where $z$ is the unique solution in $[0,1)$ of~\eqref{zdef} with $\fde'$ replaced by $\fd'$; cf.~Proposition~\ref{prop:final}. 

\begin{conj}[CLT for of final size of epidemic on MR graph with dropping]
\label{conj:mrCLT}
\mbox{}
\begin{itemize}
\item[(a)]
Suppose that $\epsE>0, \dmax<\infty$ and $z>0$. Then,
\begin{equation}
\label{CLTMRE}
\sqrt{N}\left(N^{-1}\TnMR - \rho \right) \convD N(0,\SigmaMR(\beta,\omega,\gamma)) \quad\mbox{as } N \to \infty,
\end{equation}
where
$\SigmaMR(\beta,\omega,\gamma)$ is given by Proposition~\ref{prop:mrVar}.
\item[(b)]
Suppose that $\epsE=0, \dmax<\infty, R_0>1$ and $z>0$.  Then, in the event of a major outbreak, \eqref{CLTMRE} holds with $D_{\epsilon}$ replaced by $D$ in~\eqref{btz}-\eqref{ID}.
\end{itemize}
\end{conj}
\begin{remark}[Proving Conjecture~\ref{conj:mrCLT}]
\label{rmk:eps=0}
Part (a) of Conjecture~\ref{conj:mrCLT} follows immediately from Conjecture~\ref{conj:MRcltfin} and Proposition~\ref{prop:mrVar};
see Remark~\ref{rmk:MRcltfin1} for how Conjecture~\ref{conj:MRcltfin} might be proved.
Part (b) of Conjecture~\ref{conj:mrCLT} is concerned with epidemics started by a trace of infection, i.e.~with $\epsE=0$.
Similar CLTs for the final size of a wide range of SIR epidemics (e.g.~von Bahr and Martin-L{\"o}f~\cite{vBahr:1980}, Scalia-Tomba~\cite{STomba:1985} and Ball and Neal~\cite{Ball:2003}) suggest that letting $\epsE \downarrow 0$ in the CLT with $\epsE>0$ yields the correct CLT when $\epsE=0$ for epidemics that
become established and lead to a major outbreak.  This is proved for the SIR epidemic without dropping of edges on configuration model networks in Ball~\cite{Ball:2018},
using the modified epidemic model outlined in Remark~\ref{rmk:MRcltfin1}. A similar proof should hold for the present model with dropping of edges.  
\end{remark}

\begin{remark}[The condition $z>0$]
\label{rmk:epsE=01}
%The condition $\rho<1$ is required to apply the theorems in Ethier and Kurtz~\cite{Ethier:1986}, Chapter 11; see Appendix~\ref{app:kurtz}.  Clearly, a different form of asymptotic distribution for $\TnMR$ is needed when $\rho=1$.  In most applications $\rho<1$.  However, if $\gamma=\omega=0$ (see Section~\ref{sec:configuration}) then $\rho=1$ is clearly possible, e.g.~when every individual has degree $k$ for some fixed $k \ge 3$.
The condition $z>0$ in Proposition~\ref{prop:mrVar} and Conjecture~\ref{conj:mrCLT} is required to ensure that $\taut<\infty$; recall from Section~\ref{sec:finalsize} that $z=\re^{-(\beta+\omega)\taut}$.  Note from~\eqref{finalsus1} that $z>0$ implies $\rho<1$, so the LLN and functional CLT in Ethier and Kurtz~\cite{Ethier:1986}, Chapter 11, hold for both the original and random time-scale transformed processes $\{\bWN(t)\}$ and $\{\bWNt(t)\}$ provided there is a maximum degree; see Appendix~\ref{app:kurtz}.  Further, as explained in Appendix~\ref{app:tautdelta}, if $\epsE>0$ then $z=0$ if and only if
$\gamma=\omega=\fde'(0)=0$.  Now $\fde'(0)=0$ if and only if $p_1-\epsilon_1=0$.  Thus $z>0$ unless there is no recovery of
infectives, no droping of edges and the limiting fraction of degree-$1$ susceptibles is $0$.  The same conclusion holds when 
$\epsE=0$. 
\end{remark}

\section{Extension to iid degrees: epidemics on an NSW random graph}
\label{sec:iiddegrees}
In this section we assume that the underlying network is constructed from a sequence $D_1,D_2,\ldots$ of independent and identically distributed copies of the random variable $D$, which describes the degree of a typical individual. The random variables $D_1,D_2,\ldots,D_N$ are used to construct a network of $N$ individuals, yielding a realisation of NSW random graph. The almost sure convergence results described in Theorem~\ref{thm:as} (and the corresponding time-transformed almost sure convergence result of Section~\ref{sec:final}) still hold for the present model, as noted previously, but the functional CLT and the CLT for the final size (Theorem~\ref{thm:MRclt} and Conjecture~\ref{conj:mrCLT}) need modifying, as the variability in the empirical degree distribution of the random network (and hence in the initial conditions for the effective degree process $\{\bWN(t)\}$) is of the same order of magnitude as that of the process itself. The modified results for epidemics on an NSW random graph are presented in Theorem~\ref{thm:NSWtemporalCLT} and Conjecture~\ref{conj:nswCLT}. In order to prove and motivate, respectively, these results we need 
a version of the functional CLT (Theorem 11.2.3) in Ethier and Kurtz~\cite{Ethier:1986} that allows for asymptotically random initial conditions; see Theorem~\ref{KurtzFCLTrandinit} below, which may be of  more general interest beyond the present paper.
 Like the above-mentioned Theorem 11.2.3, Theorem~\ref{KurtzFCLTrandinit} assumes a 
finite-dimensional state space, which for our application amounts to assuming that $\dmax < \infty$.

The limiting Gaussian process $\{\bV(t)\}$ in Theorem~\ref{thm:MRclt} admits the It\^o integral representation
\begin{equation}
\label{ItoV}
\bV(t)=\Phi(t,0)\bV(0)+\int_0^t \Phi(t,s) \,{\rm d}\bU(s) \qquad(t \ge 0),
\end{equation}
where $\{\bU(t)\}$ is a time-inhomogeneous Brownian motion (see  Ethier and Kurtz~\cite{Ethier:1986}, Theorem 11.2.3, page 458) and $\bV(0)=\lim_{N \to \infty} \sqrt{N}\left(\bWN(0)-\bw(0)\right)$. (To aid connection with Ethier and Kurtz~\cite{Ethier:1986}, $\bV(t)$ and $\bU(t)$ are now column vectors.) In Ethier and Kurtz~\cite{Ethier:1986}, Theorem 11.2.3, $\bV(0)$ is nonrandom. In Theorem~\ref{KurtzFCLTrandinit} below, we allow $\bV(0)$ to be random.

\begin{thm}[Functional CLT for process with asymptotically random initial conditions]
\label{KurtzFCLTrandinit}
\mbox{}\\Suppose that the conditions of Ethier and Kurtz~\cite{Ethier:1986}, Theorem 11.2.3, are satisfied except that \\
$\sqrt{N}\left(N^{-1}\bWN(0)-\bw(0)\right) \convD
\bV(0)$ as $N \to \infty$, where $\bV(0) \sim N(0,\Sigma_0)$.  Then
\begin{equation}
\label{FCLTrandinit}
\sqrt{N}\left(\{N^{-1}\bWN(t)\}-\{\bw(t)\}\right) \Rightarrow \{\bV(t)\} \quad\mbox{as } N \to \infty,
\end{equation}
where $\{\bV(t)\}=\{\bV(t):t \ge 0\}$ is a zero-mean
Gaussian process with covariance function given, for $t_1,t_2 \ge 0$, by
\begin{equation}
\label{covinit}
\cov\left(\bV(t_1), \bV(t_2)\right)=\Phi(t_1,0)\Sigma_0 \Phi(t_2,0)^{\top} + \int_0^{\min(t_1,t_2)}\Phi(t_1,u)
G(\bw(u))\Phi(t_2,u) ^{\top} \,{\rm d}u.
\end{equation}
\end{thm}

\begin{proof}
It is easily seen that the proof of Ethier and Kurtz~\cite{Ethier:1986}, Theorem 11.2.3, continues to hold in this more general setting.  In particular, the limiting process satisfies~\eqref{ItoV}, where now $\bV(0) \sim N(0,\Sigma_0)$, so $\{\bV(t)\}$ is a zero-mean Gaussian process.  Further, the time-inhomogeneous Brownian motion $\{\bU(t)\}$ arises as the weak limit, as $N \to \infty$, of the (suitably centred and scaled) Poisson processes used to construct realisations of $\{\bWN(t)\}$ $(N=1,2,\ldots)$, and hence is independent of $\bV(0)$.  The covariance function in~\eqref{covinit} then follows immediately from~\eqref{ItoV}.
$\Box$	
\end{proof}

\begin{remark}[Computing the asymptotic variance]
\label{rem:NSWasympVar}
Setting $t_1=t_2=t$ in~\eqref{covinit} and differentiating as in Remark~\ref{rmk:cltvart} shows that 
$\Sigma(t) = \var(\bV(t))$ satisfies the ODE~\eqref{diffSigma} but now with initial condition
$\Sigma(0)=\Sigma_0$.
%We can derive a system of ODEs for the asymptotic variance from Theorem~\ref{KurtzFCLTrandinit} in the same way as in the MR garph situation in Remark~\ref{rmk:cltvart}. Here equation~\eqref{covinit} yields that
%\begin{equation*}
%\Sigma(t) = \var(\bV(t)) = \Phi(t,0) \Sigma_0 \Phi(t,0)^\top + \int_0^t \Phi(t,u) G(\bw(u)) \Phi(t,u)^\top \, {\rm d}u.
%\end{equation*}
%It then follows from this and~\eqref{Phi} (the ODE defining $\Phi$) that
%\begin{align}
%\frac{d}{dt} \Sigma_{\rm NSW}(t) & = \frac{d}{dt} \left( \Phi(t,0) \Sigma_0 \Phi(t,0)^\top \right) + \frac{d}{dt} \Sigma_{\rm MR}(t) \nonumber \\
% & = \partial F(\bw(t)) \Sigma_0 + \Sigma_0  \partial F(\bw(t))^\top + \frac{d}{dt} \Sigma_{\rm MR}(t) \nonumber \\
% & = \partial F(\bw(t)) (\Sigma_0+\Sigma(t)) + (\Sigma_0+\Sigma(t))  \partial F(\bw(t))^\top + G(\bw(t)). \label{diffSigmaNSW}
%\end{align}
%\red{Someone please check my (Dave's) matrix/vector differentiation here! Details could be omitted, just giving $\frac{d}{dt} \Sigma_{\rm NSW}(t) =$ the final equation.}
\end{remark}

\begin{remark}[Non-Gaussian limiting initial conditions]
The covariance function~\eqref{covinit} also holds when $\bV(0)$ is non-Gaussian, provided $\E[\bV(0)]=0$ and $\var(\bV(0))=\Sigma_0$, though of course $\{\bV(t)\}$ is no longer Gaussian.  
%Note that, writing $\bw(t)=(w_1(t),w_2(t),\ldots)$, with this labelling of states}
%\[
%\phi_{ij}(t,0)=\dfrac{\partial w_i(t)}{\partial w_j(0)} \qquad(i,j=1,2,\ldots).
%\]
%It follows that Theorem~\ref{KurtzFCLTrandinit} puts the approximation results in Pollett {\it et al.}~\cite{Pollett:2010} in a fully rigorous asymptotic framework.}
\end{remark}

\begin{thm}[Functional CLT for epidemic on NSW graph with dropping]
\label{thm:NSWtemporalCLT}
\mbox{}\\Suppose that $\sqrt{N}\left(N^{-1}(\bXN(0),\bYN(0),Z^N_E(0))-(\bx(0),\by(0),z_E(0))\right) \convD N(0,\Sigma_0)$ as $N\to\infty$. Then the same functional CLT holds as in the MR graph situation (Theorem~\ref{thm:MRclt}), but with the covariance function of $\{\bV(t)\}$ changed in accordance with equation~\eqref{covinit} and Remark~\ref{rem:NSWasympVar} to reflect the randomness in the initial conditions.
\end{thm}
\begin{proof}
The details of the proof, applying Theorem~\ref{KurtzFCLTrandinit}, are exactly the same as those in Appendix~\ref{app:kurtz} where Theorem 11.2.3 of Ethier and Kurtz~\cite{Ethier:1986} is applied to prove Theorem~\ref{thm:MRclt}.
\end{proof}

\begin{remark}[The asymptotic variance matrix $\Sigma_0$]
\label{rem:NSWSigma0} 
Note that $\Sigma_0$ in Theorem~\ref{thm:NSWtemporalCLT} depends on how the initial infectives are chosen from the population. An example and some discussion can be found in Section~\ref{sec:implementation}. Also note that Theorem~\ref{thm:NSWtemporalCLT} as presented allows for the possibility of some initially recovered individuals in the population. This is to simplify the presentation of the theorem; the assumption of no initially recovered individuals implies that $Z^N_E(0)=0$, from which it follows that $z_E(0)=0$ and the last row and column of $\Sigma_0$ have all entries 0.
\end{remark}

Next, we use Theorem~\ref{KurtzFCLTrandinit} to conjecture a CLT for the final size of the epidemic on an NSW random graph. For $N=1,2,\ldots$, let $\DN$ denote a random variable with distribution given by the empirical distribution of
$D_1,D_2,\ldots,D_N$, so
\begin{equation}
\label{eq:empD}
\Pp\left(\DN=k\right)=N^{-1}\sum_{i=1}^N 1_{\{D_i=k\}}\qquad (k=0,1,\ldots).
\end{equation}
For $N=1,2,\ldots$, let $\TnNSW$ be the final size of the epidemic on an NSW configuration model random graph having $N$ vertices.  We consider epidemics initiated by a trace of infection and assume that the variability in the initial conditions
is owing entirely to the variability in $\DN$.
\begin{conj}[CLT for final size of epidemic on NSW graph with dropping]
\label{conj:nswCLT}
\mbox{}\\Suppose that $\epsE=0$, $\dmax<\infty$, $R_0>1$ and $z>0$.  Then, in the event of a major outbreak,
\begin{equation}
\label{CLTNSWE}
\sqrt{N}\left(N^{-1}\TnNSW - \rho \right) \convD N(0,\SigmaNSW(\beta,\omega,\gamma)) \quad\mbox{as } N \to \infty,
\end{equation}
where
\begin{equation}
\label{sigmatNSW1}
\sigma_{\rm NSW}^2(\beta, \omega, \gamma)=\sigma_{\rm MR}^2(\beta, \omega, \gamma)+\sigma_0^2(\beta, \omega, \gamma),
\end{equation}
with $\sigma_{\rm MR}^2(\beta, \omega, \gamma)$ given by~\eqref{MRCLTvar} (replacing $D_{\epsilon}$ by $D$ in~\eqref{btz}-\eqref{ID}) and
\begin{align}
\label{sigma0}
&\sigma_0^2(\beta, \omega, \gamma)=\nonumber\\
&f_D\left(\psit(z)^2\right)-(1-\rho)^2+\btz^2 \psit(z)^2 z^2 f_D''\left(\psit(z)^2\right)\nonumber\\
&+\btz f_D'\left(\psit(z)^2\right)z\left[z\btz-2\psit(z)\right]\nonumber\\
&+\btz^2z^2\left(\frac{(\beta+\omega+\gamma)z-\gamma}{\beta+\omega}\right)^2\left(\sigma_D^2+\mud^2\right)\nonumber\\
&-2\btz^2 z^2 \mud \left(\frac{(\beta+\omega+\gamma)z-\gamma}{\beta+\omega}\right)
\left[\frac{(\beta+\omega+\gamma)z-\gamma}{\beta+\omega}+\frac{(\beta+\omega+\gamma)}{\beta}\psit(z)\right].
\end{align}
\end{conj}

We now give the argument leading to this conjecture.  Suppose, for the time being, that $\epsE>0$ and 
consider the random time-scale transformed process $\{\bWNt(t)\}$, defined in
Section~\ref{sec:final}, but now for the epidemic on an NSW network.  Using~\eqref{FCLTt}
and Theorem~\ref{KurtzFCLTrandinit}, for any $t_0 \in [0,\taut)$,
\begin{equation*}
\sqrt{N}\left(\{N^{-1}\bWNt(t):0 \le t \le t_0\}-\{\bwt(t):0 \le t \le t_0\}\right) \Rightarrow \{\bVt_{\rm NSW}(t)\} \quad\mbox{as } N \to \infty,
\end{equation*}
where $\{\bVt_{\rm NSW}(t):0 \le t \le t_0\}$ is a zero-mean
Gaussian process with variance-covariance matrix at time $t$ given by
\begin{equation}
\label{equ:SigmatNSW}
\SigmatNSW(t)=\SigmatMR(t)+\Sigmat^0(t);
\end{equation}
$\SigmatMR(t)$ is given by~\eqref{Sigma1} and $\Sigmat^0(t)=\Phi(t,0)\Sigma_0 \Phi(t,0)^{\top}$, with
$\Sigma_0$ being defined as in Theorem~\ref{KurtzFCLTrandinit}.
Then arguing as in the derivation of Proposition~\ref{prop:MRcltfin} yields, for any $\delta \in (0,y_E(0))$,

\begin{equation}
\label{CLT2NSW}
\sqrt{N}\left(N^{-1} \bWN(\tau^N_{\delta})-\bw(\tau_{\delta})\right) \convD N\left(\bzero, \Sigma_{{\rm NSW},\delta} \right),
\quad\mbox{as } N \to \infty,
\end{equation}
where
\begin{equation}
\label{equ:SigmadNSW}
\Sigma_{{\rm NSW},\delta} =\Bd \SigmatNSW(\tautd) \Bd^{\top}.
\end{equation}

We now assume that~\eqref{CLT2NSW} extends to the case $\delta=0$, so~\eqref{CLTNSWE} holds with 
\[
\SigmaNSW(\beta,\omega,\gamma)=(\bone, \bzero,0)\Sigma_{{\rm NSW},0}(\bone, \bzero,0)^{\top};
\]
cf.~\eqref{equ:SigmaMR}.
Thus, using~\eqref{equ:SigmatNSW} and~\eqref{equ:SigmadNSW},
\begin{equation}
\label{sigmatNSW}
\sigma_{\rm NSW}^2(\beta, \omega, \gamma)=\sigma_{\rm MR}^2(\beta, \omega, \gamma)+\sigma_0^2(\beta, \omega, \gamma),
\end{equation}
where
\begin{eqnarray}
\sigma_0^2(\beta, \omega, \gamma)&=&(\bone, \bzero,0)B \Sigmat^0(\taut) B^{\top} (\bone, \bzero,0)^{\top}\nonumber\\
&=&(\bone, b(\taut) \bp,0)\Sigmat^0(\taut)(\bone, b(\taut) \bp,0)^{\top},\label{equ:sigmat0NSW}
\end{eqnarray}
using~\eqref{one00B}.

We now assume that the above extends in the obvious way to $\epsE=0$ and calculate the resulting asymptotic variance
$\sigma_{\rm NSW}^2(\beta, \omega, \gamma)$.
Write
\begin{equation}
\label{sigma0a}
\Sigmat^0(\taut)=
\begin{bmatrix}
\Sigmat^0_{XX}(\taut) & \Sigmat^0_{XY}(\taut) & \Sigmat^0_{XZ}(\taut)\\
\Sigmat^0_{YX}(\taut) & \Sigmat^0_{YY}(\taut) & \Sigmat^0_{YZ}(\taut)\\
\Sigmat^0_{ZX}(\taut) & \Sigmat^0_{ZY}(\taut) & \Sigmat^0_{ZZ}(\taut)
\end{bmatrix}.
\end{equation}
Then
\begin{align}
\label{sigma02}
\sigma_0^2(\beta, \omega, \gamma)&=\bone \Sigmat^0_{XX}(\taut) \bone^{\top} +2b(\taut) \bp \Sigmat^0_{YX}(\taut) \bone^{\top}
+b(\taut)^2 \bp \Sigmat^0_{YY}(\taut) \bp^{\top}\nonumber\\
&\hspace*{-1cm}=\lim_{N \to \infty} N\left[\var\left(\xt^N(\taut)\right)+2b(\taut)\cov\left(\xt^N(\taut), \yEt^N(\taut)\right)
+b(\taut)^2 \var\left(\yEt^N(\taut)\right)\right],
\end{align}
where $\xt^N(\taut)$ and $\yEt^N(\taut)$ are the deterministic `number' of susceptible individuals and infectious half-edges,
given by~\eqref{finalsus} and~\eqref{yet}, respectively, but with (random) initial conditions induced by the
NSW random graph on $N$ vertices.

Recall the function $\psi$ and the random variable $\DN$, defined at~\eqref{eq:psi} and~\eqref{eq:empD}, respectively. It follows from~\eqref{finalsus} that
\begin{equation}
\label{xtNtaut}
\xt^N(\taut)=f_{\DN}\left(\psi(\taut)\right)
\end{equation}
and, from~\eqref{yet}, that
\begin{equation}
\label{yEtNtaut}
\yEt^N(\taut)=\frac{\beta+\omega+\gamma}{\beta+\omega}\mu_{\DN} \re^{-2(\beta+\omega)\taut}-\frac{\gamma}{\beta+\omega}\mu_{\DN}\re^{-(\beta+\omega)\taut}-\re^{-(\beta+\omega)\taut}f_{\DN}'\left(\psi(\taut)\right).
\end{equation}

Let $\theta \in [0,1]$.  Note, for example, that $f_{\DN}(\theta)=N^{-1}\sum_{i=1}^N \theta^{D_i}$, so
 $\var\left(f_{\DN}(\theta)\right)=N^{-1}\left[f_D(\theta^2)-f_D(\theta)^2\right]$ and
$f_{\DN}(\theta)$ is
asymptotically normally distributed by the CLT for independent and identically distributed random variables. This and similar elementary calculations show that
\begin{eqnarray}
\lim_{N \to \infty} N \var \left(f_{\DN}(\theta)\right)&=&f_D(\theta^2)-f_D(\theta)^2, \label{varfdn}\\
\lim_{N \to \infty} N \var\left(\mu_{\DN}\right)&=&\sigma^2_D (=\var(D)),\label{mudn}\\
\lim_{N \to \infty} N \var\left(f_{\DN}'(\theta)\right)&=&\theta^2f_D''(\theta^2)+f_D'(\theta^2)-f_D'(\theta)^2,\label{varfdashdn}\\
\lim_{N \to \infty} N \cov \left(\mu_{\DN}, f_{\DN}(\theta)\right)&=&\theta f_D'(\theta)-\mud f_D(\theta),\label{covmudnfdn}\\
\lim_{N \to \infty} N \cov \left(\mu_{\DN}, f_{\DN}'(\theta)\right)&=&\theta f_D''(\theta)+f_D'(\theta)-\mud f_D'(\theta),\label{covmudnfdashdn}\\
\lim_{N \to \infty} N \cov \left(f_{\DN}(\theta), f_{\DN}'(\theta)\right)&=&\theta f_D'(\theta^2)-f_D(\theta)f_D'(\theta).\label{covfdnfdashdn}
\end{eqnarray}

Recall that $z=\re^{-(\beta+\omega)\taut}$, $\psit(z)=p_{\omega}+(1-p_{\omega})z$ and
$\rho=1-f_D\left(\psit(z)\right)$ (see~\eqref{zdef} and Proposition~\ref{prop:final}(b).  Setting $\delta=0$ in~\eqref{finaltautrace} then gives (cf.~\eqref{zdef})
\begin{equation}
\label{fd1psitz}
f_D'\left(\psit(z)\right)=\left[\frac{(\beta+\omega+\gamma)z-\gamma}{\beta+\omega}\right]\mud.
\end{equation}
Then, using~\eqref{xtNtaut} and~\eqref{varfdn},
\begin{equation}
\label{varnxt}
\lim_{N \to \infty} N \var\left(\xt^N(\taut)\right)=f_D\left(\psit(z)^2\right)-(1-\rho)^2,
\end{equation}
using~\eqref{xtNtaut}, \eqref{yEtNtaut}, \eqref{covmudnfdn} and~\eqref{covfdnfdashdn}
\begin{equation}
\label{covnxtyet}
\lim_{N \to \infty} N\cov\left(\xt^N(\taut), \yEt^N(\taut)\right)=
z \psit(z)\left[\left(z+\frac{\gamma}{\beta+\omega}(z-1)\right)^2\mud-f_D'\left(\psit(z)^2\right)\right]
\end{equation}
and
\begin{align}
\label{varnyet}
\lim_{N \to \infty} N & \var\left(\yEt^N(\taut)\right)=
z^2\left[\left(z+\frac{\gamma}{\beta+\omega}(z-1)\right)^2\left(\sigma_D^2+\mud^2-2\mud\right)\right.\nonumber\\
&\left.+\psit(z)^2f_D''\left(\psit(z)^2\right)
+f_D'\left(\psit(z)^2\right)-2\left(z+\frac{\gamma}{\beta+\omega}(z-1)\right)\psit(z)f_D''\left(\psit(z)\right)\right].
\end{align}

It follows from~\eqref{xet}, \eqref{ataut}, \eqref{btaut} (all with $D_{\epsilon}$ replaced by $D$) and~\eqref{fd1psitz}, that
\begin{equation}
\label{btaut1}
b(\taut)=\frac{\beta\left[\frac{(\beta+\omega+\gamma)z-\gamma}{\beta+\omega}\right] \mud}
{z\left[\beta f_D''\left(\psit(z)\right)-(\beta+\omega+\gamma)\mud\right]},
\end{equation}
so
\begin{equation}
\label{btautzfd2}
b(\taut)z f_D''\left(\psit(z)\right)=\left[(\beta+\omega+\gamma)\left(\frac{1}{\beta+\omega}
+\frac{b(\taut)}{\beta}\right)z-\frac{\gamma}{\beta+\omega}\right]\mud\
\end{equation}
Note that $b(\taut)=\btz$, where $\btz$ is given by~\eqref{btz} with $D_{\epsilon}$ replaced by $D$.  Substituting~\eqref{varnxt}, \eqref{covnxtyet} and~\eqref{varnyet} into~\eqref{sigma02}, and invoking~\eqref{fd1psitz}
and~\eqref{btautzfd2},  yields~\eqref{sigma0} after a little algebra.

\begin{remark}[Proving Conjecture~\ref{conj:nswCLT}]
\label{rem:nswCLT}
The two remaining steps required to prove Conjecture~\ref{conj:nswCLT} are to justify (i) that~\eqref{CLT2NSW} holds when $\delta=0$ and (ii) letting $\epsE \downarrow 0$ to obtain a CLT in the event of a major outbreak; cf.~Remarks~\ref{rmk:MRcltfin1} 
and~\ref{rmk:eps=0} which discuss these steps, respectively, for an epidemic on a MR random graph.  As for epidemics on 
MR random graphs, the proofs in Ball~\cite{Ball:2018} for the SIR epidemic without dropping of edges on an NSW random graph should extend to the model with dropping of edges.
\end{remark}

\begin{remark}[Conjecture~\ref{conj:nswCLT} with $\epsE>0$]
\label{rem:nswCLTepsE}
It is possible to extend Conjecture~\ref{conj:nswCLT} to consider also the case $\epsE>0$ and obtain an analogous result
to Conjecture~\ref{conj:mrCLT}(a).  The asymptotic variance $\sigma_{\rm NSW}^2(\beta, \omega, \gamma)$ is given 
by~\eqref{sigmatNSW} and~\eqref{equ:sigmat0NSW} but now $\Sigmat^0(\taut)$ depends on how the initial infectives are chosen.  
\end{remark}

\section{Increased recovery rate instead of dropping edges}
\label{sec:relatedmodel}
%In the model with preventive dropping described in Section~\ref{sec:model}, an infectious individual has infectious contacts with each neighbour at rate $\beta$, recovers at rate $\gamma$, and susceptible neighbours drop edges to it independently at rate $\omega$. An alternative and equivalent description of the model (see also Section~\ref{sec:ED}) is that the infectious individual sends out warnings to each neighbour \emph{independently} at rate $\omega$, and susceptible individuals who receive such a warning immediately drop the corresponding edge.

Recall the equivalent formulation of the model with dropping in which an infectious individual sends out warnings to each neighbour \emph{independently} at rate $\omega$, and susceptible individuals who receive such a warning immediately drop the corresponding edge.
Consider a different but related model where, instead of sending out warnings to each neighbour at rate $\omega$ \emph{independently}, one single warning (at rate $\omega$) is used for all neighbours simultaneously (and all of them immediately drop the edges). The effect of this change is that edge droppings become \emph{dependent}. However, from the point of view of a given susceptible neighbour the probability that it drops its edge to a given infective is unchanged. Thus, for a given susceptible, such a warning (where all susceptible neighbours drop their edges) has the same effect as if its infective neighbour recovered. Hence, we consider a model without dropping, but with recovery rate $\gamma+\omega$ instead of $\gamma$. We use $(\gamma,\omega)$ and $(\gamma+\omega,0)$ to refer to the two models, where the first component refers to the recovery rate and the second component to the dropping rate.

The above reasoning suggests that the dropping model $(\gamma,\omega)$ should in some ways resemble this modified $(\gamma+\omega,0)$ model. In fact, we have seen already in Section~\ref{sec:finalsize} (Remark~\ref{rmk:binding final1}) that, as $N\to\infty$, the scaled process of susceptibles in the two epidemics converge to the same LLN limit, and the same LLN holds for the final fraction getting infected. However, the two models are \emph{stochastically} different, even for the process of susceptibles. The underlying reason for this difference is that independent warning signals makes the total number of infections \emph{less} variable compared to having one warning signal to all susceptible neighbours. Consequently, the probability of a major outbreak is \emph{greater} in the dropping model $(\gamma,\omega)$ than in the modified $(\gamma+\omega,0)$ model, as we prove in Theorem~\ref{prop:pmajor} below.
%Consequently, the number of infections made by an infective is less random in the $(\gamma,\omega)$ dropping model. 
%This \emph{increases} the probability of a major outbreak compared to the modified $(\gamma+\omega,0)$ model as we prove in Theorem~\ref{prop:pmajor} below. 
Furthermore, we expect that the decrease in variability of the number of infections made by an infective \emph{decreases} the limiting variance of both the whole process of susceptibles and the final size in the event of a major outbreak compared to the modified $(\gamma+\omega,0)$ model. This is illustrated by the numerical results in Section~\ref{sec:incRecovery}. 

Consider the beginning of an outbreak and an infectious individual having $k$ susceptible neighbours. Let $Y_k^{(\gamma,\omega)}$ be the number of these $k$ neighbours that the infectious individual infects in the dropping model and define $Y_k^{(\gamma+\omega,0)}$ similarly for the modified model. We compute the distributions of these two offspring random variables.

In the $(\gamma,\omega)$-model we first condition on the infectious period $I$, which has an $\Exp(\gamma)$ distribution, i.e.~an exponential distribution with rate $\gamma$ and hence mean $\gamma^{-1}$. Given the duration of the infectious period $I=t$, the infectious individual infects each of its $k$ susceptible neighbours independently, and a given neighbour is infected if and only if there is an infectious contact before $t$ and the edge has not been dropped before then. Thus, conditional upon $I=t$, the probability that the given neighbour is infected is
\begin{equation*}
\int_0^t \beta \re^{-(\beta+\omega)s}\,{\rm d}s = \frac{\beta}{\beta+\omega}\left(1-\re^{-(\beta+\omega)t}\right).
\end{equation*}
Given $I=t$, the number of neighbours infected follows a binomial distribution with parameters $k$ and the probability above. Hence, if we relax the conditioning, it follows that $Y_k^{(\gamma,\omega)}$ has the mixed-Binomial distribution
\begin{equation}
\label{Mbindrop}
Y_k^{(\gamma,\omega)}\sim \MBin\left( k,\ \frac{\beta}{\beta+\omega}\left(1-\re^{-(\beta+\omega)I}\right) \right), \text{ where }I\sim \Exp(\gamma).
\end{equation}
Setting $\gamma=\gamma+\omega$ and $\omega=0$ yields immediately that
\begin{equation}
\label{Mbinmod}
Y_k^{(\gamma+\omega,0)}\sim \MBin\left( k,\ 1-\re^{-\beta I^*} \right), \text{ where }I^*\sim \Exp(\gamma+\omega).
\end{equation}
It is not hard to show that
\begin{equation}\label{eq:mean}
\E\left[Y_k^{(\gamma,\omega)}\right]= \E\left[Y_k^{(\gamma+\omega,0)}\right] = k\frac\beta{\beta+\gamma+\omega},
\end{equation} 
and that $\var\left(Y_k^{(\gamma,\omega)}\right)< \var\left(Y_k^{(\gamma+\omega,0)}\right)$.

Suppose that the epidemic is initiated by a single individual, chosen uniformly at random from the entire population, becoming infective. Then the number of susceptible neighbours of the initial infective is distributed according to $D$ and, during the early stages of an outbreak in a large population, the number of susceptible neighbours of a subsequently infected individual is distributed as $\tilde{D}-1$ (see Section~\ref{sec:model}). These results hold for both models. It follows that the early stages of the dropping model in a large population can be approximated, on a generation basis, by a Galton--Watson branching process having offspring distribution that is a mixture of $Y_k^{(\gamma,\omega)}$, $k=0,1,\ldots$, with mixing probabilities $p_k,$ $k=0,1,\ldots,$ in the initial generation and mixing probabilities $\tilde{p}_k,$ $k=0,1,\ldots,$ in all subsequent generations, where $\tilde{p}_k=\mud^{-1}(k+1)p_{k+1}$. (Note that $\tilde{p}_k$, $k=0,1,\ldots$, is the probability mass function of $\tilde{D}-1$.) A similar approximation holds for the modified model, except $Y_k^{(\gamma,\omega)}$ is replaced by $Y_k^{(\gamma+\omega,0)}$. These approximations can be made rigorous in the limit as the population size $N \to \infty$ by using a coupling argument, as in e.g.~Ball and Sirl~\cite{Ball:2012}.  In the limit as $N \to \infty$, the probability of a major outbreak in the epidemic model is given by the probability that the corresponding approximating branching process does not go extinct.

The following lemma, proved in Appendix~\ref{app:PGFcomparison}, is required for the proof of Theorem~\ref{prop:pmajor} below, which shows that the probability of a major outbreak is greater in the dropping model than in the corresponding modified model. First, some more notation is required.  For $k=1,2,\ldots$ let $f_k^{(\gamma,\omega)}(s)=\E\left[s^{Y_k^{(\gamma,\omega)}}\right]$, $s \in \mathbb{R}$, denote the probability-generating function (PGF) of $Y_k^{(\gamma,\omega)}$, the number of neighbours that an infectious individual with $k$ susceptible neighbours infects in the early stages of the $(\gamma,\omega)$ dropping model, and define $f_k^{(\gamma+\omega)}(s)$ similarly for the $(\gamma+\omega,0)$ modified model. Let $f_0^{(\gamma,\omega)}(s)=f_0^{(\gamma+\omega)}(s)=1$  $(s \in \mathbb{R})$. Then, for the dropping model, the approximating branching process has offspring PGF $f^{(\gamma,\omega)}(s)= \sum_{k=0}^{\infty} p_k f_k^{(\gamma,\omega)}(s)$ in the first generation and offspring PGF $\tilde{f}^{(\gamma,\omega)}(s)= \sum_{k=0}^{\infty} \tilde{p}_k f_k^{(\gamma,\omega)}(s)$ in all subsequent generations, with analogous results holding for the $(\gamma+\omega,0)$-model.

\begin{lem}
\label{lemma:PGFcomparison}
Suppose that $\beta>0$ and $\gamma>0$.  Then, for $k=0,1,\ldots$,
\begin{equation}
\label{fkcomp}
f_k^{(\gamma,\omega)}(s) \le f_k^{(\gamma+\omega)}(s) \qquad (0 \le s \le 1),
\end{equation}
with strict inequality for all $s \in [0,1)$ when $k \ge 2$.
\end{lem}

\begin{thm}[Probability of a major outbreak]\label{prop:pmajor}
\mbox{}
\begin{itemize}
\item[(a)]
The basic reproduction number $R_0$ for both the dropping and modified models is given by~\eqref{R_0}.
\item[(b)]
Suppose that $R_0>1$ and the epidemic is initiated by a single infective individual, chosen uniformly at random from the population. Then the probability of a major outbreak $p_{\rm maj}^{(\gamma,\omega)}$ for the $(\gamma,\omega)$ dropping model is strictly greater than the probability of a major outbreak $p_{\rm maj}^{(\gamma+\omega,0)}$ for the modified $(\gamma+\omega,0)$-model, i.e.
\begin{equation}\label{eq:pmajor}
p_{\rm maj}^{(\gamma,\omega)}> p_{\rm maj}^{(\gamma+\omega,0)}.
\end{equation}
\end{itemize}
%Moreover, the inequality is strict if $\Pp(D \ge 3)>0$, $\beta>0$ and $\gamma>0$.
\end{thm}

\begin{proof}
The basic reproduction number is given by the offspring mean of a typical (i.e.~non-initial generation) infective, so for both models, using~\eqref{eq:mean},
\[
R_0=\sum_{i=1}^{\infty}\tilde{p}_k k\frac\beta{\beta+\gamma+\omega}=\frac\beta{\beta+\gamma+\omega}\left(\mu_D+\frac{\sigma_D^2}{\mu_D}-1\right),
\]
which proves part (a).

Turning to part (b), suppose that $R_0>1$. Then, using standard branching process theory gives that, for the dropping model, the probability of a major outbreak is given by $p_{\rm maj}^{(\gamma,\omega)}=1-f^{(\gamma,\omega)}(\sigma^{(\gamma,\omega)})$, where $\sigma^{(\gamma,\omega)}$ is the unique solution in $[0,1)$ of $\tilde{f}^{(\gamma,\omega)}(s)=s$; cf.\ Kenah and Robins~\cite{Kenah:2007} and Ball and Sirl~\cite{Ball:2013}. Analogously, for the modified model, $p_{\rm maj}^{(\gamma+\omega,0)}=1-f^{(\gamma+\omega,0)}(\sigma^{(\gamma+\omega,0)})$, where $\sigma^{(\gamma+\omega,0)}$ is the unique solution in $[0,1)$ of $\tilde{f}^{(\gamma+\omega,0)}(s)=s$.

Note that if $\Pp(D \ge 3)=0$ then $R_0 \le 1$, so $R_0>1$ implies that $\Pp(D \ge 3)>0$.  It then follows
immediately from Lemma~\ref{lemma:PGFcomparison} that $f^{(\gamma,\omega)}(s)< f^{(\gamma+\omega,0)}(s)$ and $\tilde{f}^{(\gamma,\omega)}(s)< \tilde{f}^{(\gamma+\omega,0)}(s)$ for all $s \in [0,1)$. Hence, since $\tilde{f}^{(\gamma,\omega)}(1)=\tilde{f}^{(\gamma+\omega,0)}(1)=1$ and the derivative of both
$\tilde{f}^{(\gamma,\omega)}$ and $\tilde{f}^{(\gamma+\omega,0)}$ at $s=1$ is $R_0>1$ , it follows that $\sigma^{(\gamma,\omega)}<\sigma^{(\gamma+\omega,0)}$, whence $f^{(\gamma,\omega)}(\sigma^{(\gamma,\omega)})< f^{(\gamma+\omega,0)}(\sigma^{(\gamma,\omega)})< f^{(\gamma+\omega,0)}(\sigma^{(\gamma+\omega,0)})$, as $f^{(\gamma+\omega,0)}$ is strictly increasing on $[0, 1]$. Thus we obtain our statement~\eqref{eq:pmajor}.   $\Box$
\end{proof}
\begin{remark}[Other choices for initial infectives]
Theorem~\ref{prop:pmajor} is easily extended to other assumptions concerning initial infectives; e.g.~to an epidemic initiated by $k>1$ infective individuals chosen uniformly at random from the population, or to an epidemic initiated by an infective of a specified degree.
\end{remark}

\section{No dropping of edges}
\label{sec:nodropping}
We use the results from this paper to analyse the Markovian SIR epidemic on a configuration model network in Section~\ref{sec:MarkovSIR} and the giant component of a configuration model network in Section~\ref{sec:configuration}. Note that in the case that there is no dropping of edges, i.e.\ $\omega=0$, we are in the setting of a Markovian SIR epidemic on a configuration model network. We treat the asymptotic variance of the final size for this model in Conjecture~\ref{conj:nodroppingCLT}. If additionally, there is no recovery, i.e.\ $\omega=0=\gamma$, then in the event of a major outbreak, all individuals in the giant component eventually get infected. By using this we can apply the results from this paper to make statements about the size of the giant component in configuration model random graphs, see Conjecture~\ref{prop:giantCLT}.  

\subsection{SIR epidemic on configuration network}
\label{sec:MarkovSIR}
When $\omega=0$, the model reduces to the Markov SIR epidemic on a
configuration model network.  The formulae for the asymptotic variance of the final size for the epidemic on MR and NSW random networks simplify and become fully explicit given $z$, defined below.

%Note that $\psit(z)=z$ when $\omega=0$.  
Recall that $\epsE=\sum_{i=1}^{\infty} i\epsilon_i$.  If $\epsE>0$,
then setting $\omega=0$ in Proposition~\ref{prop:final}(a) shows that $\rho=1-\fde(z)$, where
$z$ is the unique solution in $[0,1)$ of
\begin{equation}
\label{zdefND}
(\beta+\gamma)z-\gamma=\beta \mud^{-1}\fde'(z).
\end{equation}
If $\epsE=0$, so the epidemic is started by a trace of infection, and $R_0>1$ then, using Proposition~\ref{prop:final}(b), $\rho=1-\fd(z)$, where
$z$ is the unique solution in $[0,1)$ of~\eqref{zdefND} with $\fde'$ replaced by $\fd'$.

Let $\TnMRND$ and $\TnNSWND$ denote the final size of the epidemic, with no dropping of edges, on an MR and NSW configuration model random network, respectively, each having $N$ vertices.  Let $\SigmaMRE(\beta,\gamma)=\SigmaMR(\beta, 0, \gamma)$ and $\SigmaNSWE(\beta,\gamma)=\SigmaNSW(\beta, 0, \gamma)$ denote the asymptotic variance of the final size for the epidemic on an MR and an NSW configuration model random network, respectively. The following conjecture gives fully explicit formulae for $\SigmaMRE(\beta,\gamma)$ and $\SigmaNSWE(\beta,\gamma)$ as functions of $z$, which are
derived in Appendix~\ref{app:nodroppingCLT}.

\begin{conj}[CLT for final size of epidemic on configuration model networks]
\label{conj:nodroppingCLT}
\mbox{}
\begin{itemize}
\item[(a)] For the SIR epidemic on an MR random network,
\begin{itemize}
\item[(i)]
if $\epsE>0, \dmax < \infty$ and $z>0$, then,
\begin{equation}
\label{CLTMREND}
\sqrt{N}\left(N^{-1}\TnMRND - \rho \right) \convD N(0,\SigmaMRE(\beta,\gamma)) \quad\mbox{as } N \to \infty,
\end{equation}
where
\begin{align}
\label{MRCLTvarND}
\SigmaMRE(\beta,\gamma)&=1-\rho-\fde(z^2)-h(\beta,\gamma,z)^2\left[\fde'(z^2)+z^2\fde''(z^2)\right]\nonumber\\
&\phantom{=\ }+h(\beta,\gamma,z)^2\left[\left(\frac{\gamma}{2\beta+\gamma}\right)(\sigma_D^2+\mud^2)+2\left(\frac{\gamma-(\beta+\gamma)z}{\beta}\right)^2\mud\right]\nonumber\\
&\phantom{=\ }+2h(\beta,\gamma,z)\left[z\fde'(z^2)+\left(\frac{\gamma-(\beta+\gamma)z}{\beta}\right)
\left(\frac{\beta+\gamma}{2\beta+\gamma}\right)\mud\right],
\end{align}
with
\begin{equation}
\label{hdef}
h(\beta,\gamma,z)=\frac{\gamma-(\beta+\gamma)z}{\beta+\gamma-\beta\mud^{-1}\fde''(z)};
\end{equation}
\item[(ii)] if  $\epsE=0$, $\dmax<\infty$, $R_0>1$ and $z>0$, then, in the event of a major outbreak, \eqref{CLTMREND} holds with $\epsilon=0$ and $D_{\epsilon}$ replaced by $D$ in~\eqref{MRCLTvarND} and~\eqref{hdef}.
\end{itemize}
\item[(b)]
For the epidemic on an NSW network, suppose that $\epsE=0$, $\dmax<\infty$, $R_0>1$ and $z>0$.  Then, in the event of a major outbreak,
\begin{equation}
\label{CLTNSWEND}
\sqrt{N}\left(N^{-1}\TnNSWND - \rho \right) \convD N(0,\SigmaNSWE(\beta,\gamma)) \quad\mbox{as } N \to \infty,
\end{equation}
where
\begin{align}
\label{NSWCLTvarND}
\SigmaNSWE(\beta,\gamma)=&\rho(1-\rho)+2h(\beta,\gamma,z)\left(\frac{\gamma-(\beta+\gamma)z}{\beta}\right)\left(\frac{\beta+\gamma}{2\beta+\gamma}\right)\mud\nonumber\\
&\quad+h(\beta,\gamma,z)^2\left[\frac{\gamma}{2\beta+\gamma}
+\left(\frac{\gamma-(\beta+\gamma)z}{\beta}\right)^2\right](\sigma_D^2+\mud^2)\nonumber\\
&\quad+2h(\beta,\gamma,z)^2\frac{(\beta+\gamma)[\gamma-(\beta+\gamma)z]}{\beta^2}z\mud,
\end{align}
and $h(\beta,\gamma,z)$ is given by~\eqref{hdef}, with $D_{\epsilon}$ replaced by $D$.
\end{itemize}
\end{conj}

\begin{remark}[Proof of Conjecture~\ref{conj:nodroppingCLT}]
\label{rmk:nodroppingCLTa}
Although only conjectured here, Conjecture~\ref{conj:nodroppingCLT} (and hence also Conjecture~\ref{prop:giantCLT} below)
follow as a special case of Ball~\cite{Ball:2018}, Theorems 2.1 and 2.2.   
\end{remark}

\begin{remark}[Epidemics on NSW random network with $\epsE>0$]
As for the model with dropping, Conjecture~\ref{conj:nodroppingCLT}(b) can be extended to include the case $\epsE>0$;
the asymptotic variance $\SigmaNSWE(\beta,\gamma)$ then depends on how the initial infectives are chosen
(cf.~Remark~\ref{rem:nswCLTepsE}).   
\end{remark}
\subsection{Configuration model giant component}
\label{sec:configuration}
If $\omega=\gamma=0$ then the epidemic ultimately infects all individuals in all components
of the random network that contain at least one initial infective.  Thus, under suitable conditions,
in the limit as $\epsilon \downarrow 0 $, setting $\gamma=0$ in Conjecture~\ref{conj:nodroppingCLT} (a)(ii) and (b) leads to CLTs for the size of the largest connected (i.e.~giant) component in MR and NSW configuration model random graphs, respectively.

Let $\kappa=\E[D(D-2)]=\sigma_D^2+\mud^2-2\mud$ and note that, setting $\omega=\gamma=0$ in the formula for $R_0$, $\kappa >0$ if and only if $R_0>1$.  The above configuration model random graphs possess a giant component if and only if $\kappa>0$, see e.g.~Durrett~\cite{Durrett:2007}, Theorem 3.1.3.  Suppose that $\kappa>0$. Setting $\gamma=0$ and $D_{\epsilon}=D$ in~\eqref{zdefND} shows that $z$ is now given by the unique solution in $[0,1)$ of
\begin{equation}
\label{MRz}
\mud z=f_D'(z).
\end{equation}
and the asymptotic fraction of vertices in the giant components of the above configuration model random graphs is given by $\rho=1-f_D(z)$.

Let $\RnMR$ and $\RnNSW$ denote respectively the size of the giant component in an MR and an NSW random graph on $N$ vertices.  Setting $\gamma=0$ in Conjecture~\ref{conj:nodroppingCLT} (a)(ii) and (b) yields the following conjecture.

\begin{conj}[CLT for the size of the giant component]
\label{prop:giantCLT}
\mbox{}\\Suppose that $\kappa>0$, $\dmax<\infty$ and $p_1>0$.  Then, 
\begin{description}
\item(a)\quad  for an MR random graph,
\begin{equation}
\label{CLTMR}
\sqrt{N}\left(N^{-1}\RnMR -\rho\right) \convD N(0,\SigmaMRGC)\quad\mbox{as } N \to \infty,
\end{equation}
where
\begin{align}
\label{MRGCvar}
\SigmaMRGC=&1-\rho-\fd(z^2)-\frac{z^2}{\left[1-\mu_D^{-1}f_D''(z)\right]}\left[2\fd'(z^2)-\mud\right]\nonumber\\
&\quad-\frac{z^2}{\left[1-\mu_D^{-1}f_D''(z)\right]^2}\left[\fd'(z^2)+z^2\fd''(z^2)-2\mud z^2\right];
\end{align}
\item(b)\quad  for an NSW random graph,
\begin{equation}
\label{CLTNSW}
\sqrt{N}\left(N^{-1}\RnNSW -\rho\right) \convD N(0,\SigmaNSWGC)\quad\mbox{as } N \to \infty,
\end{equation}
where
\begin{align}
\label{NSWGCvar}
\SigmaNSWGC=&\rho(1-\rho)+\frac{z^2}{\left[1-\mu_D^{-1}f_D''(z)\right]}\mud\nonumber\\
&\quad+\frac{z^4}{\left[1-\mu_D^{-1}f_D''(z)\right]^2}\left(\sigma_D^2+\mud^2-2\mud\right).
\end{align}
\end{description}
\end{conj}

It is easily verified that the expressions~\eqref{MRGCvar} and~\eqref{NSWGCvar} for the asymptotic variances $\SigmaMR$ and $\SigmaNSW$ coincide with those first obtained by Ball and Neal~\cite{BallNeal:2017} using a completely different method; a CLT was conjectured in that paper and subsequently proved for an MR random graph in Barbour and R{\"o}llin~\cite{Barbour:2017}.  The
results proved in these two papers allow for unbounded degrees under suitable conditions. 
%\begin{remark}
%Although the results in this section are stated as conjectures, they all follow as special cases of theorems proved in Ball~\cite{Ball:2018}.   
%\end{remark}
\section{Numerical examples}
\label{sec:numerical}

In this section we give numerical results which exemplify some of the limit theorems  and support some of the conjectures presented in the paper and give examples of using those limiting results for approximation. Such approximations follow from our asymptotic results in exactly the same way as the approximate distribution of the sum of independent and identically distributed random variables follows from the classical CLT. For example, we can use equation~\eqref{FCLT} in Theorem~\ref{thm:MRclt} to say that, for large $N$, the distribution of $\bWN(t)$ is approximately that of
$N \bw(t) + \sqrt{N}\bV(t)$,
%\begin{equation*}
%\bWN(t) \approx N \bw(t) + \sqrt{N}\bV(t),
%\end{equation*}
from which approximations for the mean and variance of $\bWN(t)$ follow immediately from the corresponding properties of the Gaussian process $\bV(t)$. We also explore numerically some aspects of the behaviour of the model we have analysed, using the asymptotic results we have derived. In our numerical examples relating to the temporal evolution of the epidemic we look only at the mean and variance of the number of infective individuals in the population, we do not investigate any other quantities of interest or explicitly investigate the covariance/correlation structure in any way.

In this section we use the notation $D\sim\mbox{Poi}(\lambda)$ or $D\sim\mbox{Geo}(p)$ to denote that the network degree distribution follows a standard Poisson or Geometric distribution with mass functions $p_k=\re^{-\lambda}\lambda^k/k!$ ($k=0,1,\dots$) or $p_k=p(1-p)^k$ ($k=0,1,\dots$), respectively. In particular we shall use repeatedly in our examples the distributions $D\sim\mbox{Poi}(5)$ and $D\sim\mbox{Geo}(1/6)$. These distributions both have mean 5 and their standard deviations are $\sqrt{5}\approx 2.2$ and $\sqrt{30}\approx5.5$ respectively.

First, however, we discuss some of the issues that arise in relation to the numerical implementation of our analytical results.

\subsection{Implementation}
\label{sec:implementation}

The numerical implementation of our asymptotic results concerning epidemic final size (the formulae laid out in Propositions~\ref{prop:final} and~\ref{prop:mrVar} and Conjecture~\ref{conj:nswCLT}) is straightforward, involving root-finding, numerical integration and derivatives up to order 3 of the degree distribution PGF $f_D$. For the degree distributions we use, we have $f_D^{(i)}(s) = \lambda^i \re^{-\lambda(1-s)}$ when $D\sim\mbox{Poi}(\lambda)$ and  $f_D^{(i)}(s) = \frac{i! p (1-p)^i}{(1-(1-p)s)^{i+1}}$ when $D\sim\mbox{Geo}(p)$, with both formulae being valid for $i=0,1,\dots$.
In the final size examples we always use the version of these results with $\epsilon_E=0$, i.e.\ we work under the asymptotic regime where the epidemic starts with a trace of infection. The results concerning the evolution of the epidemic through time (Theorems~\ref{thm:as},~\ref{thm:MRclt} and~\ref{thm:NSWtemporalCLT}) warrant discussion of some issues that arise.

An obvious first issue is initial conditions $(\bx(0),\by(0),z_E(0))$ and $\Sigma(0)$ for the system of ODEs given by~\eqref{diff1}-\eqref{diff3} together with the variance/covariance-related matrix ODE~\eqref{diffSigma} (see also Remark~\ref{rem:NSWSigma0}). In an MR network we take the initial infectives to comprise a fixed number of individuals, with numbers of individuals of the various degrees chosen in the same proportions as they are present in the whole population. In an NSW network we choose the required number of initial infectives uniformly at random from the population.
%In either case this amounts to assuming that the initial cases are infected from outside the network.
Ideally we might want the initial conditions to represent a large outbreak initiated by few initial infectives; this is a rather more complex situation and could be addressed using the results of Ball and House~\cite{BalHou:2017}.

Let $\epsilon$ be the proportion of individuals initially infected in the limit as $N\to\infty$. It is straightforward to show that $x_i(0) = \lim_{N\to\infty} N^{-1} E[X_i^N(0)] = p_i (1-\epsilon)$ and similarly that $y_i(0) = p_i \epsilon$ and $z_E(0) = 0$ (cf.\ the paragraph immediately before Theorem~\ref{thm:as}; with a NSW network these limits hold almost surely). Turning to $\Sigma(0)$, in the case of an MR network we have chosen the initial conditions so that there is no variability; i.e. all elements of $\Sigma_{\rm MR}(0)$ are zero. With an NSW network there is variability in the initial conditions; to characterise it we let $i_0^N = [\epsilon N]$ be the number of initially infected individuals (or assume that $i_0^N$ is a function of $N$ such that $\lim_{N\to\infty} N^{-1}i_0^N = \epsilon$) and use the notation $\sigma_{x_i,x_j}(0)$ for the $(i,j)$-th element of the submatrix of $\Sigma_{\rm NSW}(0)$ corresponding to the susceptible elements (cf.\ the partitioning in~\eqref{sigma0a}), so for example $\sigma_{x_i,y_j}(0) = \lim_{N\to\infty} N^{-1} \cov(X_i^N(0),Y_j^N(0))$. We find that the following elements of $\Sigma_{\rm NSW}(0)$ are non-zero: for all $i$, $\sigma_{x_i,x_i}(0) = p_i (1-p_i)(1-\epsilon)$ and $\sigma_{y_i,y_i}(0) = p_i (1-p_i) \epsilon$; and for all $i\neq j$, $\sigma_{x_i,x_j}(0) = -p_ip_j(1-\epsilon)$ and $\sigma_{y_i,y_j}(0) = -p_ip_j\epsilon$. Derivations can be found in Appendix~\ref{app:ODEinitConds}.

After solving the ODE systems numerically we can calculate the asymptotic means and variances for other quantities of interest, for example to approximate $I^N(t)$, the number of infected individuals at time $t$, we use
\begin{equation*}
\lim_{N\to\infty} N^{-1}E[I^N(t)] = \sum_{i=0}^\infty y_i(t) \quad \mbox{and} \quad \lim_{N\to\infty} N^{-1/2} \var[I^N(t)] = \sum_{i=0}^{\infty} \sum_{j=0}^{\infty} \sigma_{y_i,y_j}(t).
\end{equation*}
The final ODE-related issue is choosing the value of $M$, the maximum degree, to use when the degree distribution does not have finite support.  (This amounts to setting $x_i(t)=y_i(t)=0$ for all $t\geq0$ and $i=M+1,M+2,\dots$.) The upper bound $M$ needs to be large enough that the approximation is accurate but not so large that the systems of ODEs are impractical to solve numerically (the number of ODEs grows like $M^2$). To decide on an appropriate value for $M$ we compare plots of the asymptotic means and variances of $I(t)$ (i.e.\ the solid lines in the lower plots of Figure~\ref{fig:processApprox}), increasing $M$ until there is no observable difference in these plots. We also compare the predicted relative `final' size $x(0)-x(t_{\rm end})$ from the numerical ODE solution to the asymptotic final size predicted by Proposition~\ref{prop:final}. For the degree distributions we find that $M=15$ is sufficient when $D\sim\mbox{Poi}(5)$ and $M=50$ when $D\sim\mbox{Geo}(1/6)$.

Simulation of the epidemic process is relatively straightforward. Given a sequence of degrees (either [MR] a specified sequence or [NSW] independent realisations from the distribution $\{p_k\}$) we (i) generate the network, (ii) choose initial infectives, (iii) spread the epidemic on the network. There is therefore randomness in each simulation deriving not just from the evolution of the epidemic, but also the graph construction and, in the case of an NSW graph, the degree sequence. When we calculate confidence intervals (CIs) for quantities associated with simulations of the temporal evolution of the epidemic they are calculated independently for each time point; i.e. they are not confidence bands for the process. Endpoints of CIs for standard deviations are calculated as the square roots of the endpoints of standard symmetric (in terms of probability) CIs for the variance.

\subsection{Convergence and approximation of temporal properties}
\label{sec:cgceAppTime}
First we demonstrate numerically some of the limit theorems from earlier sections, showing both how the convergence is realised and thus how these limit theorems can be used for approximation. We give examples only with an NSW graph construction, but much the same observations apply in the MR graph scenario.

In Figure~\ref{fig:processApprox} we demonstrate using Theorem~\ref{thm:NSWtemporalCLT} for approximation of the temporal evolution of the epidemic, comparing simulated trajectories of the prevalence $I^N(t)$ (for $N=1000$) versus time $t$ of the model with predictions from the functional central limit theorem, for a Poisson and a Geometric degree distribution. The upper plots show the simulated trajectories together with the mean and a central 95\% probability band predicted by the CLT; they suggest that the approximation is fairly good. The lower plots compare the mean and variance of the prevalence through time with the LLN and CLT based asymptotic predictions.

In Figure~\ref{fig:processCgce} we investigate the convergence of the distribution of $I^N(t)$ to its $N\to\infty$ limit at three time points $t_1$, $t_2$ and $t_3$. The times are chosen so that $t_2$ is close to the time of peak prevalence and $t_1$ and $t_3$ are when prevalence is increasing and decreasing, respectively, at a level roughly half that of the peak prevalence. (Effectively we are examining the upper-right plot of Figure~\ref{fig:processApprox} in detail at these three time points.) In this figure we have used a geometric degree distribution, but very similar conclusions are obtained using different distributions. This convergence is further investigated/demonstrated in Figure~\ref{fig:KolDist}, where, separately for each of the same three time points, we plot the Kolmogorov distance between the empirical and asymptotic distributions of the number of infectives against population size $N$.

\begin{figure}
\begin{center}
\begin{tabular}{cc}
(a) $D$ Poisson & (b) $D$ Geometric \\
\includegraphics[width=\hfigwidth]{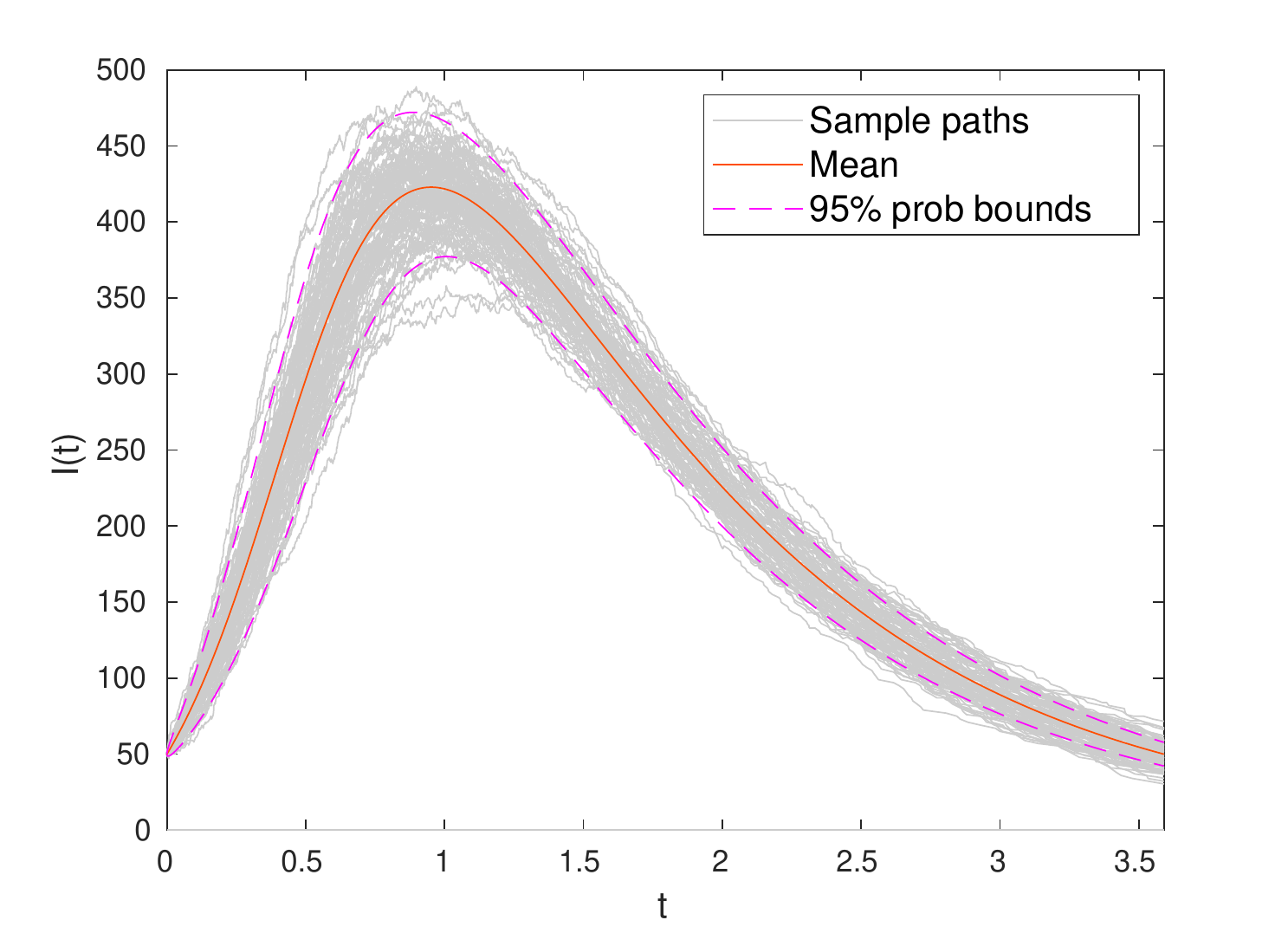} & \includegraphics[width=\hfigwidth]{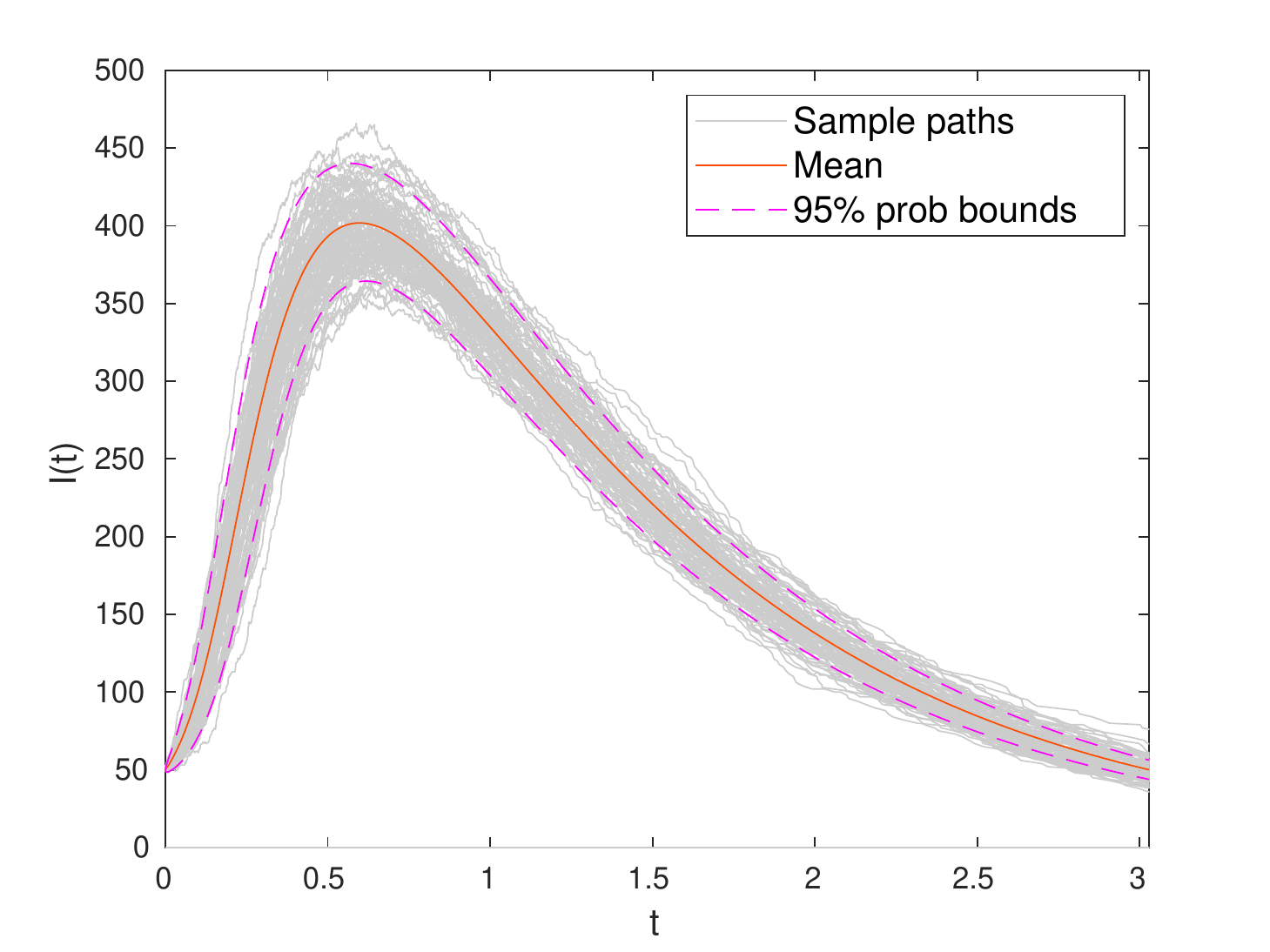}  \\
\includegraphics[width=\hfigwidth]{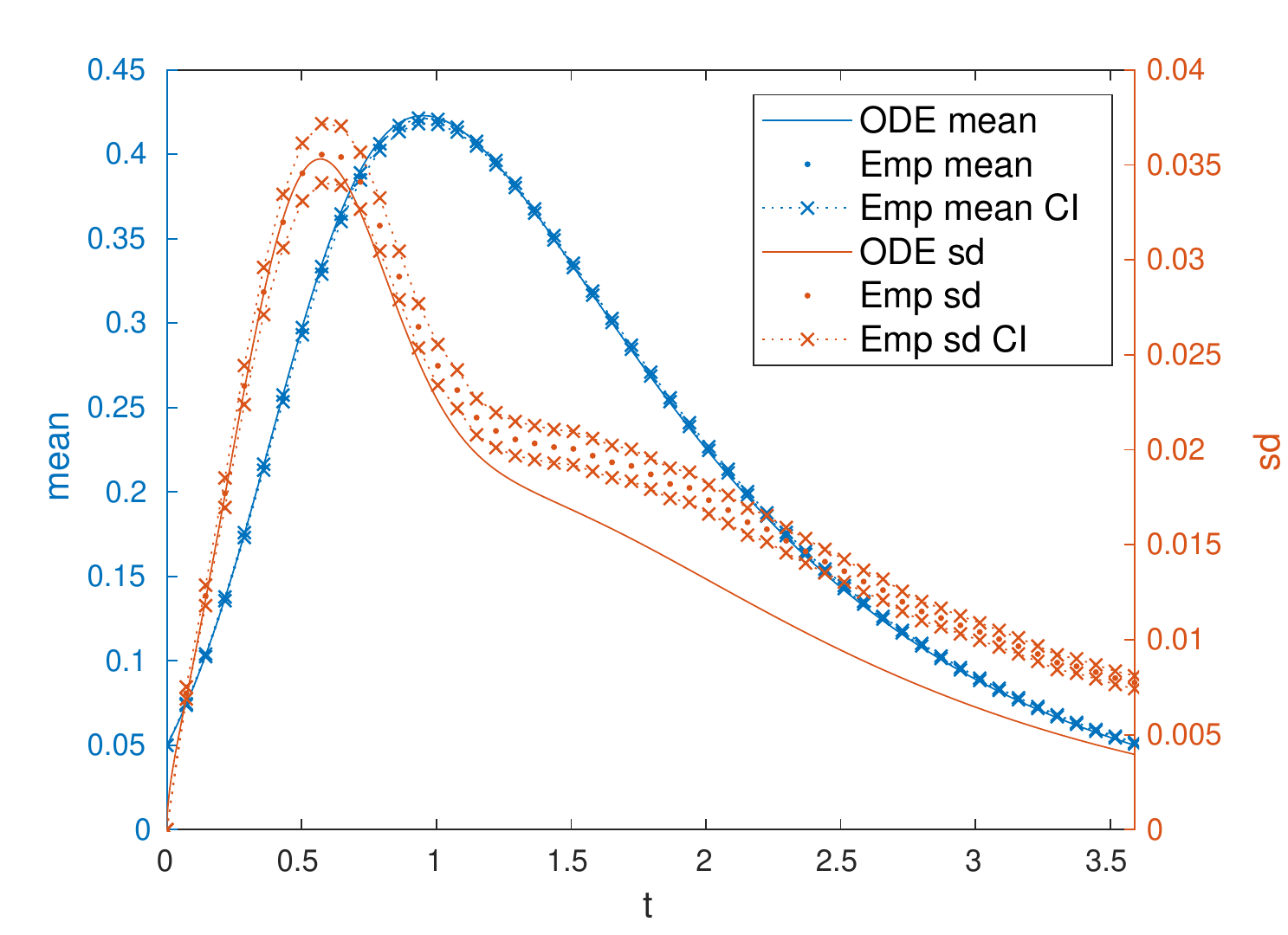}& \includegraphics[width=\hfigwidth]{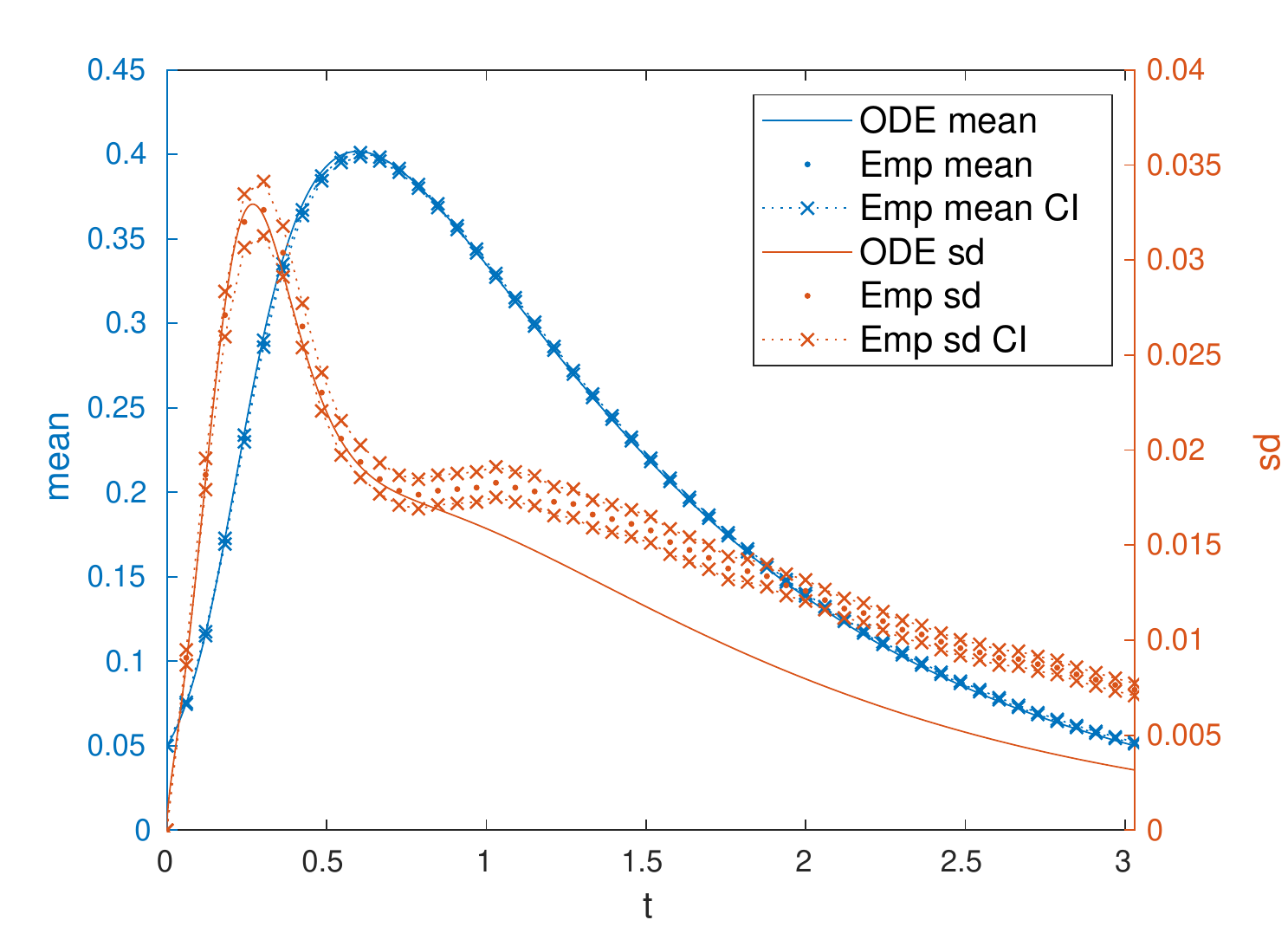}
\end{tabular}
\end{center}
\caption{Demonstration of approximation implied by Theorem~\ref{thm:NSWtemporalCLT}, for (a) $D\sim\mbox{Poi}(5)$ and (b) $D\sim\mbox{Geo}(1/6)$. The upper plots show 100 simulated sample paths of the number of infectives $I^N(t)$ against time $t$, together with the mean and central 95\% probability bands for $I^N(t)$ predicted by the functional CLT. The lower plots show asymptotic values and estimates from 1000 simulations of the scaled mean and standard deviation of the number of infectives through time.
Other parameters are $N=1000$, $\beta=3/2$, $\gamma=1$, $\omega=1$, $i^N_0=0.05N$. (All plots are truncated at the time when the proportion of individuals that are infective drops below 0.05.)}
\label{fig:processApprox}
\end{figure}

\begin{figure}
\begin{center}
\begin{tabular}{rccc}
 & $t_1=0.35$ & $t_2=0.6$ & $t_3=1.5$ \\
\begin{sideways}\hspace*{1cm}$N=200$\end{sideways}
 & \includegraphics[width=\tfigwidth]{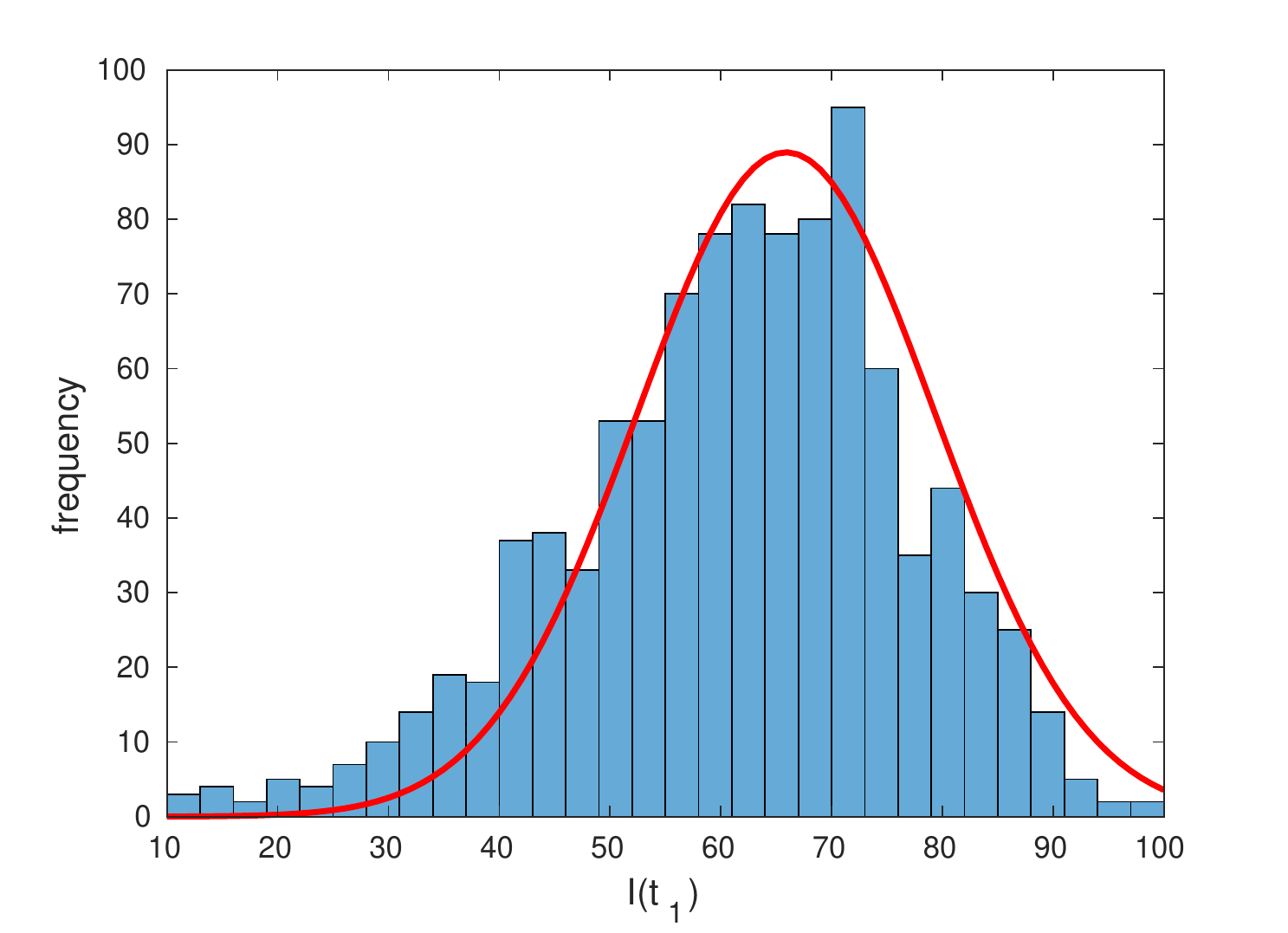}
 & \includegraphics[width=\tfigwidth]{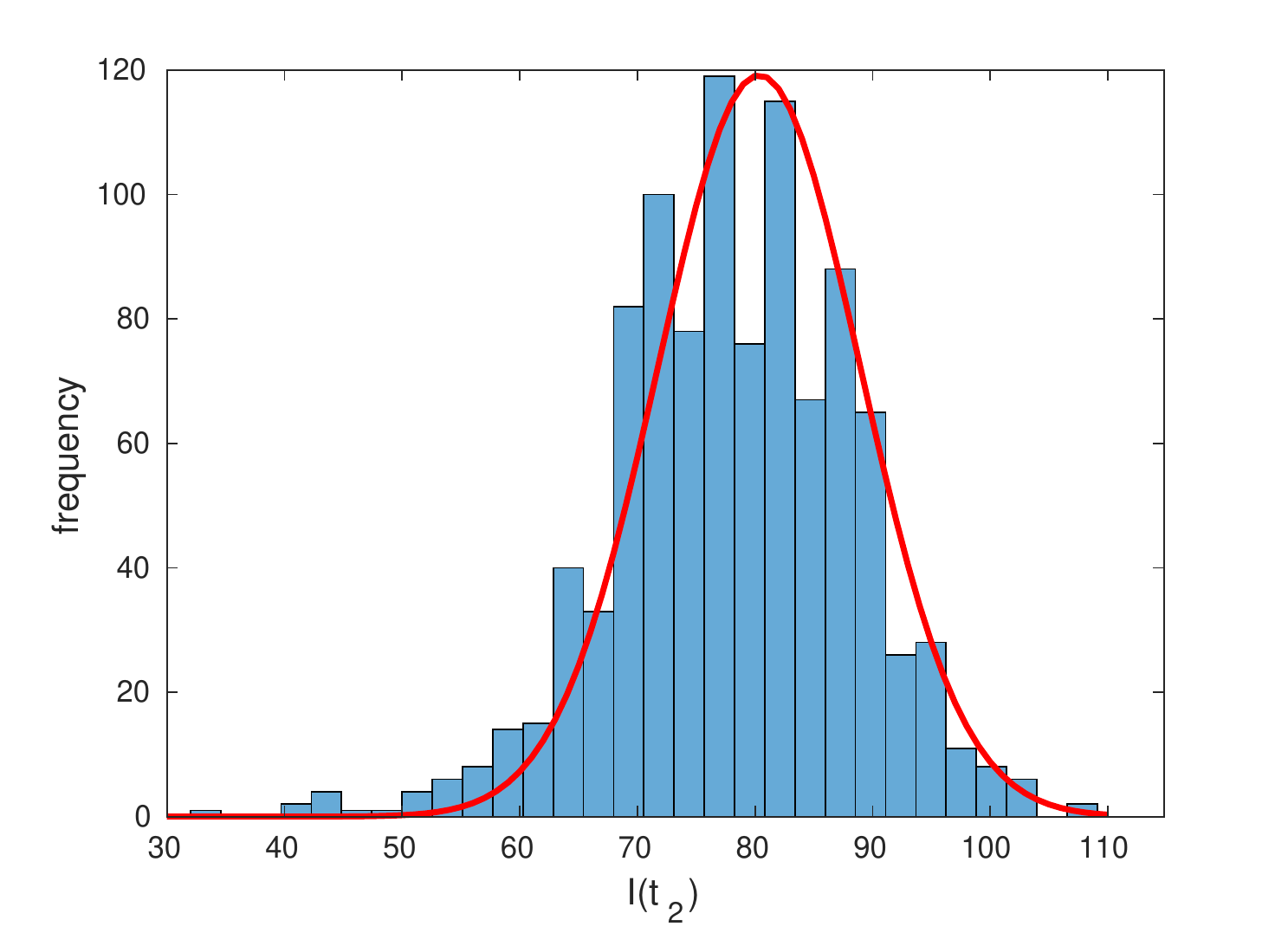}
 & \includegraphics[width=\tfigwidth]{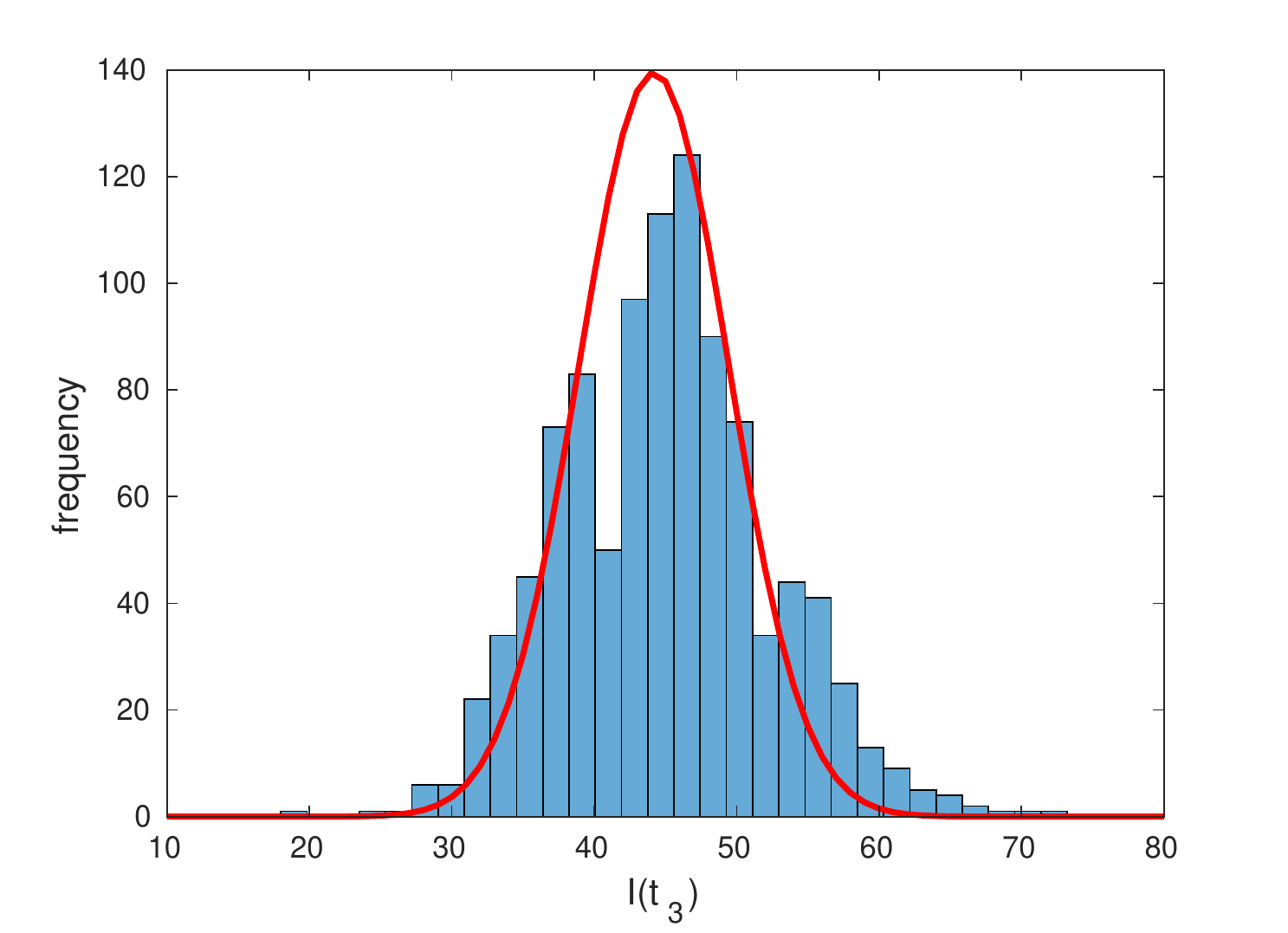} \\
\begin{sideways}\hspace*{1cm}$N=1000$\end{sideways}
 & \includegraphics[width=\tfigwidth]{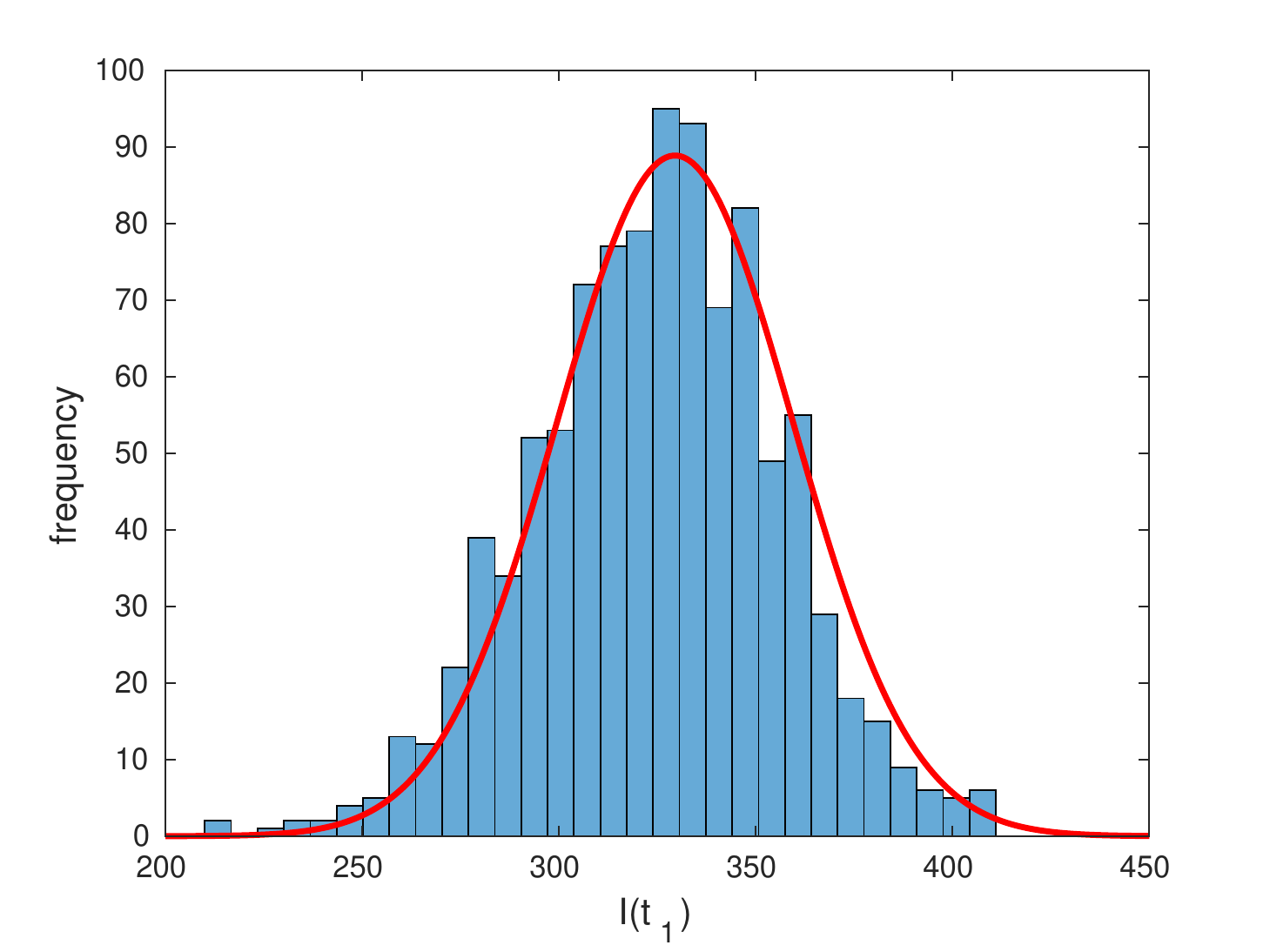}
 & \includegraphics[width=\tfigwidth]{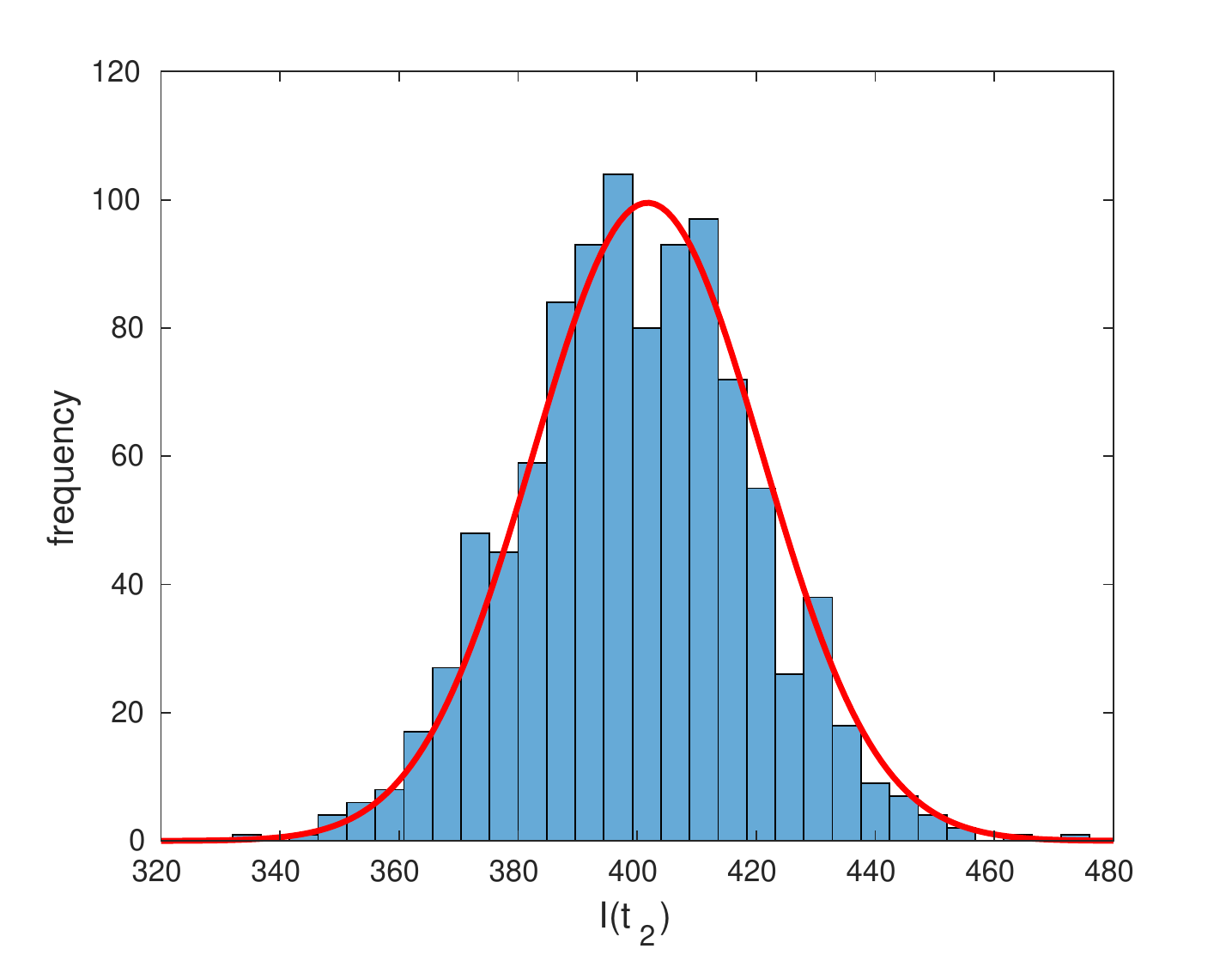}
 & \includegraphics[width=\tfigwidth]{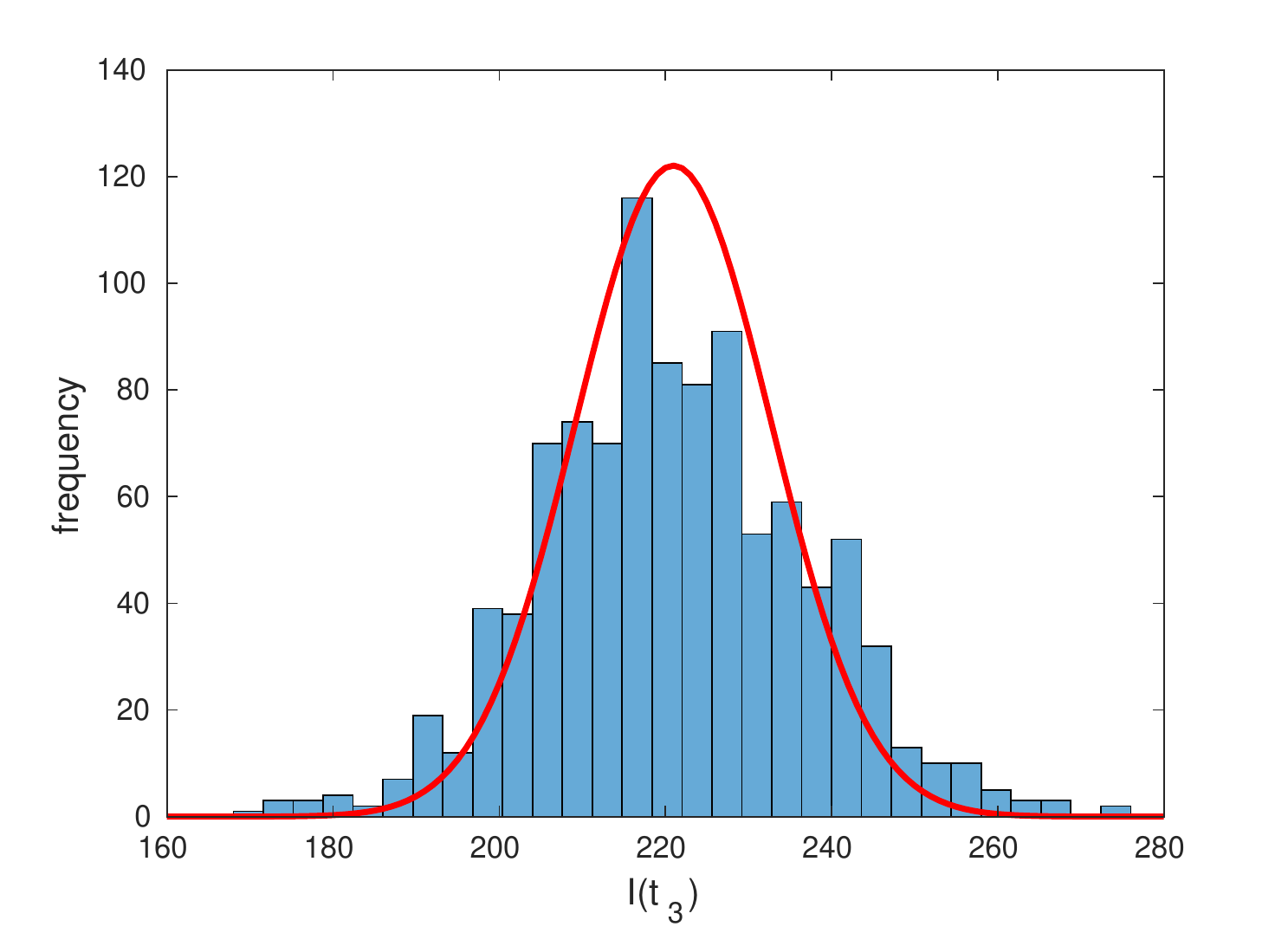} \\
\begin{sideways}\hspace*{1cm}$N=5000$\end{sideways}
 & \includegraphics[width=\tfigwidth]{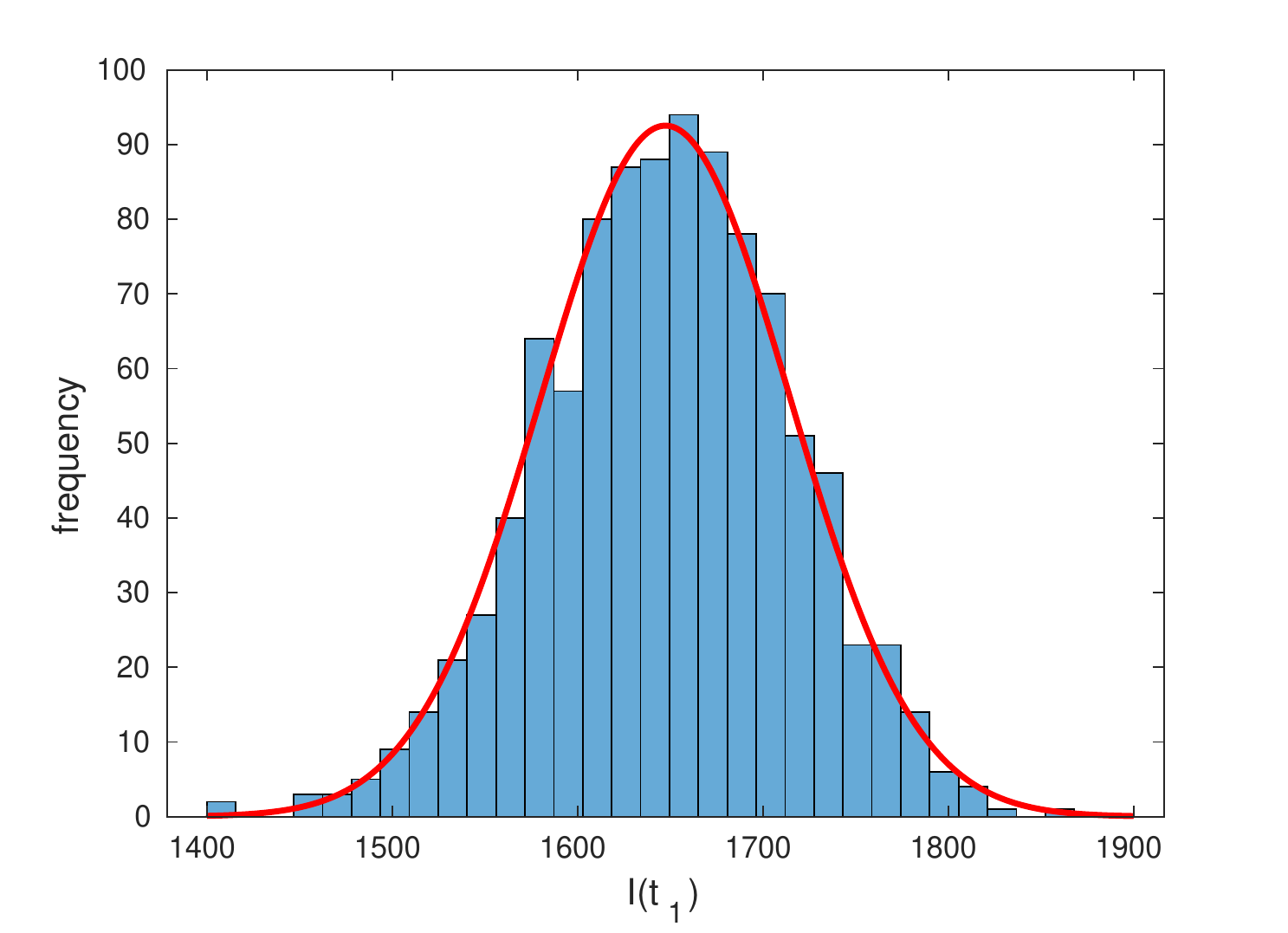}
 & \includegraphics[width=\tfigwidth]{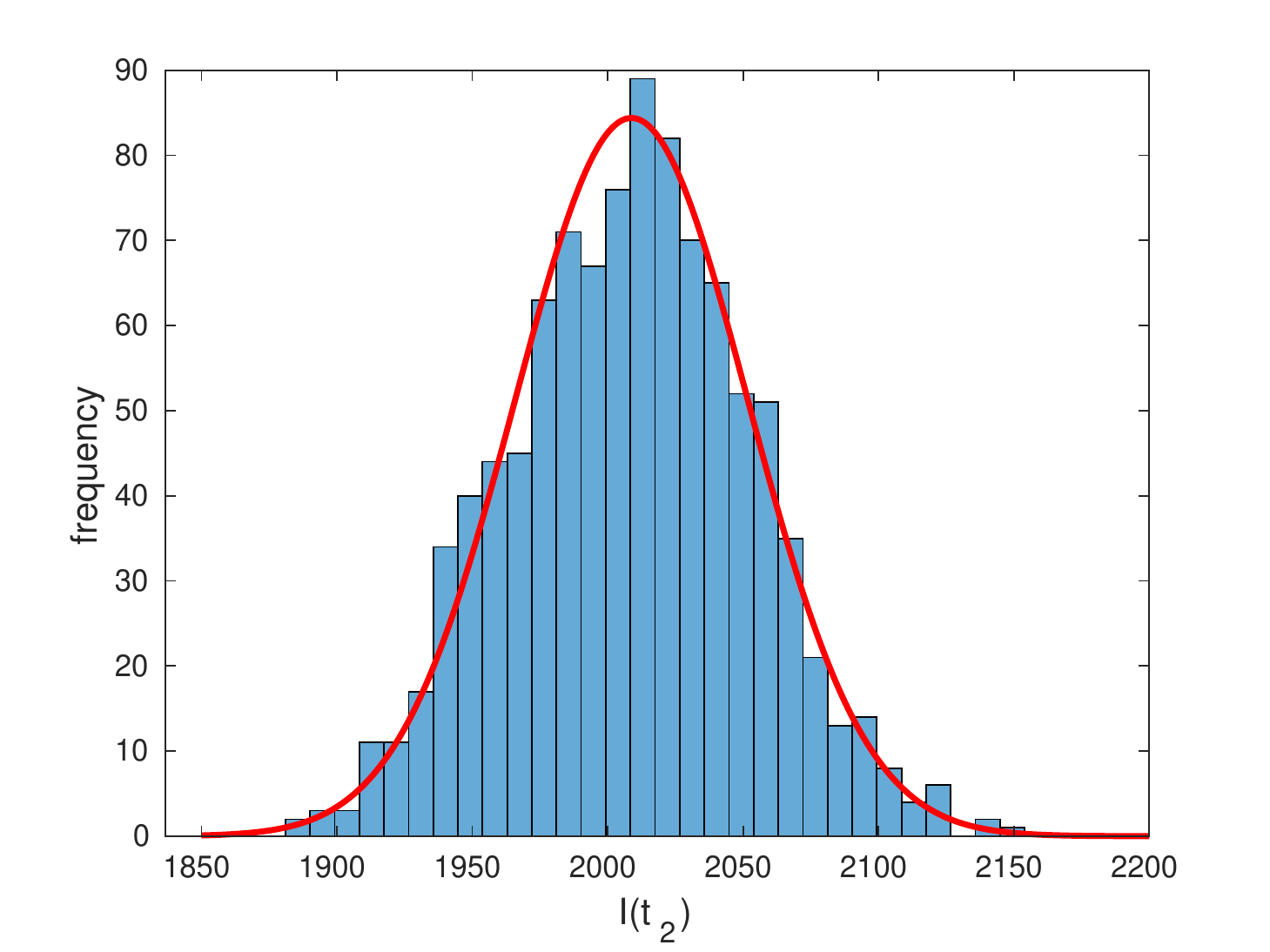}
 & \includegraphics[width=\tfigwidth]{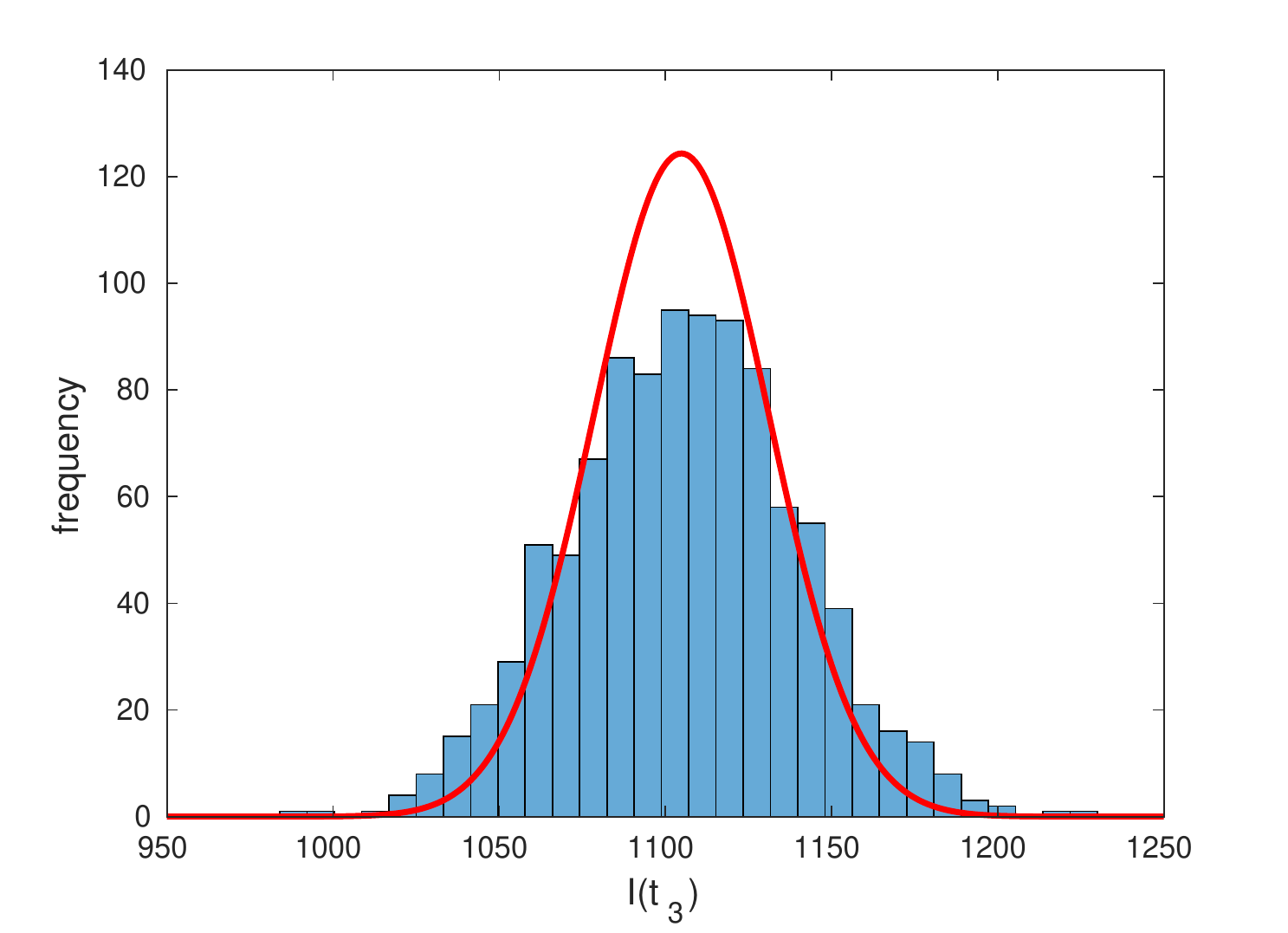}
\end{tabular}
\end{center}
\caption{Demonstration of convergence described by Theorem~\ref{thm:NSWtemporalCLT}, for $D\sim\mbox{Geo}(1/6)$. Histograms of $I^N(t_i)$ (based on 1000 simulated trajectories) and normal approximation for three fixed time points $t_1=0.35$, $t_2=0.6$, $t_3=1.5$ and 3 population sizes $N=200$, $N=1000$, $N=5000$. Other parameters are $\beta=3/2$, $\gamma=1$, $\omega=1$, $i^N_0=0.05N$.}
\label{fig:processCgce}
\end{figure}

\begin{figure}
\begin{center}
\includegraphics[width=\figwidth]{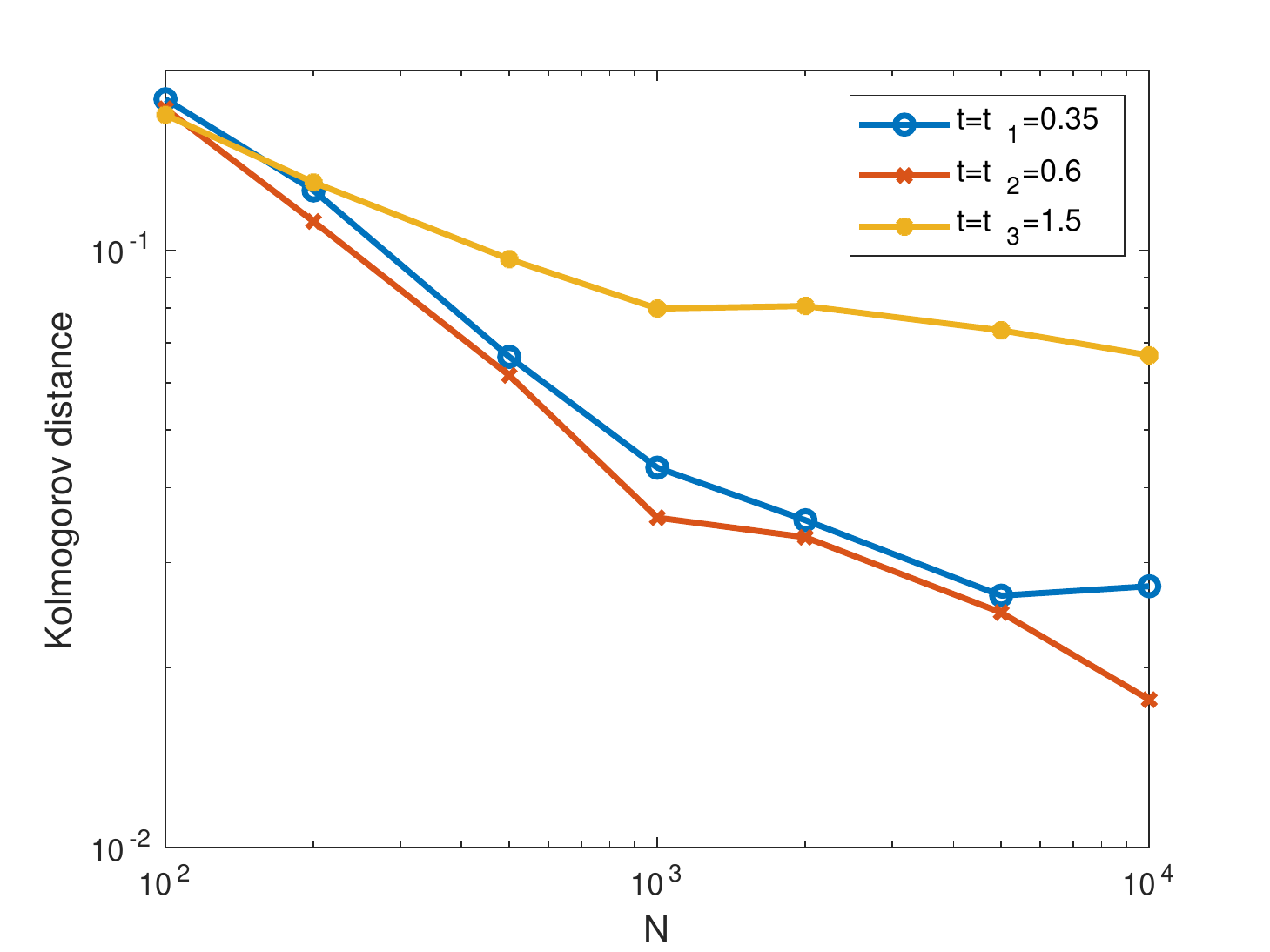}
\end{center}
\caption{Demonstration of convergence described by Theorem~\ref{thm:NSWtemporalCLT}, for $D\sim\mbox{Geo}(1/6)$. The distribution of $I^N(t_i)$ (based on 5000 simulated trajectories) and its normal approximation are compared using the Kolmogorov distance at three fixed time points $t_1=0.35$, $t_2=0.6$, $t_3=1.5$ for population sizes $N=100, 200, 500, 1000, 2000, 5000, 10000$. Other parameters are $\beta=3/2$, $\gamma=1$, $\omega=1$, $i^N_0=0.05N$.}
\label{fig:KolDist}
\end{figure}

Broadly speaking, Figure~\ref{fig:processApprox} and similar plots for other population sizes, together with Figures~\ref{fig:processCgce} and~\ref{fig:KolDist} and similar plots for other degree distributions, show that the predicted convergence is apparent, but seems slower for the later times. Even for quite small population sizes in the low hundreds, the asymptotic approximation to the mean behaviour of the epidemic is excellent. 
With smaller population sizes of a few hundred the approximation of the variability seems quite good in the early phase of epidemic growth, begins to worsen at or slightly before the time of peak prevalence and consistently underestimates the variability of $I^N(t)$ after that. As the population size increases, the approximation for the standard deviation improves but not as quickly as one might hope: the agreement between asymptotic and empirical distributions seems to improve fairly slowly as $N$ increases from 200 to 5000. Thus we can be very confident in using an LLN-based approximation for nearly any population size; but CLT-based approximations must be used with some caution, particularly at and after the time of peak prevalence. For these later times, a CLT-based approximation seems to systematically underestimate the variability in the number of infectives in the population. On a slightly more theoretical note, the plausibly linear (though also decidely noisy) behaviour of the plots in Figure~\ref{fig:KolDist} is consistent with these Kolmogorov distances tending to 0 as $N\to\infty$. Consistent with the observations above, this convergence is at roughly the same rate for the time points in the early growth phase and near peak prevalence but much more slowly for the later time point $t=t_3$ in the phase where the infection is dying out.

%\red{Another thing we haven't mentioned yet, but is of interest, is the following: of the initial fraction infected is small then the variance/sd is not monotone increasing then decreasing, it can have a second local max to the right of the `main' one. In the interests of brevity probably this should be left out---it was noted by Ball and House (Sec 9.1) and possibly earlier too?}

\subsection{Approximation of epidemic final size}
\label{sec:CgceAppFS}
In Figure~\ref{fig:FSapprox} we demonstrate approximation results for the final size of major outbreaks in our epidemic model on an NSW graph (Conjecture~\ref{conj:nswCLT}). Again we see that the approximation is quite reasonable for relatively small population sizes in the low hundreds and becomes very good indeed for population sizes in the thousands.

\begin{figure}
\begin{center}
\begin{tabular}{rcc}
 & (a) Poisson degree & (b) Geometric degree \\
\begin{sideways}\hspace*{1.5cm}$N=200$\end{sideways}
 & \includegraphics[width=\hfigwidth]{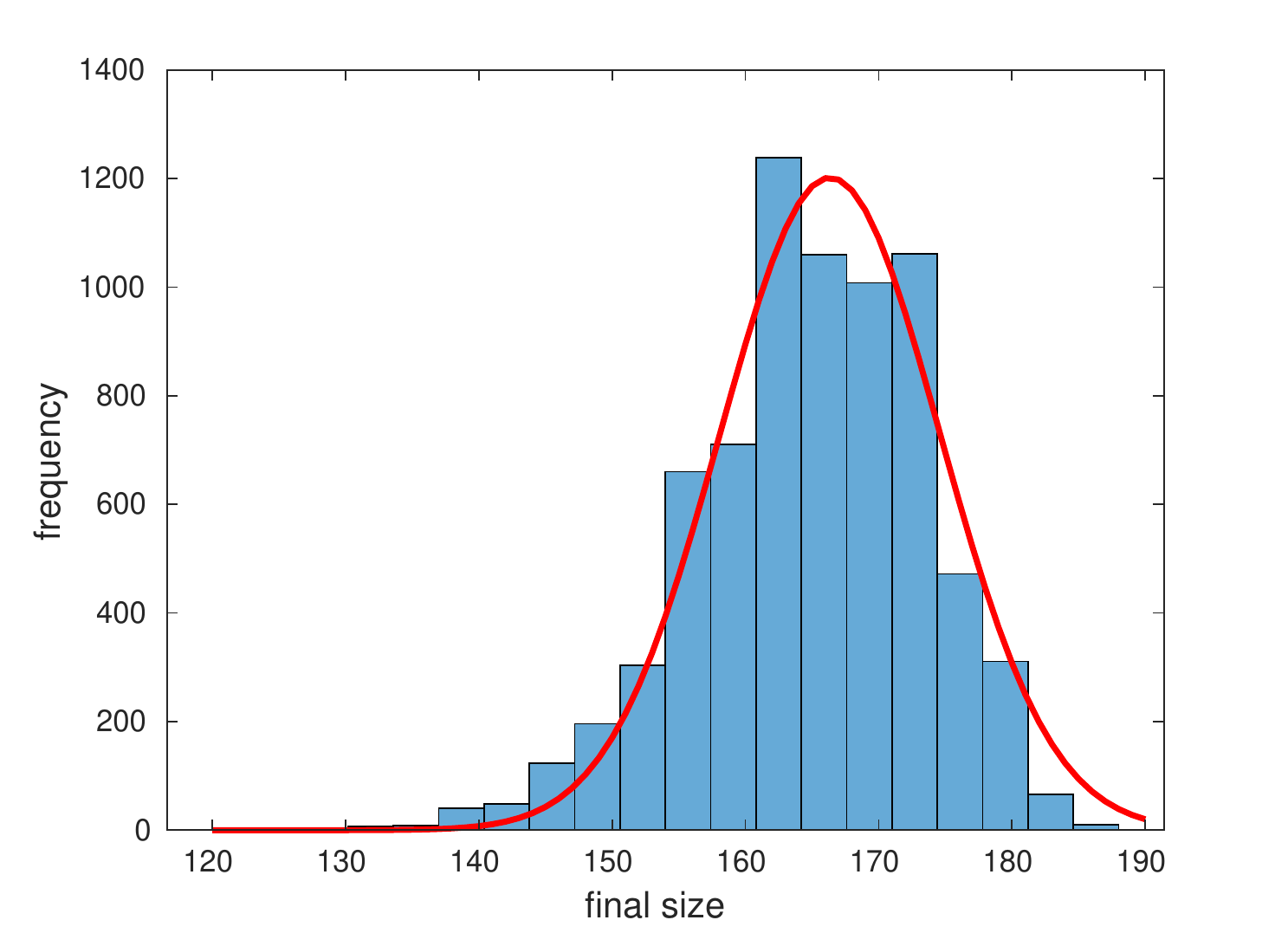}
 & \includegraphics[width=\hfigwidth]{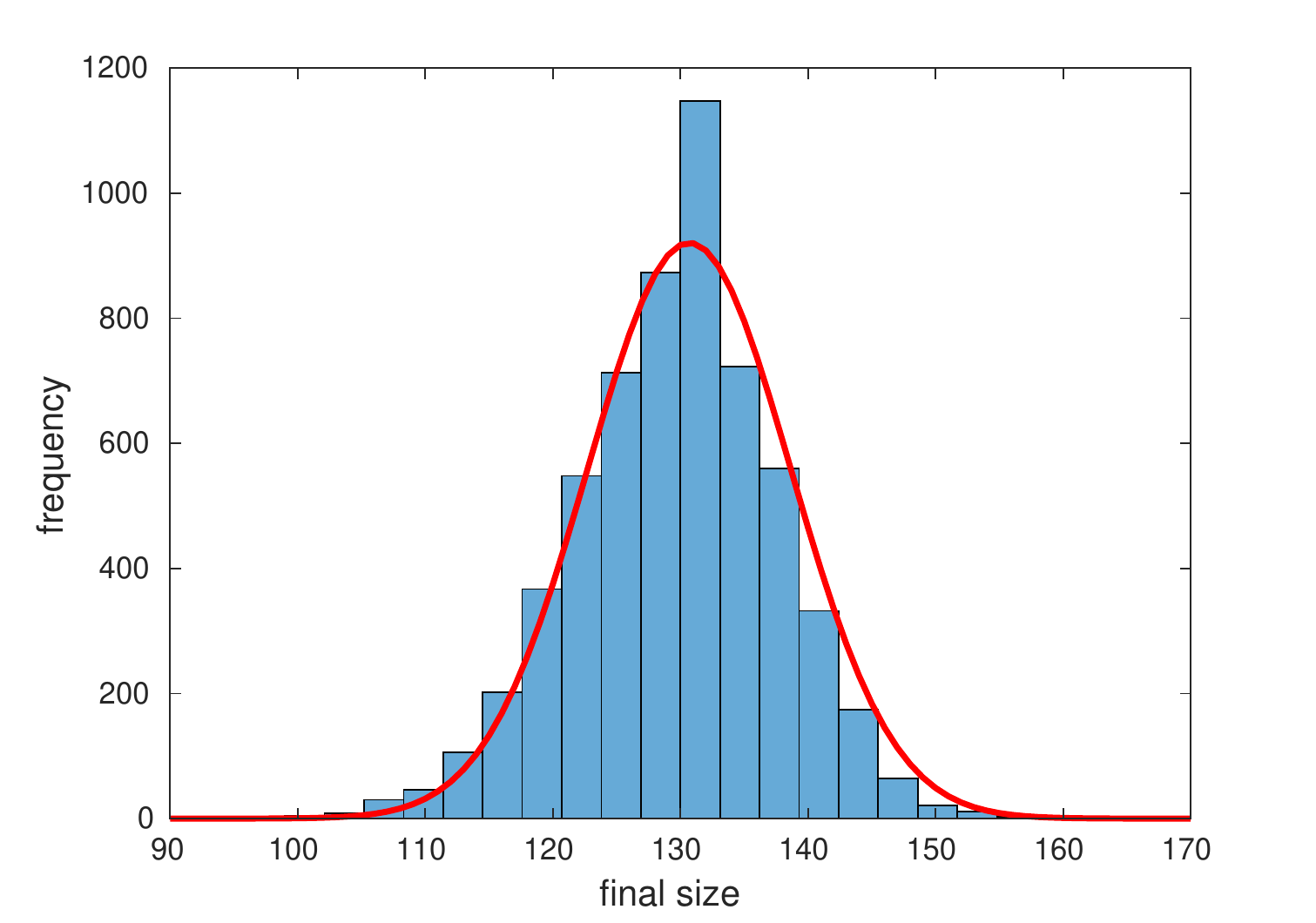} \\
\begin{sideways}\hspace*{1.5cm}$N=2000$\end{sideways}
 & \includegraphics[width=\hfigwidth]{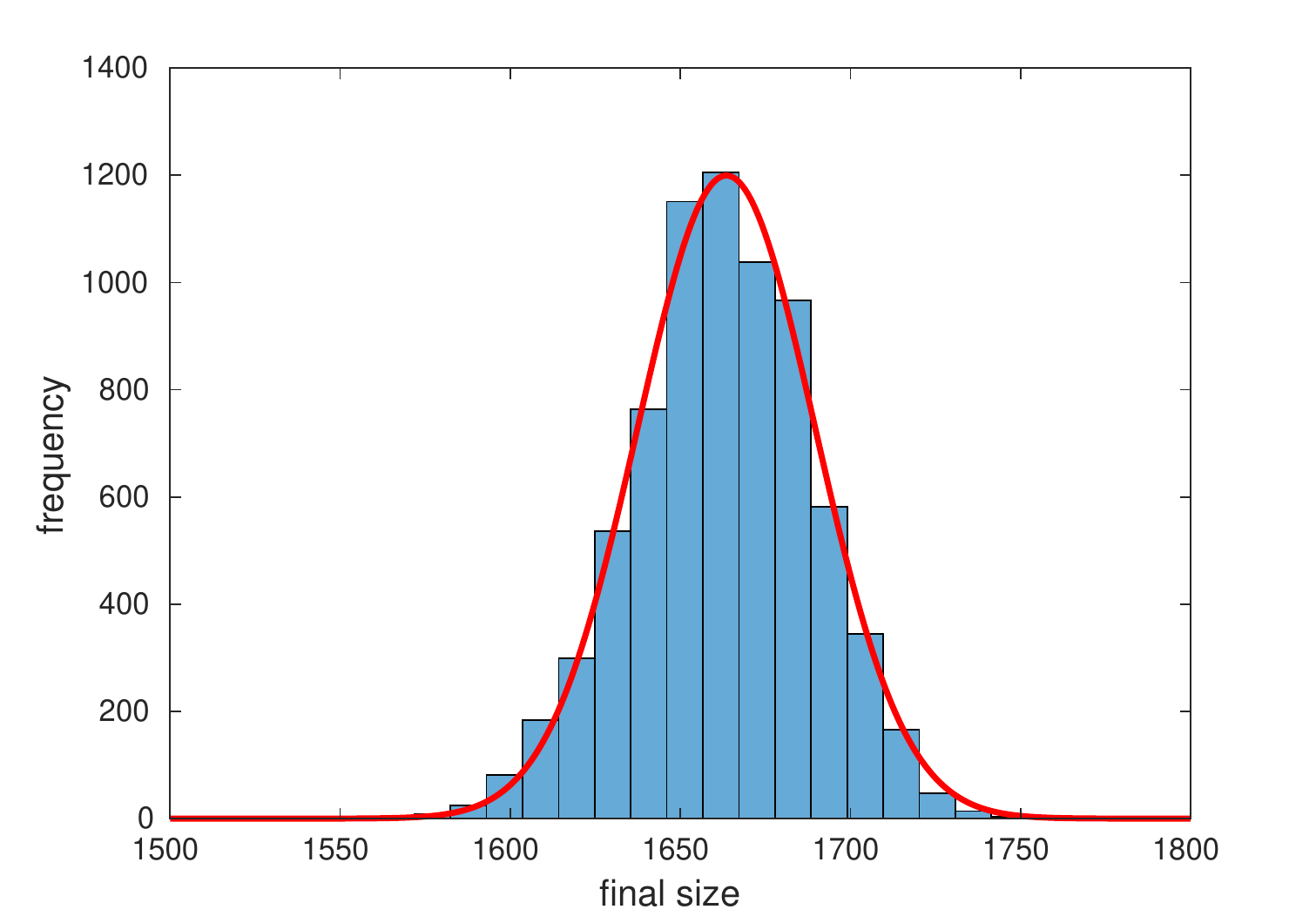}
 & \includegraphics[width=\hfigwidth]{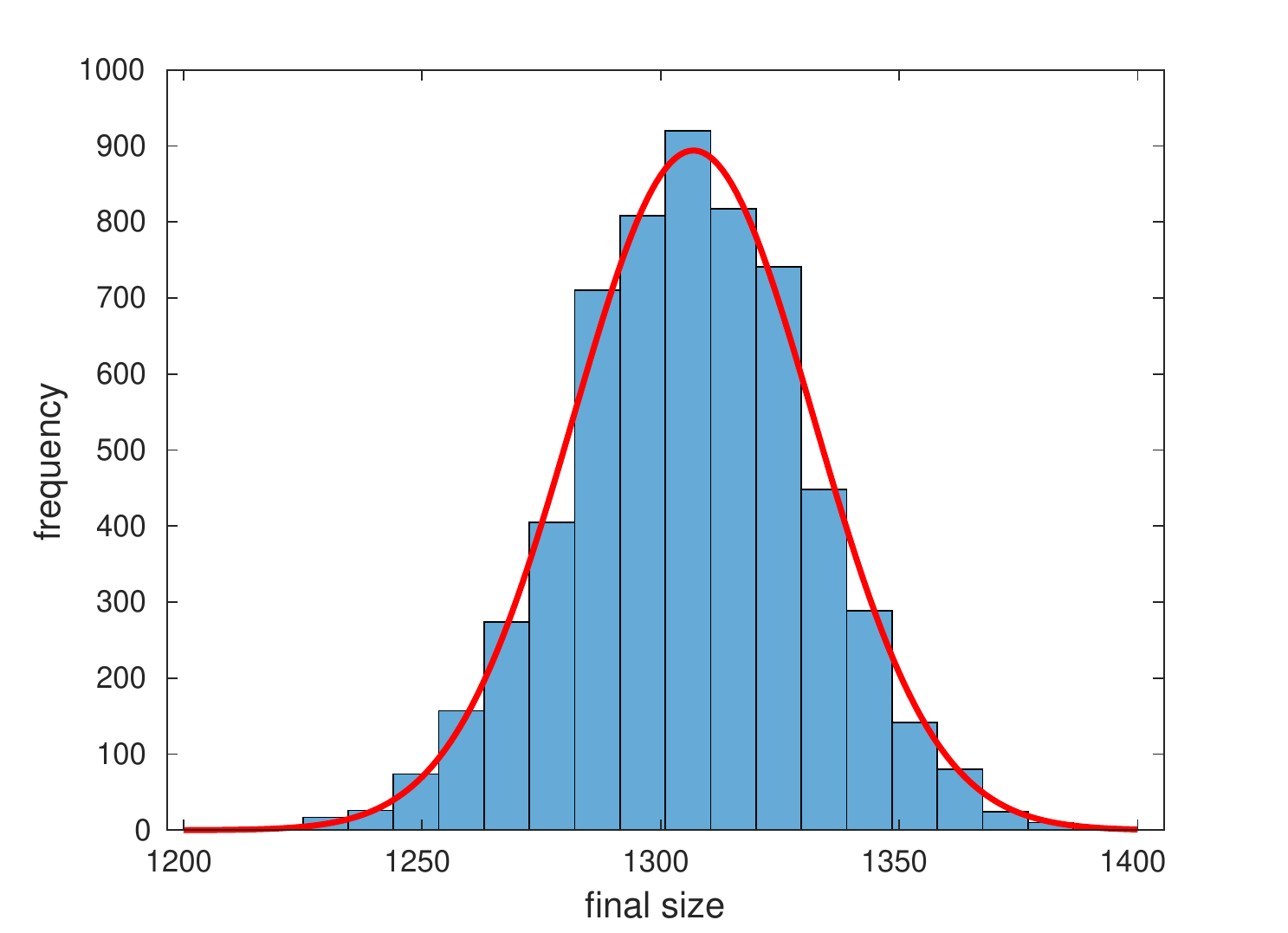}
\end{tabular}
\end{center}
\caption{Histograms (based on 10,000 simulations) and normal approximation of the final size distribution of major outbreaks for epidemics on graphs with (a) $D\sim\mbox{Poi}(5)$, (b) $D\sim\mbox{Geo}(1/6)$ and varying populations sizes $N=200, 2000$. Other parameter values are $\beta=3/2$, $\gamma=\omega=1$ and $i^N_0=10$.}
\label{fig:FSapprox}
\end{figure}

\subsection{The effect of dropping}
\label{sec:droppingEffect}

Next we investigate the behaviour of our model in respect of the introduction of the dropping mechanism. Starting with an epidemic without dropping we examine the behaviour of $R_0$ and $\rho$ (the fraction of the population that is ultimately infected in the limiting
determinstic model -- see Section~\ref{sec:finalsize}) as the dropping rate $\omega$ is increased from 0 (no dropping) to a value which brings the model below threshold. Figure~\ref{fig:dropping} does this for two `starting' models, one with a Poisson and one with a geometric degree distribution, both well above threshold with with $\rho$ comfortably above 0.5. (Recall that $R_0$ and $\rho$ are independent of whether the network is MR or NSW.)
In both cases we see that increasing $\omega$ reduces the virulence and severity of the epidemic as measured by $R_0$ and $\rho$. Perhaps noteworthy is that one of the plots of the mean final size $\rho$ is concave and the other convex.

\begin{figure}
\begin{center}
\begin{tabular}{cc}
(a) $D$ Poisson & (b) $D$ Geometric \\
\includegraphics[width=\hfigwidth]{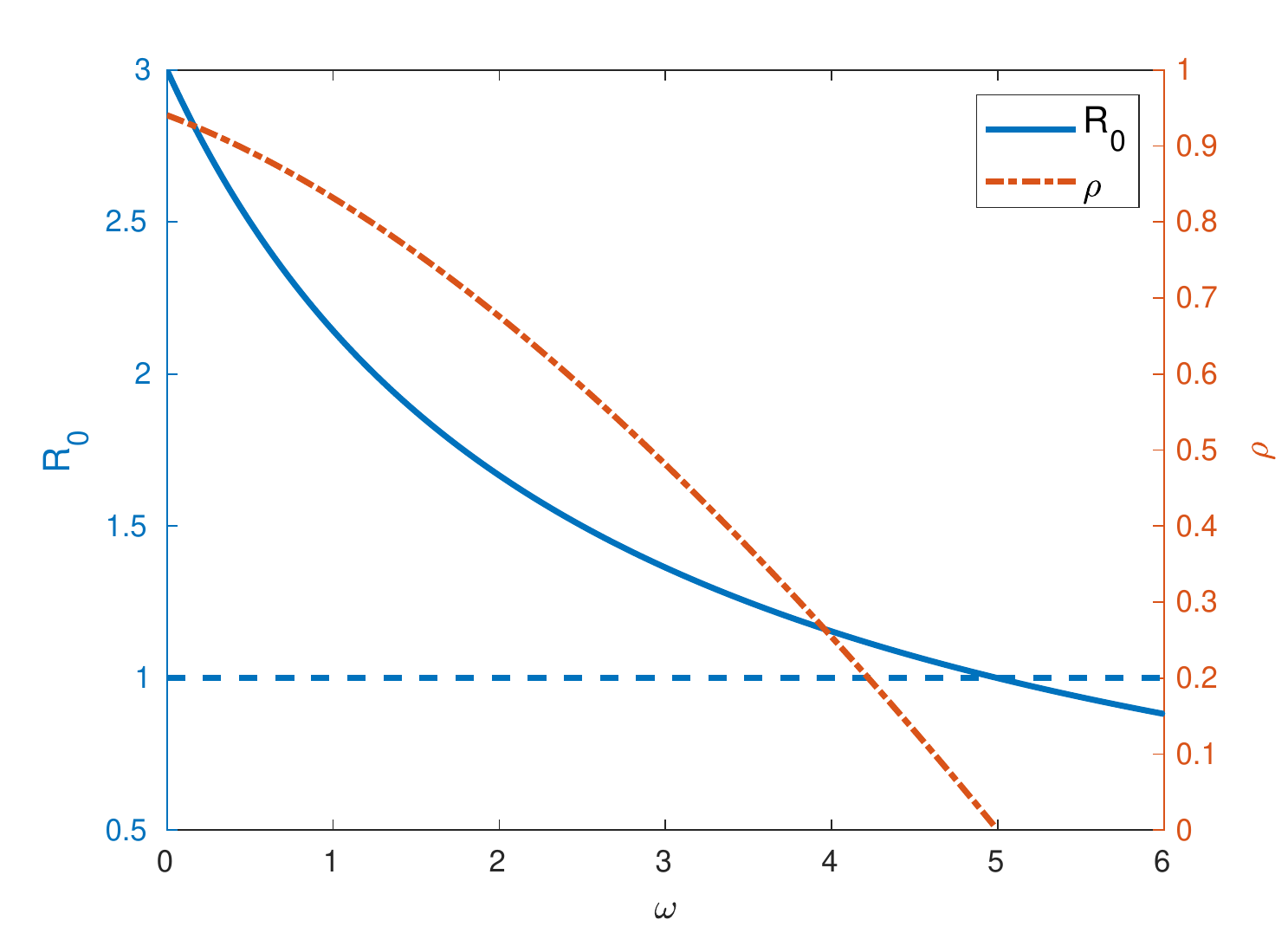} & \includegraphics[width=\hfigwidth]{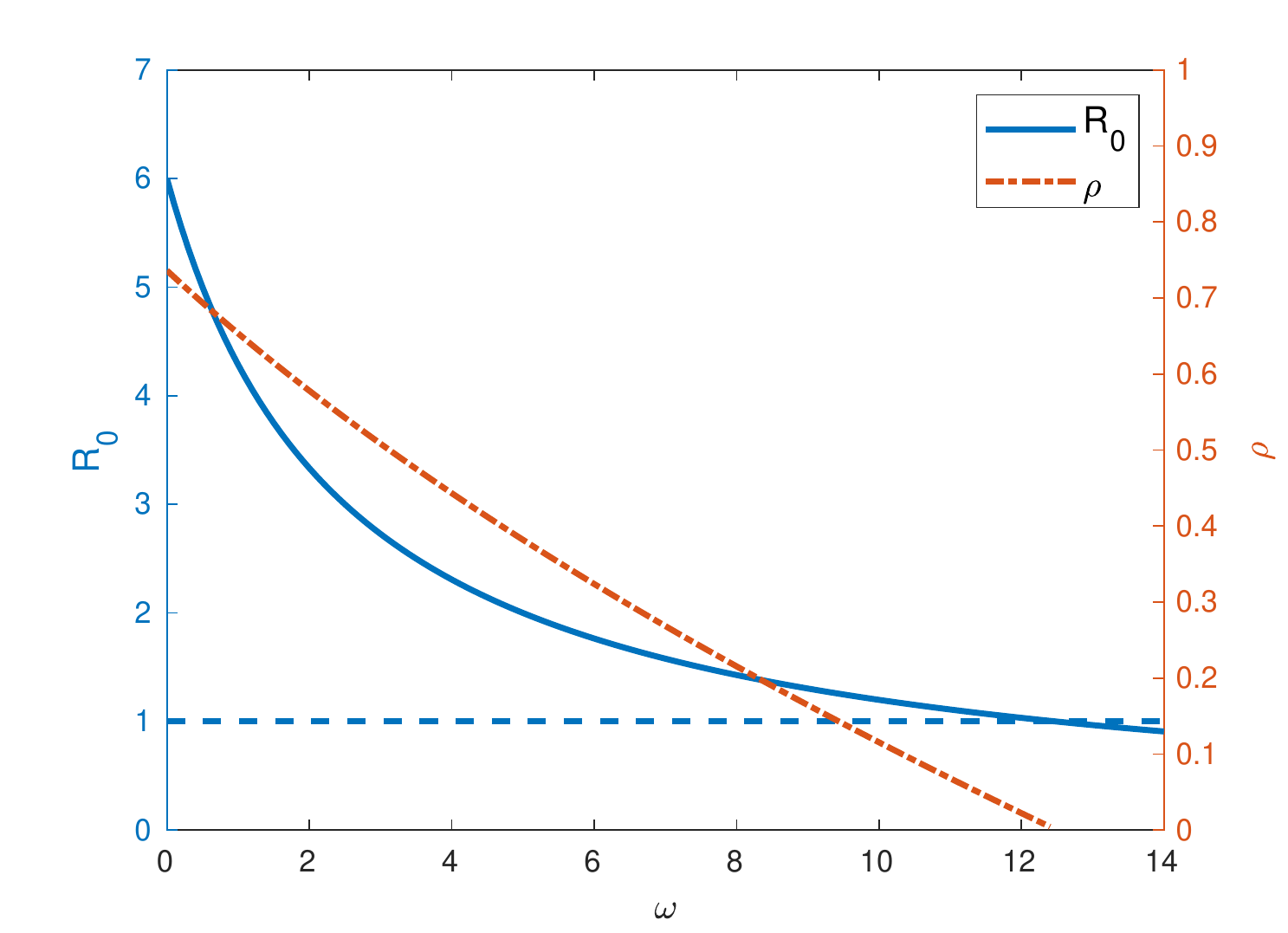}
\end{tabular}
\end{center}
\caption{Plots of $R_0$ and $\rho$ showing the impact of increasing the dropping rate from zero. Other parameters are $\beta=3/2$, $\gamma=1$, $\epsilon=0$ and (a) $D\sim\mbox{Poi}(5)$, (b) $D\sim\mbox{Geo}(1/6)$.}
\label{fig:dropping}
\end{figure}

\subsection{The effect of random graph model on variances}
\label{sec:variance}
We now demonstrate the effect of the random graph model (MR or NSW) on the variability of the final size of large outbreaks in our epidemic model.  Figure~\ref{fig:MRNSWvariance} compares how the asymptotic scaled standard deviations for the final size of a major outbreak (i.e.\ $\sigma_{\rm MR}(\beta,\gamma,\omega)$ and $\sigma_{\rm NSW}(\beta,\gamma,\omega)$ in Proposition~\ref{prop:mrVar} and Conjecture~\ref{conj:nswCLT} ) behave as dropping is included into a baseline model with no dropping. The upper plots show that these standard deviations can change quite dramatically with $\omega$; the lower plots show that the extra variability in the NSW network model can result in substantially more variability in the epidemic final size. As might be anticipated, this effect is more pronounced for the geometric compared to the Poisson case, i.e.~when the degree distribution is more variable.

\begin{figure}
\begin{center}
\begin{tabular}{cc}
(a) Poisson degree & (b) Geometric degree \\
\includegraphics[width=\hfigwidth]{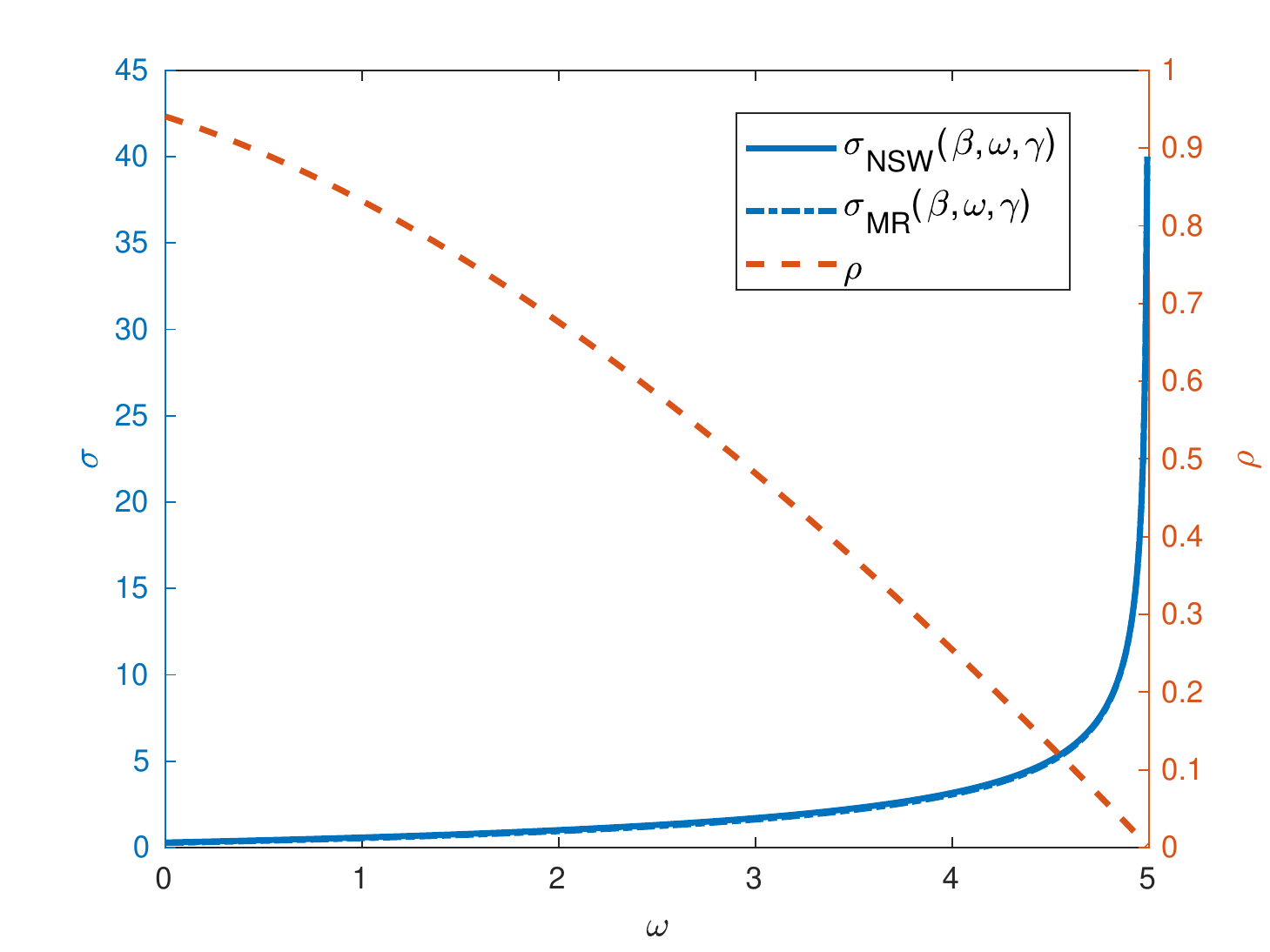} & \includegraphics[width=\hfigwidth]{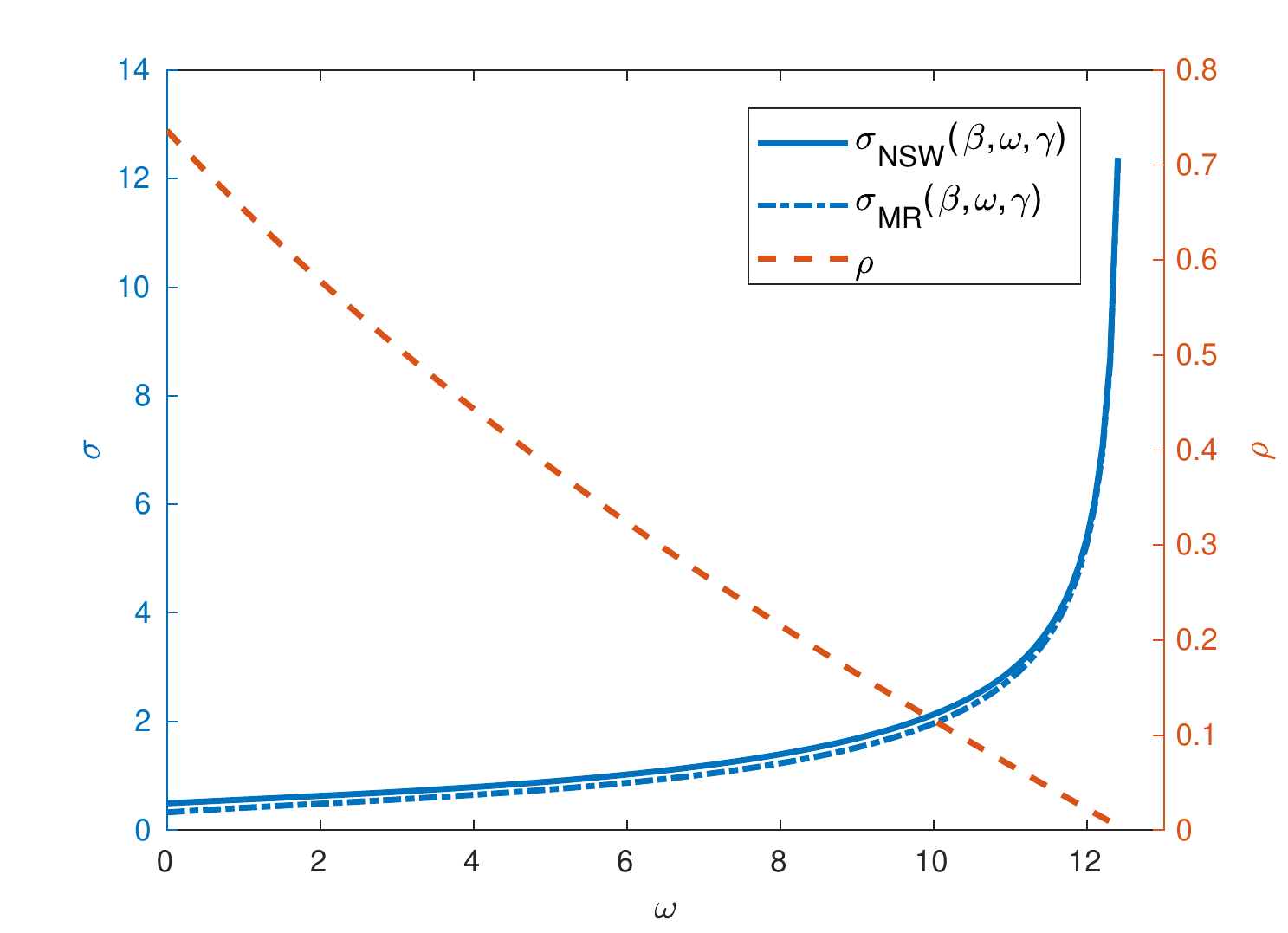} \\
\includegraphics[width=\hfigwidth]{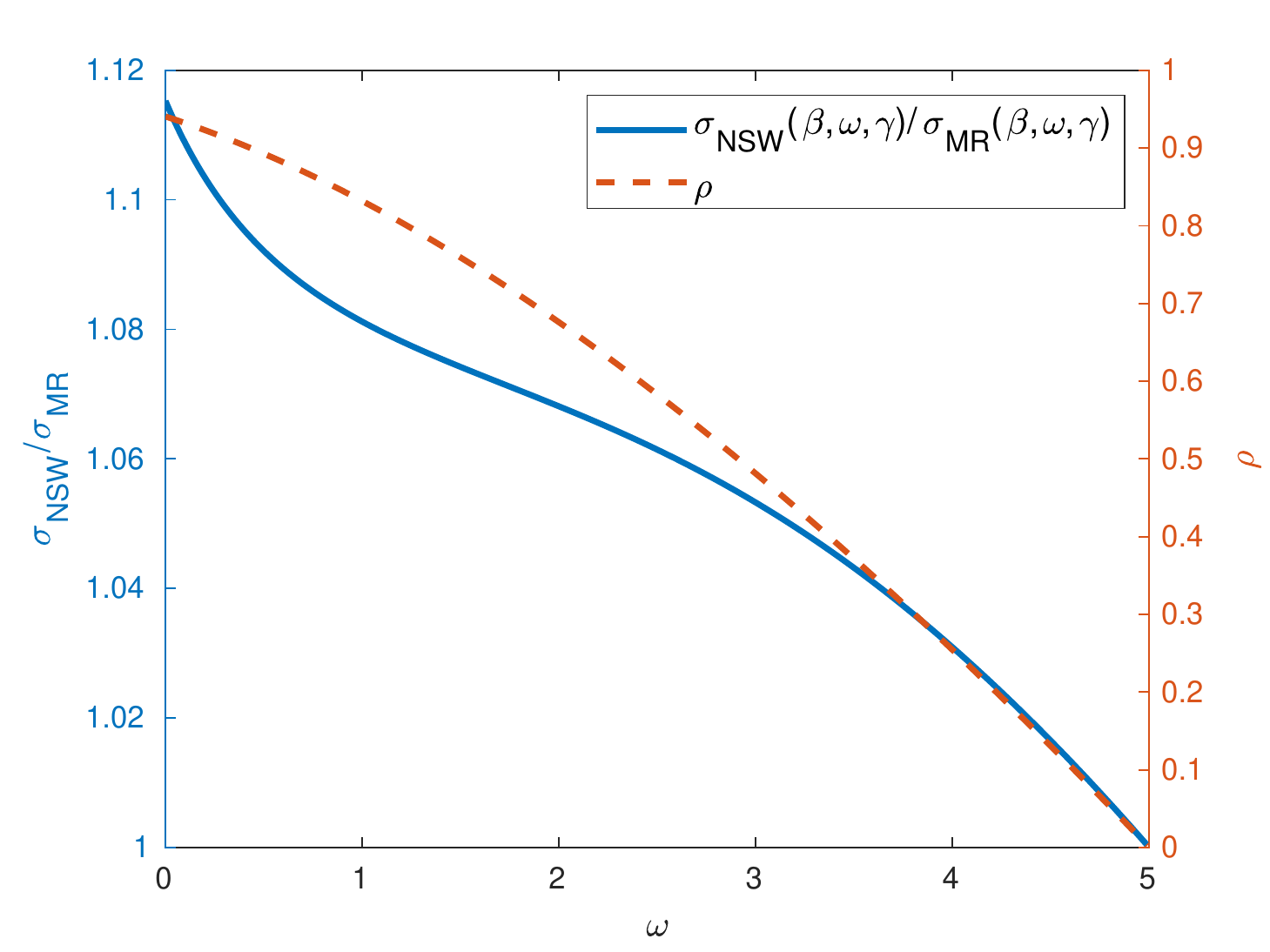} & \includegraphics[width=\hfigwidth]{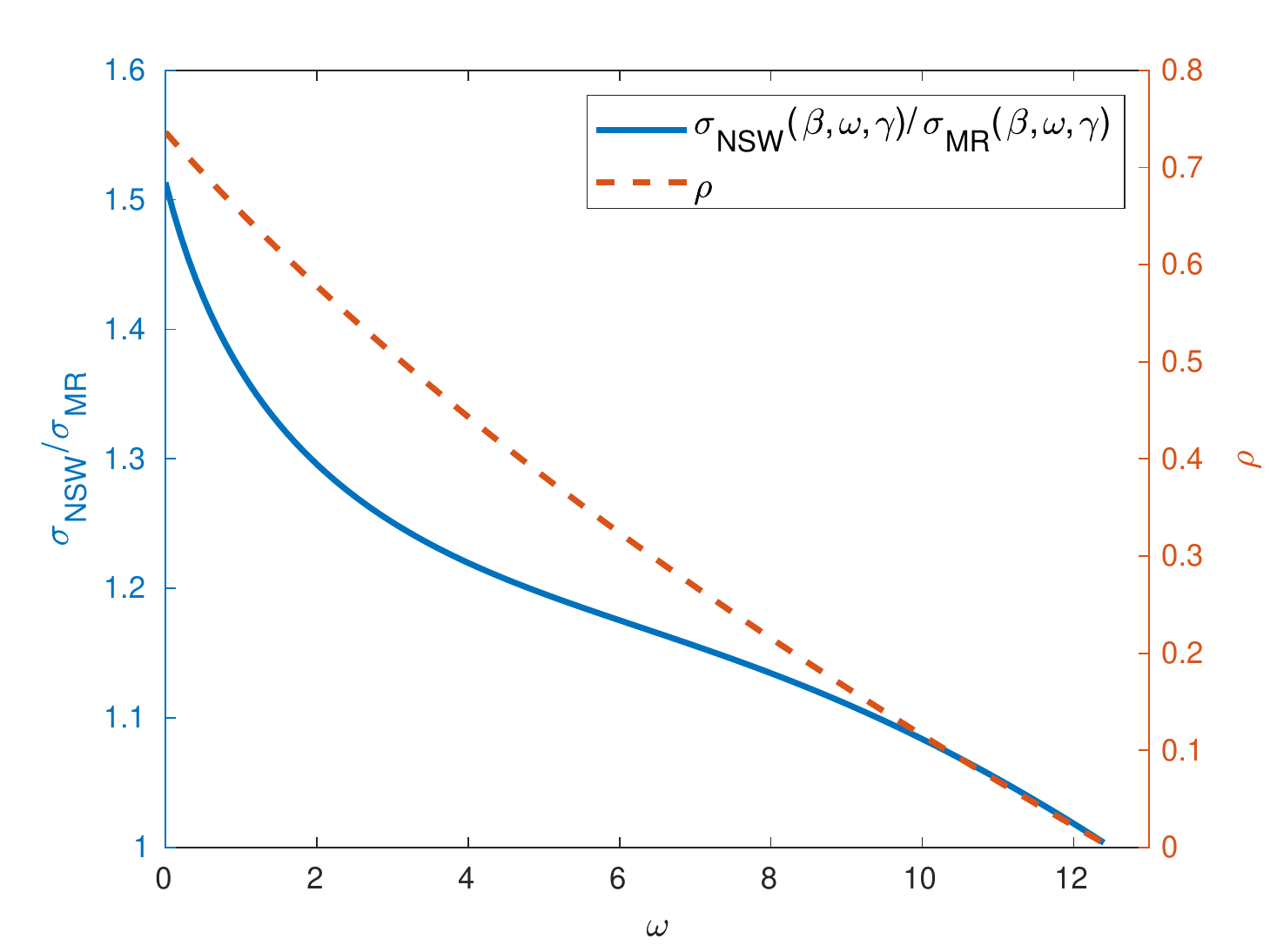}
\end{tabular}
\end{center}
\caption{Comparison of the scaled asymptotic standard deviations of the final size of a large outbreak in the MR and NSW models, as the dropping rate $\omega$ is increased from zero. In (a) the degree distribution is $D\sim\mbox{Poi}(5)$ and in (b) $D\sim\mbox{Geo}(1/6)$; other parameters are $\beta=3/2$, $\gamma=1$, $\epsilon=0$. The upper plots show the actual scaled asymptotic standard deviations $\sigma_{\rm MR}$ and $\sigma_{\rm NSW}$ and the lower plots show their ratio; all plots also show the relative final size $\rho$ for reference.}
\label{fig:MRNSWvariance}
\end{figure}

\subsection{Increased recovery rate instead of dropping}
\label{sec:incRecovery}
Lastly we investigate the relationship between our model and the related model with increased recovery rate instead of dropping, as discussed in Section~\ref{sec:relatedmodel}. We focus mainly on the claims about relative variability in the two models $E(\omega,\gamma)$ with dropping and $E(0,\gamma+\omega)$ with increased recovery rate, though the results we present also illustrate Theorem~\ref{prop:pmajor}, which gives an ordering of the major outbreak probabilities in the two models. Again we focus on the NSW graph model; similar conclusions (with less variability) are obtained with the MR graph model.

Figure~\ref{fig:incRecovery1} compares the final size distribution of the model with dropping to that of the model with increased recovery rate; again for two different degree distributions. The histograms and the normal approximation of the distribution of the size of a major outbreak confirm that the model with dropping does have a smaller variance in the size of major outbreaks and a larger chance of a major outbreak. Table~\ref{tab:incRecovery1} summarises the plots in Figure~\ref{fig:incRecovery1}. Here we see quite clearly that the major outbreak probabilities and the variances of the final size distributions are ordered as predicted by Theorem~\ref{prop:pmajor} and the argument involving differing dependence structures in Section~\ref{sec:relatedmodel}. Differences between the two degree distributions are not very marked.

\begin{figure}
\begin{center}
\begin{tabular}{cc}
(a) Poisson degree & (b) Geometric degree \\
$E(\omega,\gamma)$ & $E(\omega,\gamma)$ \\
\includegraphics[width=\hfigwidth]{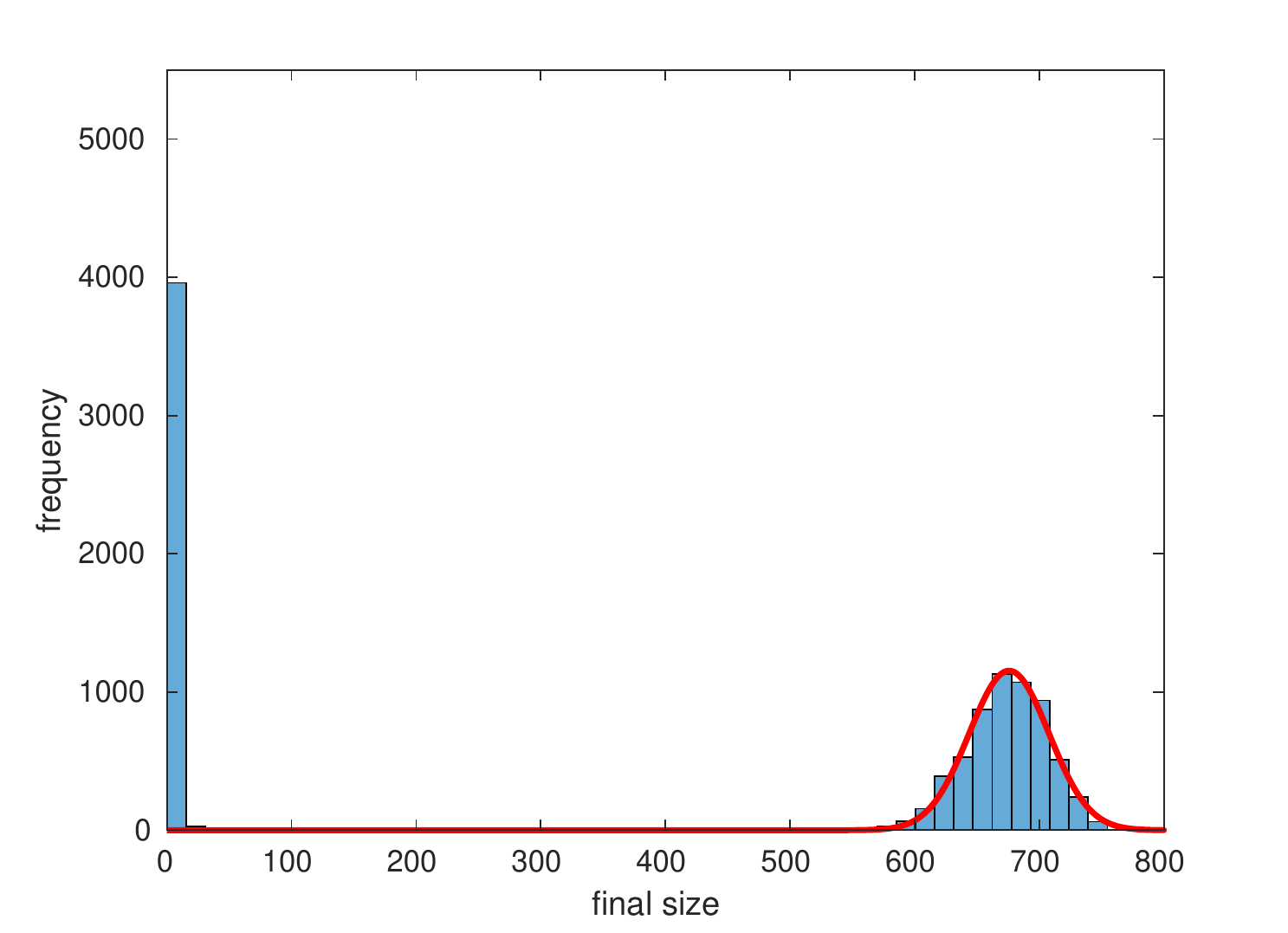} & \includegraphics[width=\hfigwidth]{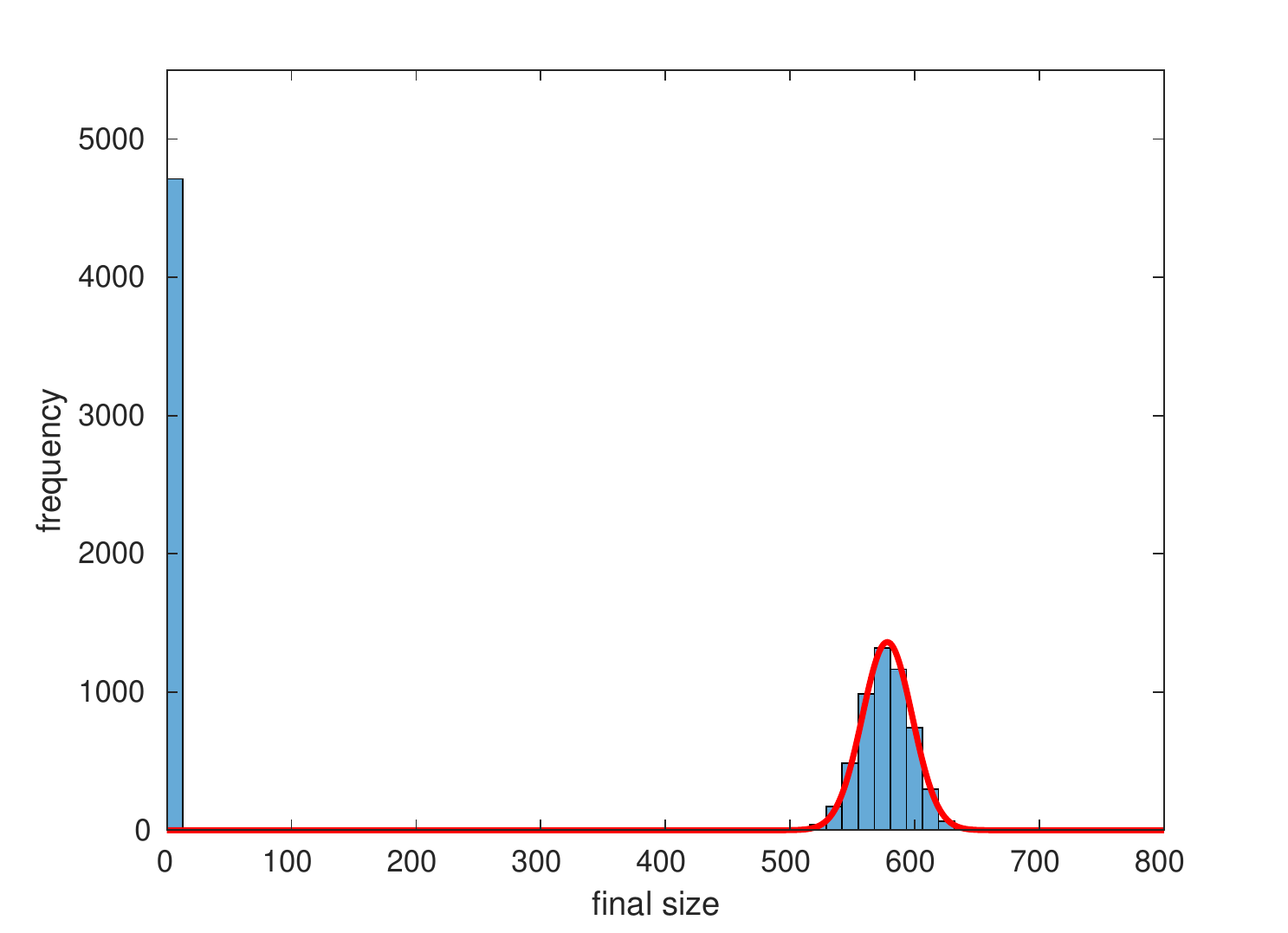} \\
$E(0,\gamma+\omega)$ & $E(0,\gamma+\omega)$ \\
\includegraphics[width=\hfigwidth]{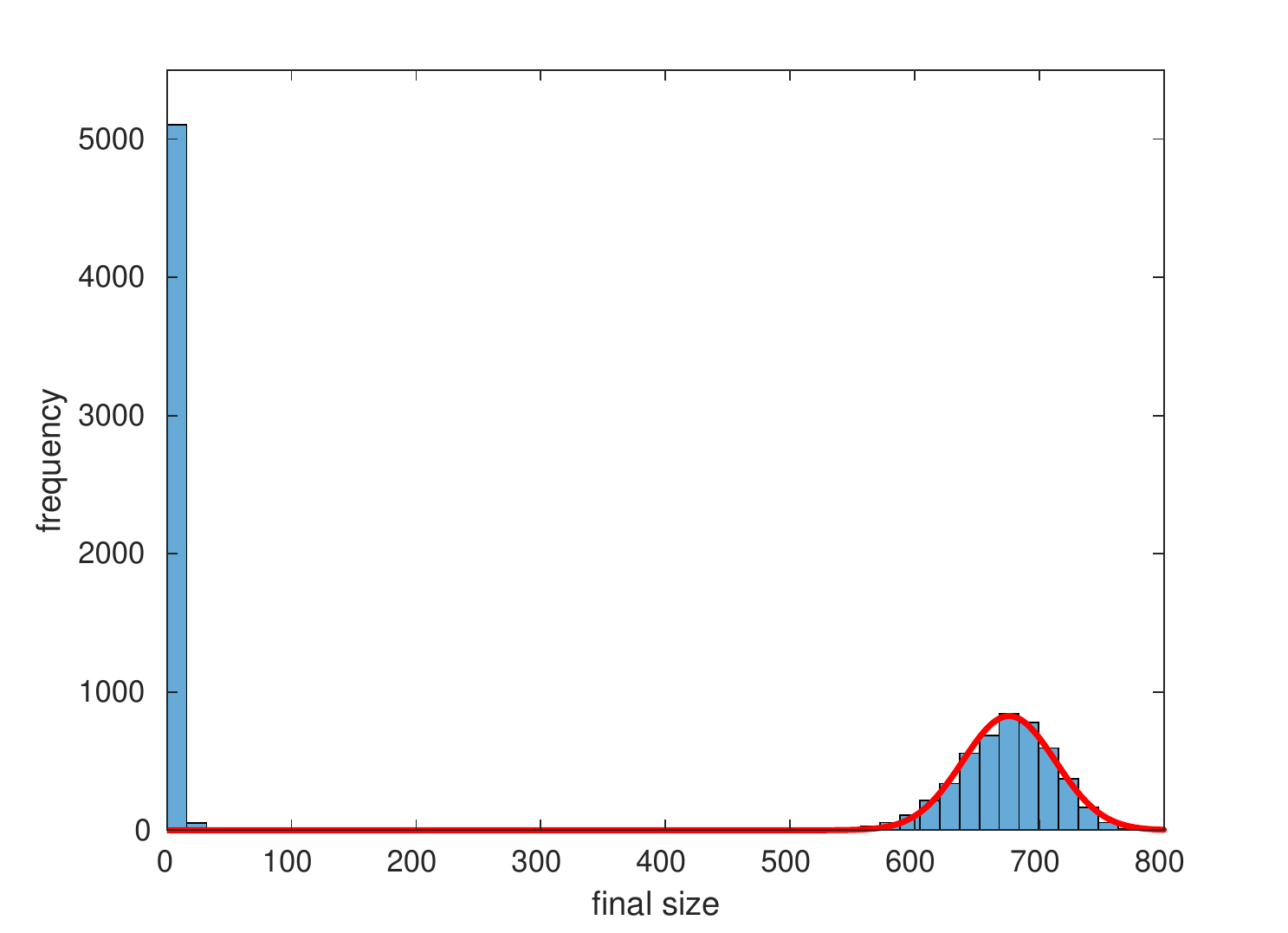} & \includegraphics[width=\hfigwidth]{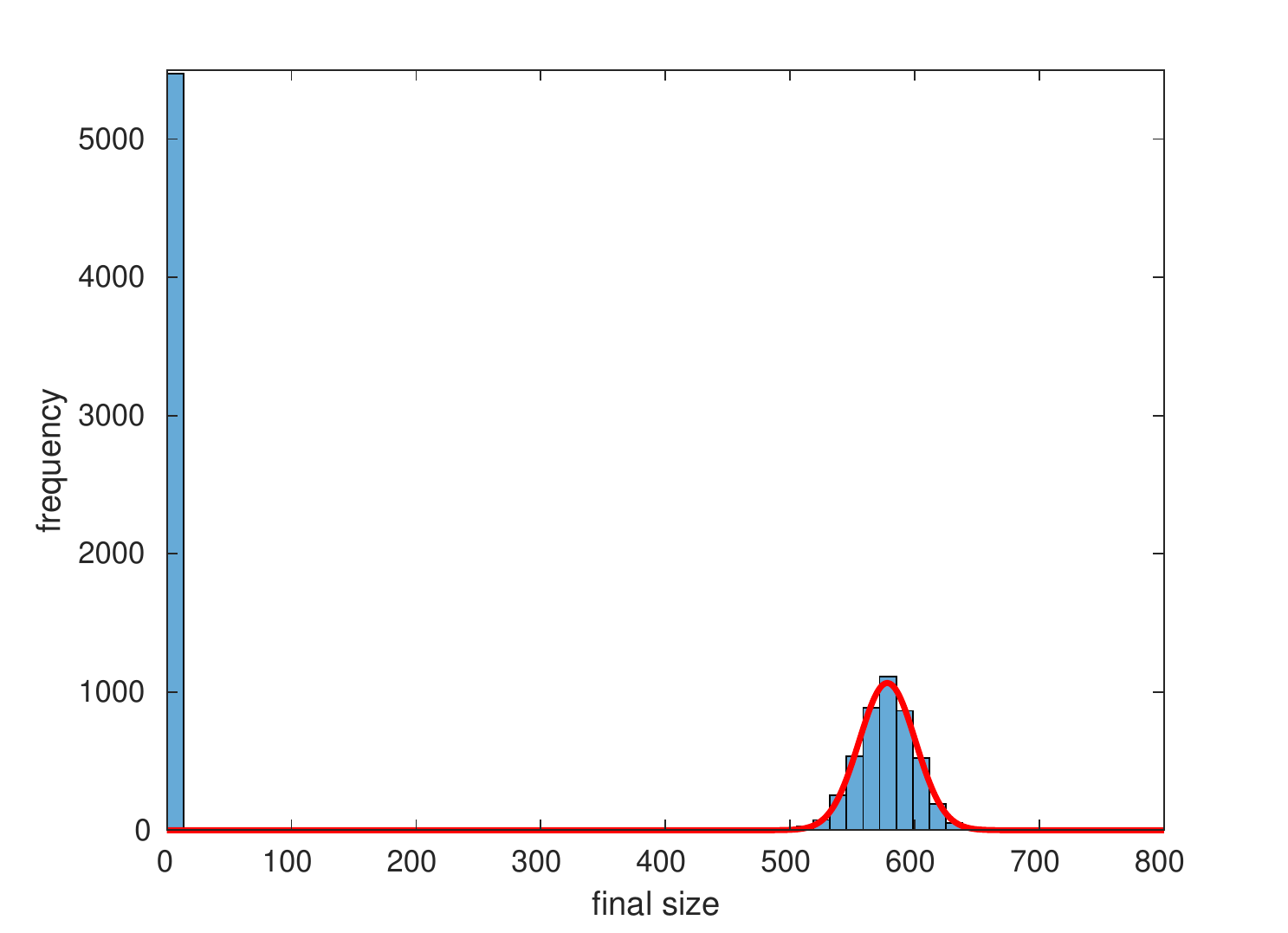}
\end{tabular}
\end{center}
\caption{Histograms of $10000$ simulated final sizes for the epidemics $E(\omega,\gamma)$ and $E(0,\gamma+\omega)$, with overlaid asymptotic approximations. Parameters are $\beta=3/2$, $\gamma=1$, $\omega=2$, $N=1000$, $i^N_0=5$ and the underlying graphs are of NSW type with (a) $D\sim\mbox{Poi}(5)$, (b) $D\sim\mbox{Geo}(1/6)$.}
\label{fig:incRecovery1}
\end{figure}

\begin{table}
\begin{center}
\begin{tabular}{m{2cm}|lll|lll}
Quantity & \multicolumn{3}{|c}{$D$ Poisson, Model $E(\omega,\gamma)$} & \multicolumn{3}{|c}{$D$ Geometric, Model $E(\omega,\gamma)$} \\ \hline
 & Asymp. & Point & 95\% CI & Asymp. & Point & 95\% CI \\
Prob.\ of MO & -- & 0.601 & (0.592,0.611) & -- & 0.529 & (0.519,0.539) \\
Mean of MO final size& 675.8 & 673.5 & (672.6,674.3) & 578.0 & 576.8 & (576.3,577.4)  \\
St.\ dev.\ of MO final size & 32.0 & 32.4 & (31.8,33.0) & 20.0 & 20.3 & (19.9,20.6) \\ \hline
 & \multicolumn{3}{|c}{$D$ Poisson, Model $E(0,\gamma+\omega)$} & \multicolumn{3}{|c}{$D$ Geometric, Model $E(0,\gamma+\omega)$} \\ \cline{2-7}
 & Asymp. & Point & 95\% CI & Asymp. & Point & 95\% CI \\
Prob.\ of MO & -- & 0.483 & (0.474,0.493) & -- & 0.453 & (0.443,0.463) \\
Mean of MO final size & 675.8 & 672.6 & (671.6,673.7) & 578.0 & 577.2 & (576.5,577.8)  \\
St.\ dev.\ of MO final size & 37.1 & 38.3 & (37.6,39.1) & 22.6 & 22.4 & (21.9,22.8)
\end{tabular}
\end{center}
\caption{Numerical summary of Figure~\ref{fig:incRecovery1}, using a final size of $0.15N$ to separate minor from major outbreaks. (In the first column we use the abbreviation MO for `major outbreak'.)}
\label{tab:incRecovery1}
\end{table}

Figure~\ref{fig:incRecovery2} shows how the discrepancy in these variabilities generally increases with the dropping rate. Interestingly, we see that with the (more variable) geometric degree distribution the relative discrepancy increases with $\omega$ for most values of $\omega$; but decreases slightly with $\omega$ when $\omega$ is sufficiently large that the size of large outbreaks gets close to zero and the variability is quite large.

\begin{figure}
\begin{center}
\begin{tabular}{cc}
(a) Poisson degree & (b) Geometric degree \\
\includegraphics[width=\hfigwidth]{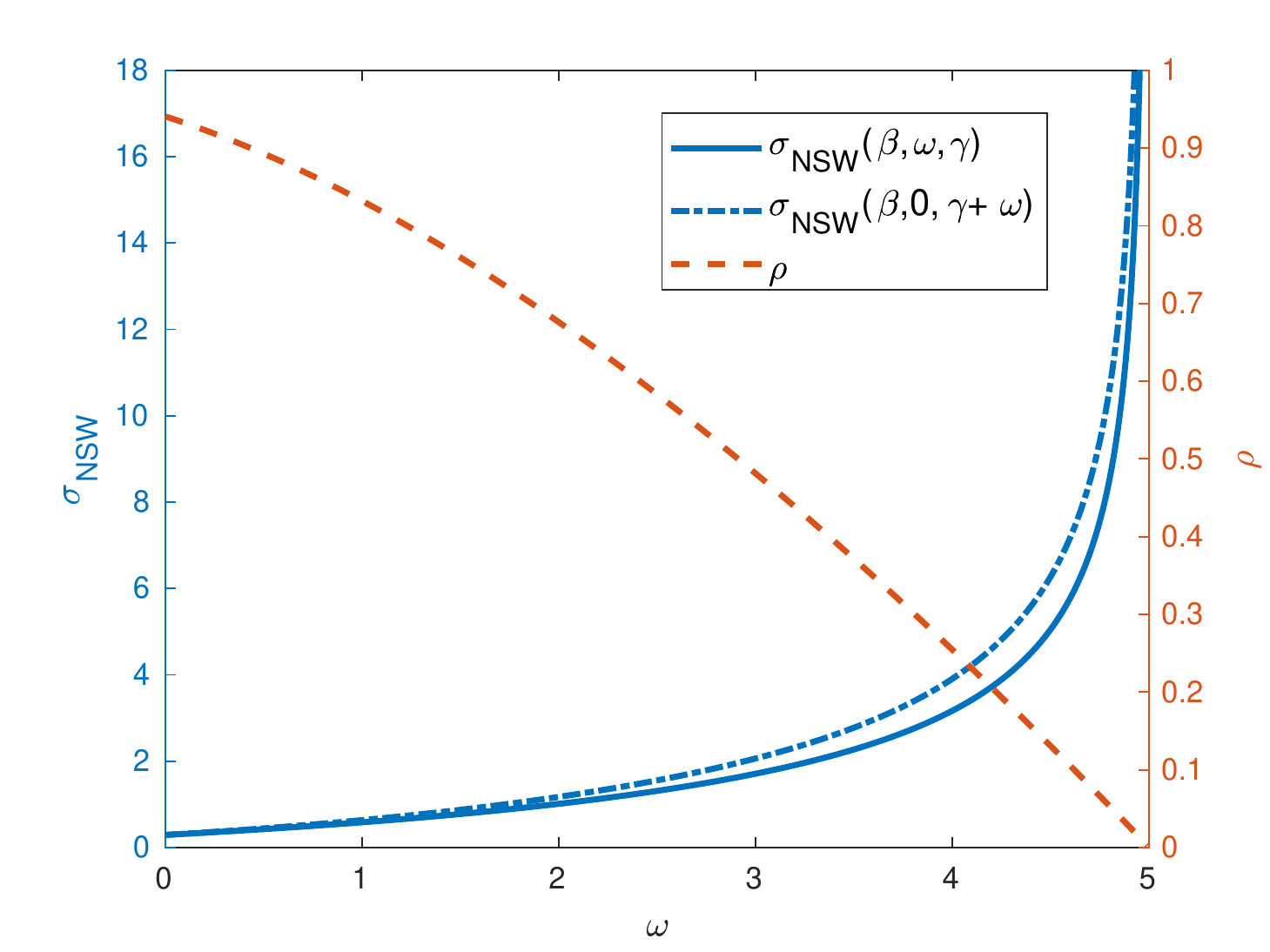} & \includegraphics[width=\hfigwidth]{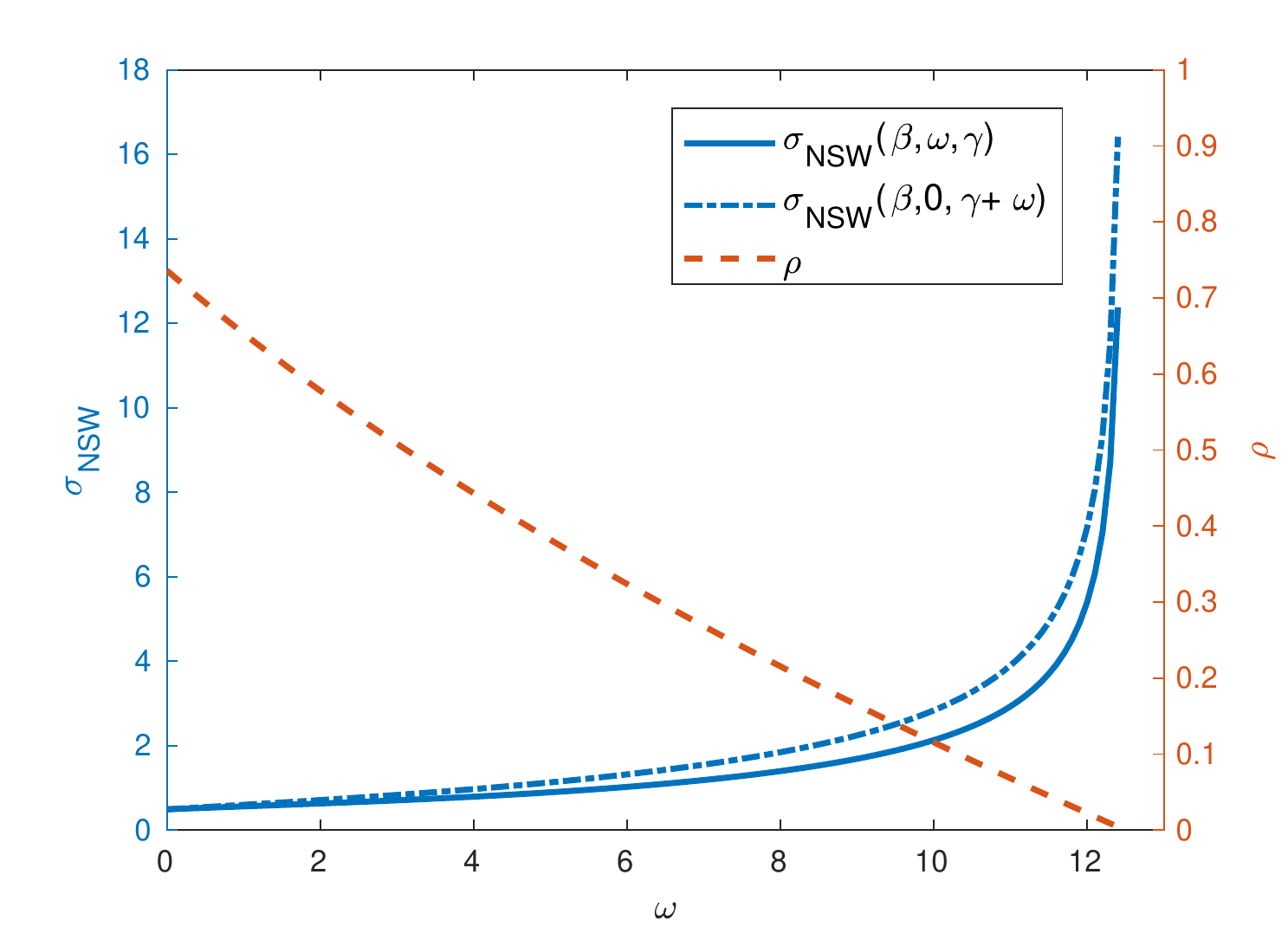} \\
\includegraphics[width=\hfigwidth]{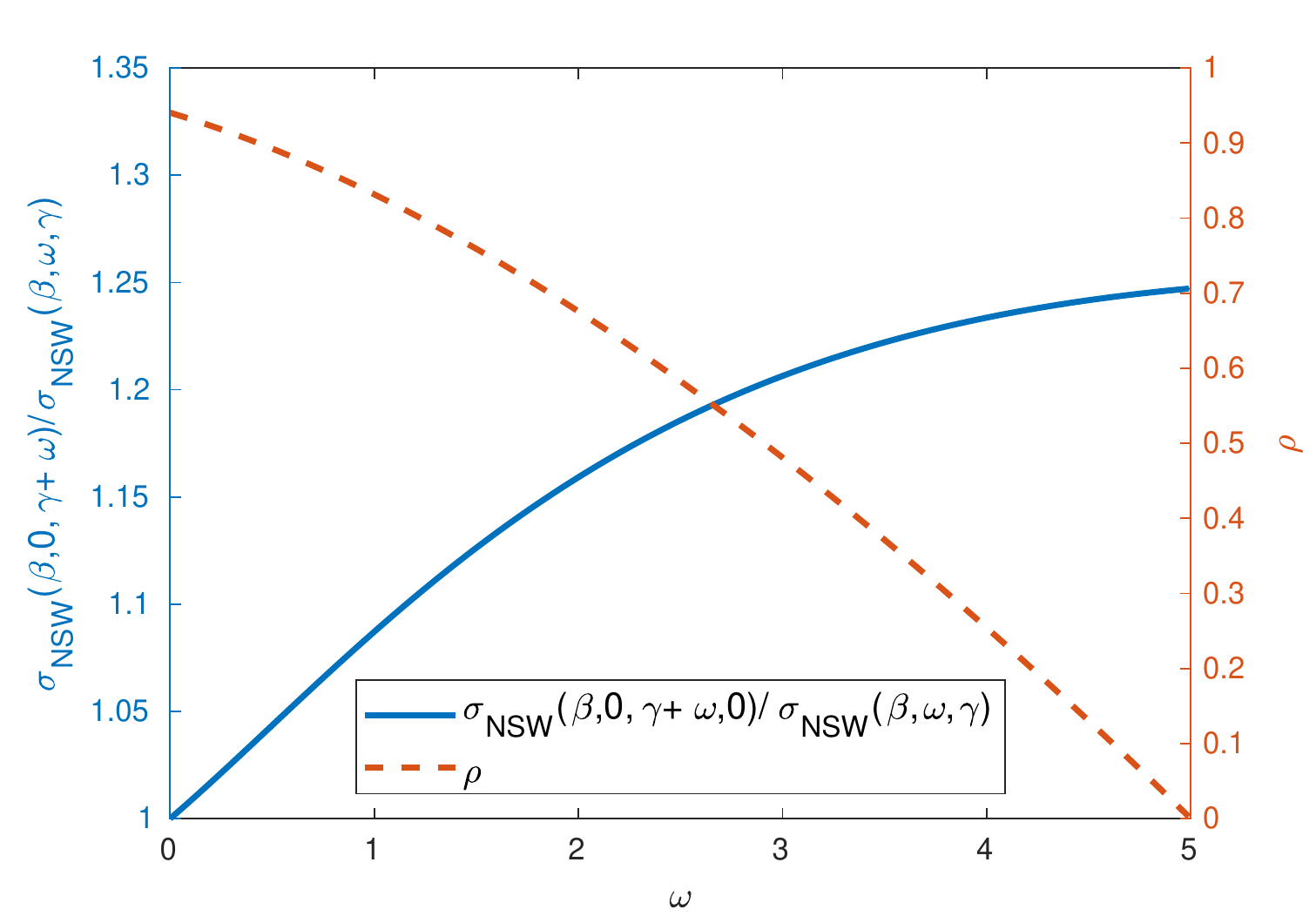} & \includegraphics[width=\hfigwidth]{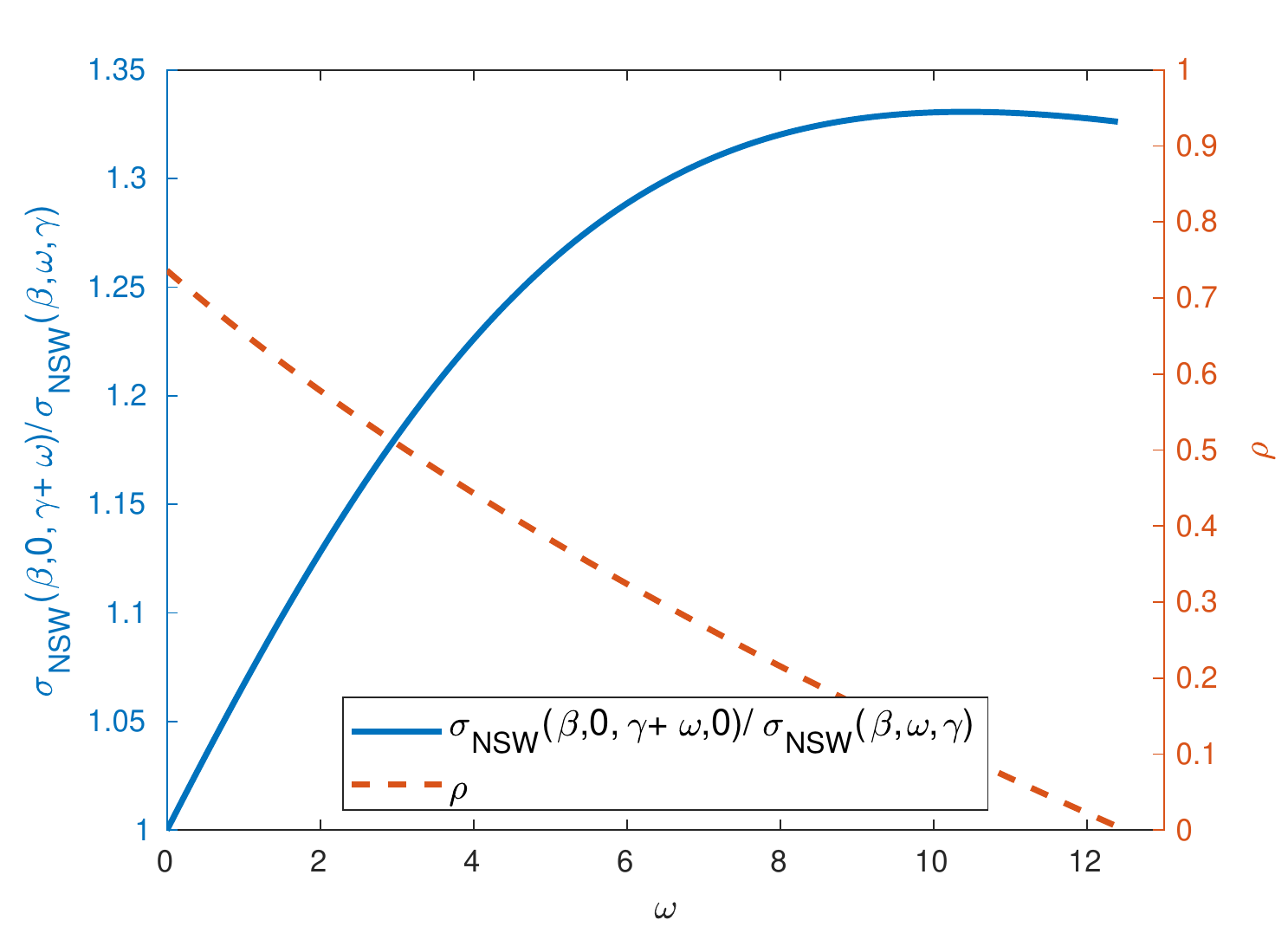}
\end{tabular}
\end{center}
\caption{Scaled asymptotic standard deviations of the final size of a large outbreak in the models $E(\omega,\gamma)$ and $E(0,\gamma+\omega)$ for increasing dropping rate $\omega$, starting from the `base model' with $\beta=3/2$, $\gamma=1$, $\omega=0$ and $\epsilon=0$. The underlying graphs are of the NSW type with (a) $D\sim\mbox{Poi}(5)$ (b) $D\sim\mbox{Geo}(1/6)$. The upper plots show the standard deviations and the lower plot their ratio; all plots also show the relative final size $\rho$ for reference.}
\label{fig:incRecovery2}
\end{figure}

Figure~\ref{fig:incRecoveryTimeCourse} shows how the asymptotic quantities relating to the mean and standard deviation of $S^N(t)$ and $I^N(t)$ compare through time for these models. In the model with dropping we denote the asymptotic mean proportion infected by $\mu^I(t; \beta,\omega,\gamma)$ and the asymptotic scaled standard deviation of $I^N(t)$ by $\sigma^I_{\rm NSW}(t;\beta,\omega,\gamma)$; we let $\mu^S(t; \beta,\omega,\gamma)$ and $\sigma^S_{\rm NSW}(t;\beta,\omega,\gamma)$ denote the corresponding quantities for the number of susceptibles $S^N(t)$. Note that the absolute scale of the standard deviations here is not directly meaningful (to approximate the standard deviation in a population of size $N$ these limiting quantities should be multiplied by $\sqrt{N}$); it is the relative values that are of interest here. Firstly, the upper plots confirm our assertions about the relative numbers of susceptibles in the two models: that the mean (LLN) behaviour of the two models is the same but the model with dropping exhibits less variability (cf.\ the final size behaviour in Figure~\ref{fig:incRecovery1} and Table~\ref{tab:incRecovery1}). In the lower plots the behaviour of the individual models $E(\omega,\gamma)$ and $E(0,\gamma+\omega)$ is broadly in keeping with that observed in Figure~\ref{fig:processApprox}, however the differences between the two models are quite stark. Even though the two models have the same final size they achieve this through very different temporal behaviour: in the $E(0,\gamma+\omega)$ model individuals are infectious for less time but during that time infect others at a higher rate.
\begin{figure}
\begin{center}
\begin{tabular}{cc}
(a) Poisson degree & (b) Geometric degree \\
\includegraphics[width=\hfigwidth]{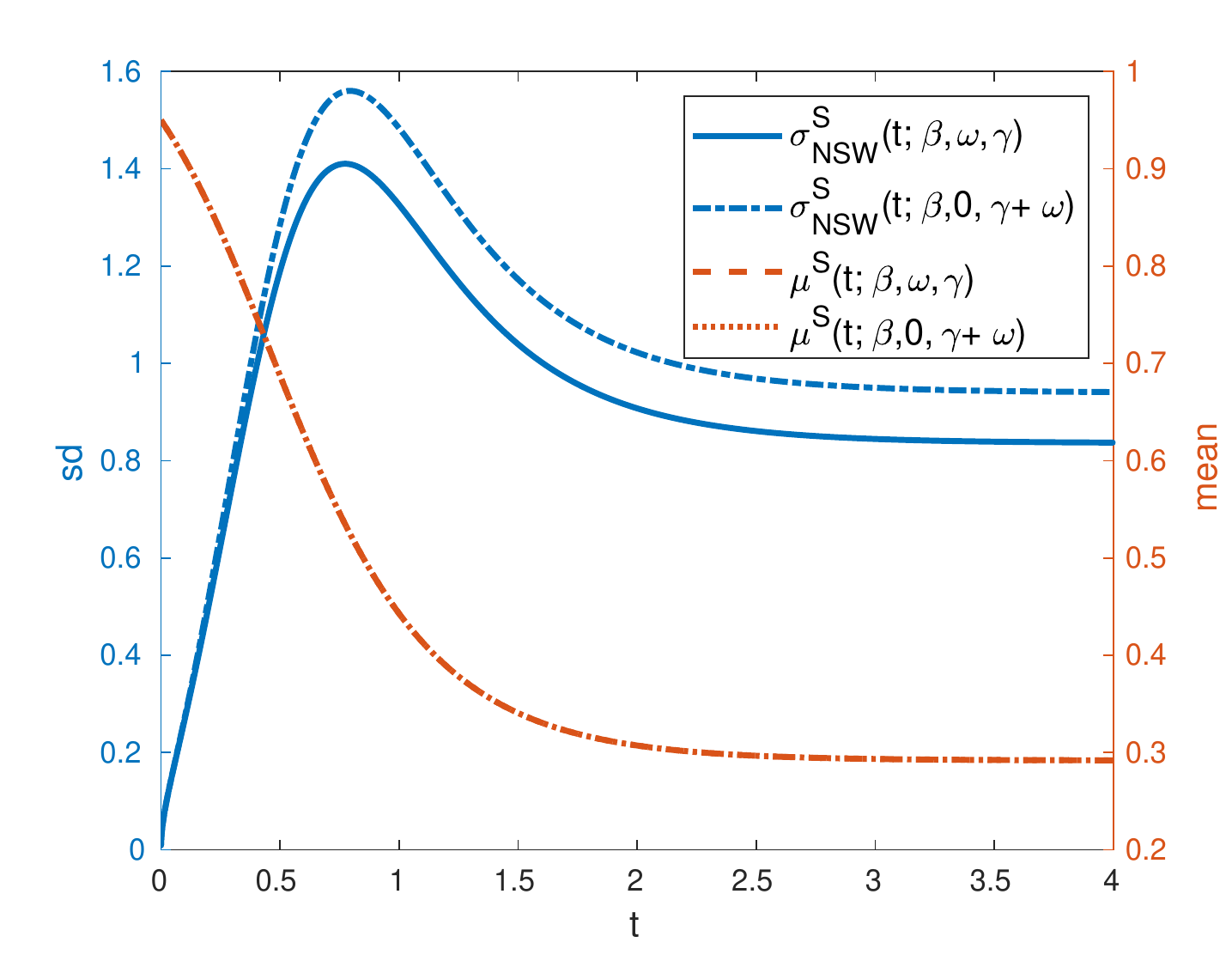} & \includegraphics[width=\hfigwidth]{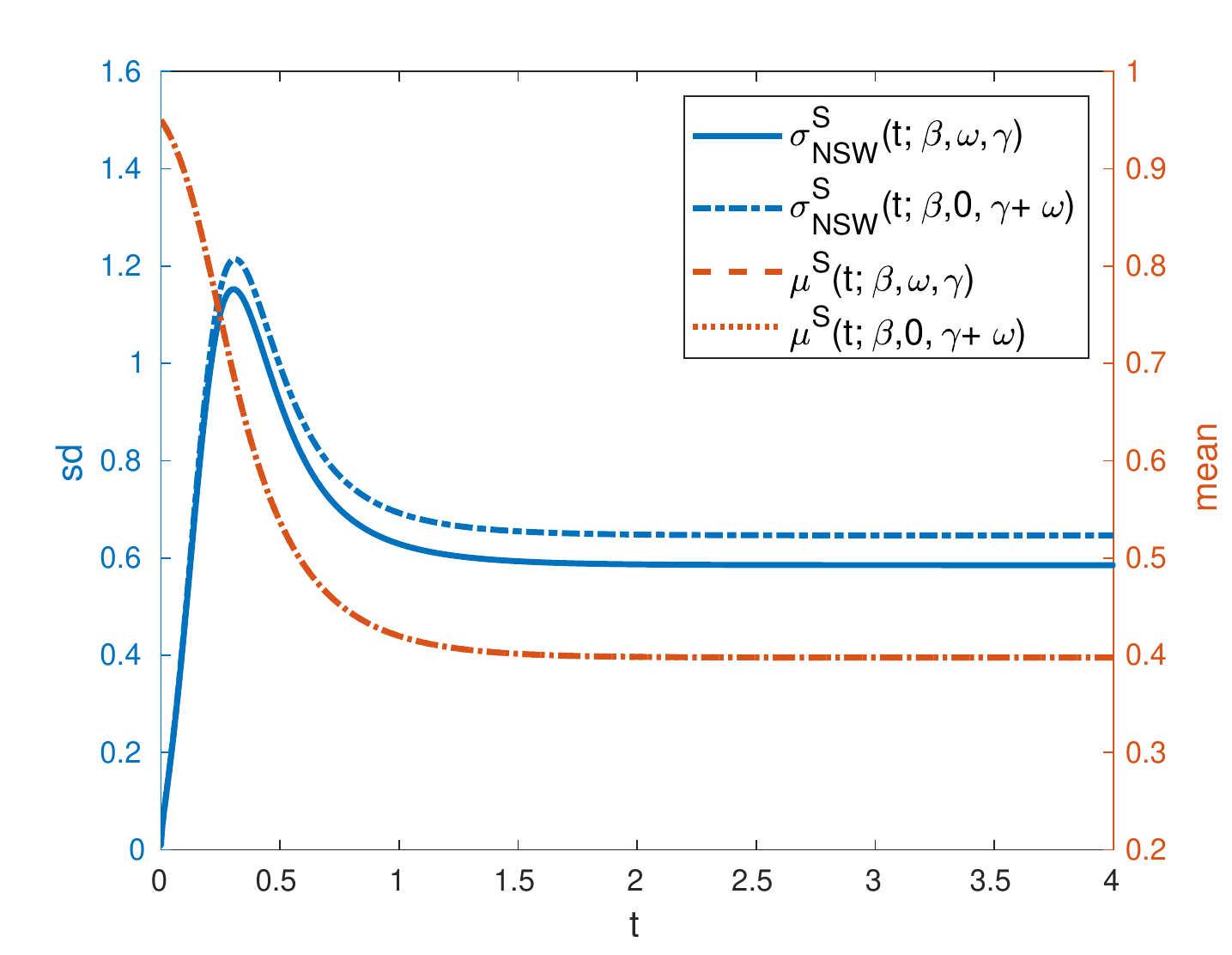} \\
\includegraphics[width=\hfigwidth]{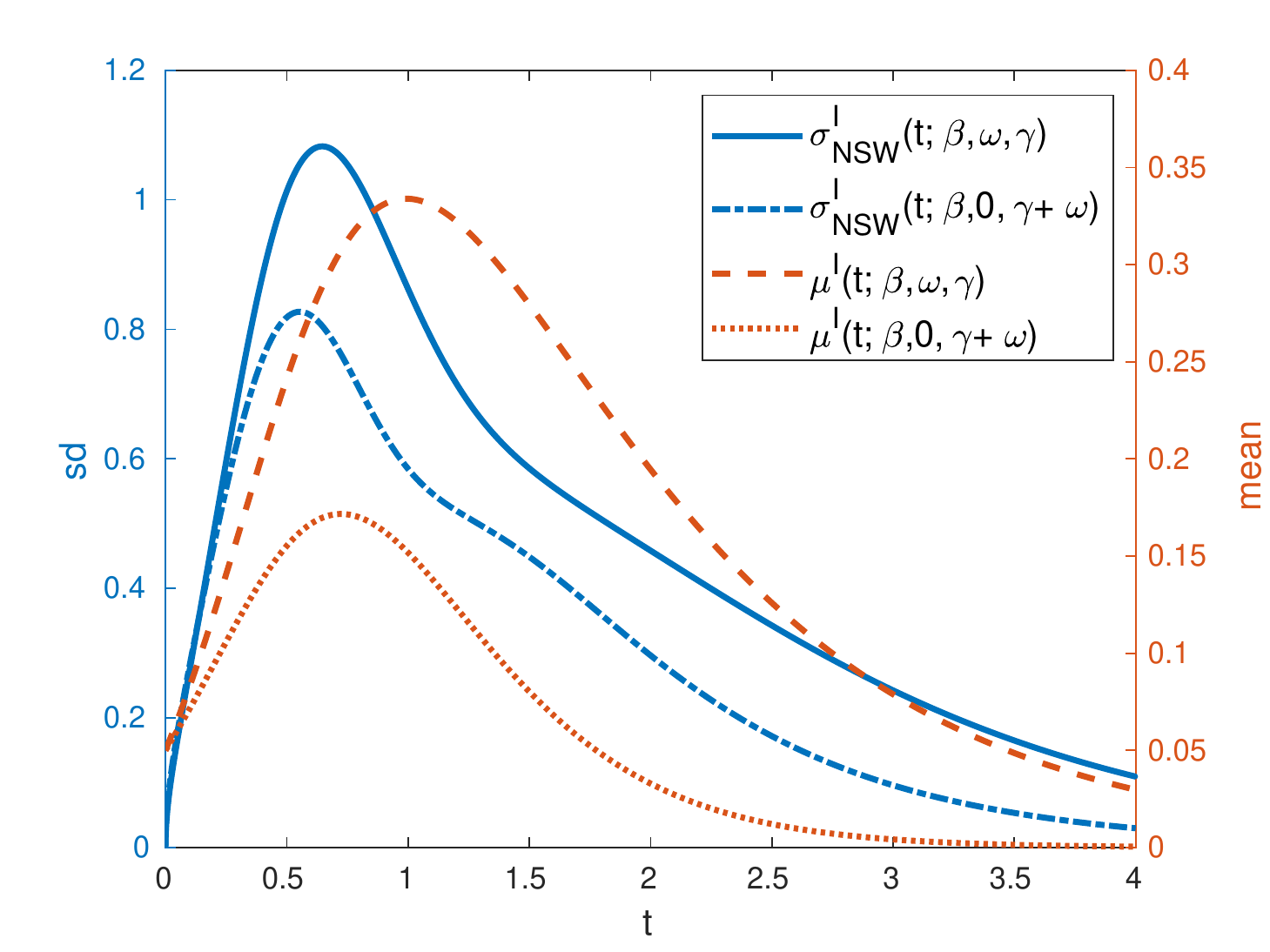} & \includegraphics[width=\hfigwidth]{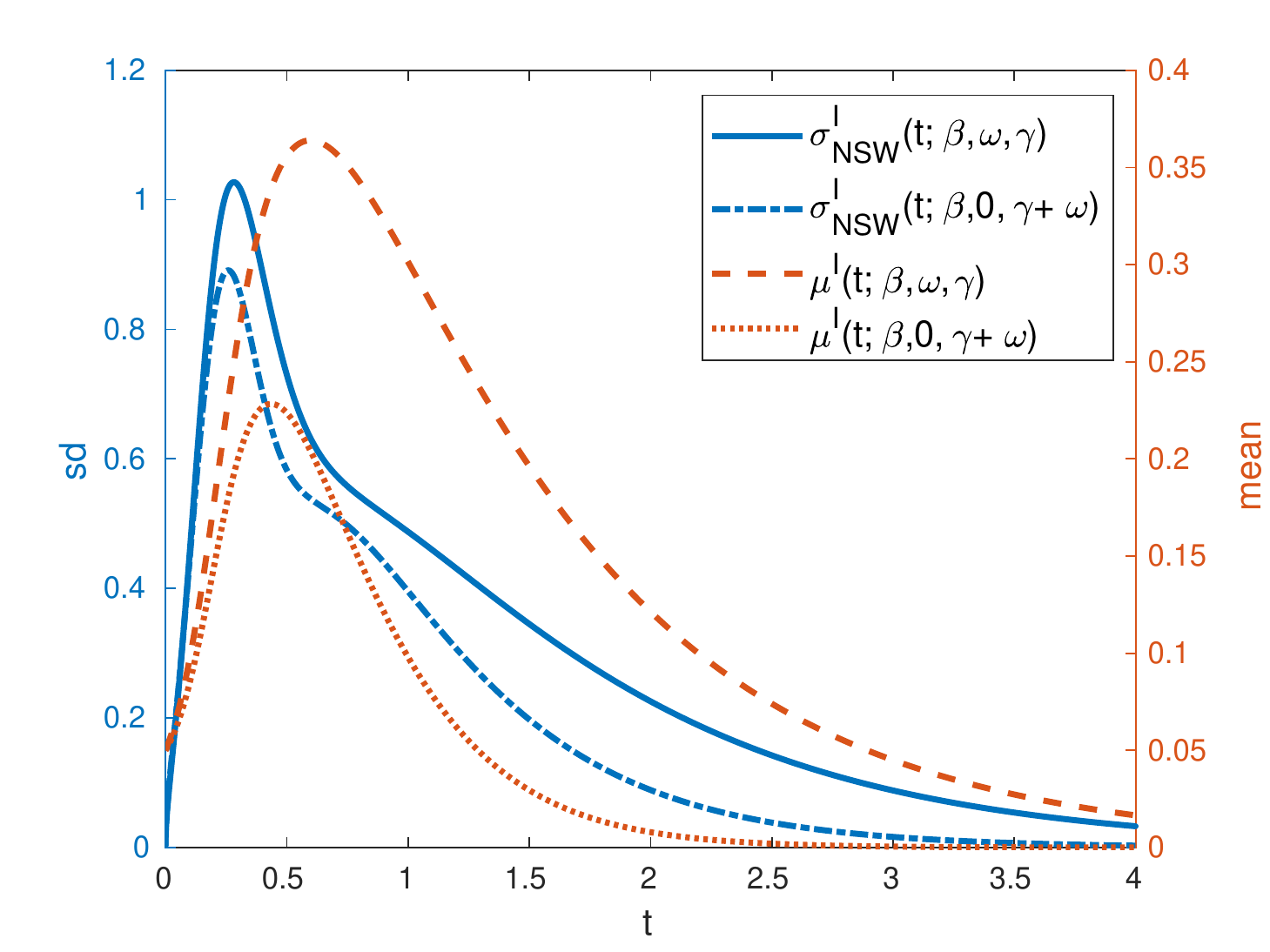}
\end{tabular}
\end{center}
\caption{Scaled asymptotic means and standard deviations of the number of susceptible and infectious individuals through time (in the upper and lower plots, respectively), comparing the model with dropping to that with increased recovery rate. Model parameters are as in Figure~\ref{fig:incRecovery1}, except that $i^N_0=0.05 N$. (Note that in the upper plots the two $\mu^S(\dots)$ quantities are exactly equal.)}
\label{fig:incRecoveryTimeCourse}
\end{figure}

\section{Concluding comments}
\label{sec:conc}
The current paper is concerned with a model for an epidemic taking place on a network in which susceptible individuals may drop their connections to infectious individuals as a preventive measure. A consequence of the behavioural dynamics is that the network changes in time, and the way the network changes depends on the epidemic process taking place on it (sometimes referred to as an adaptive network).  We derive limiting properties of the epidemic process assuming a large outbreak in a large community: the LLN and functional CLT for the epidemic process, as well as conjecture a LLN and CLT for the final number getting infected.
We also give a version of the functional CLT in Ethier and Kurtz~\cite{Ethier:1986}, Chapter 11, which allows for 
\emph{asymptotically random} initial conditions (Theorem~\ref{KurtzFCLTrandinit}).  Although it is a simple extension of
Ethier and Kurtz~\cite{Ethier:1986}, Theorem 11.2.3, we have not seen the result previously in the literature  and 
it (especially the covariance formula~\eqref{covinit}) clearly has interest and applications well beyond the present setting. Furthermore, from the analysis of the dropping model we also obtain results for the Markovian SIR epidemic on a configuration model and for the configuration model giant component. In particular, we  conjecture CLTs, with essentially fully explicit expressions for the asymptotic variances, for the final size of such epidemics on both MR and NSW random graphs, and for the 
size of the giant components of those graphs.

The above LLN and functional CLT are proved under the assumption of bounded degrees.  As noted in Remark~\ref{rmk:MRcltfin1},
the arguments in Ball~\cite{Ball:2018} should yield proofs of the final-size LLN and CLTs under this assumption.  Rigorous extension of these results to networks with unbounded degrees is a natural mathematical next step, though bounded degrees are clearly
sufficient for most biological purposes.

The simulations in Sections~\ref{sec:cgceAppTime} and~\ref{sec:CgceAppFS} show that the limiting approximations kick in for moderate population sizes. Further, from the numerical investigations, dropping of edges seems to have the greatest preventive effect when the basic reproduction number $R_0$ is not too large, more specifically when it is close to the epidemic threshold value of one. In fact, if $R_0$ is moderate in the absence of dropping of edges, a fairly small dropping rate can make the epidemic sub-critical implying that large outbreaks are no longer possible in the large population limit.

This paper is inspired by the model in Britton et al.~\cite{Britton2016}, who study only the initial stages of an outbreak. In the current paper, in order to make progress in the analysis of the complete outbreak, we assume that edges can only be dropped, in contrast to~\cite{Britton2016}, which allows for some of the dropped edges to rewire to other individuals. It would of course be of interest to study limiting properties of this more general dropping/rewiring model. However, the effective degree approach does not apply immediately in a rigorous fashion to this setting, and rigorous analysis of the non-initial stages of the  model including rewiring is left as an open problem. The model with rewiring is considered further in Leung~et al.~\cite{Leung:2018},
where it is demonstrated that such rewiring of edges, although always beneficial to the susceptible individual, can have an adverse effect at the population level. Other possible forms of social distancing include reducing contacts rather than dropping edges completely (e.g.~Viljoen et al.~\cite{Viljoen2014} and Zhang et al.~\cite{Zhang2014}) or only temporarily dropping the edge (e.g.~Althouse and H\'ebert-Dufresne\cite{Althouse2014}). 

Another extension of the current model would be to allow the network to change in time also for reasons other than the epidemic process. One could for example consider some type of dynamic network model as the base network model (e.g.~one of the dynamic network models of Leung and Diekmann~\cite{Leung:2016}), and increase the dropping rate indirectly by decreasing the rate of creation of new edges and/or increasing the rewiring rate between susceptible-infectious pairs of individuals, see e.g.~Reniers and Armbruster~\cite{Reniers2012} for a simulation study where partnership dissolution rates depend on the HIV status of the couple. Obviously, rigorous analysis of such models will be appreciably harder, if indeed possible.

Finally, we note that we have restricted ourselves to the Markovian setting throughout this paper. As always, this assumption is not realistic and is made for mathematical convenience. In the setting of this paper, it is possible to generalize some of our results to include non-exponentially distributed infectious periods. Using a susceptibility set argument, as in e.g.~Ball and Sirl~\cite{Ball:2013}, Section 2.1.2, we can prove results for the deterministic final size similar to Proposition~\ref{prop:final}(b). Specifically, if the infectious period follows a random variable $I$, the deterministic final size is the same as that for a standard  SIR epidemic on a configuration model network in which the infectious period is distributed as $I'=\min(I,W)$, where $W$ is independent of $I$ and has an exponential distribution with rate $\omega$.
Recently, Sherborne et al.~\cite{Sherborne:2018} have extended edge-based compartmental models of epidemics on networks
to allow for non-Markovian transmission and recovery processes, and that methodology should enable the limiting deterministic
model for our model with dropping of edges and non-exponentially distributed infectious periods to be determined, as can be done using the binding site formulation of Leung and Diekmann~\cite{Leung:2016}.  It seems likely that our effective degree approach, together with LLN and functional CLT theorems in Wang~\cite{Wang:1975,Wang:1977} for
age and density dependent population processes, can be used to put such  deterministic models in a fully rigorous asymptotic framework and provide an associated functional CLT.

%\appendix

\setcounter{section}{0}
\renewcommand{\thesection}{\Alph{section}}
\renewcommand\thefigure{\thesection.\arabic{figure}}
\renewcommand\theequation{\thesection.\arabic{equation}}

\section{Derivation of drift function $F(\bx,\by,z_E)$}
\label{app:drift}
In this appendix we derive the expression~\eqref{driftF} for $F(\bx,\by,z_E)$.  First note that~\eqref{lij1} and
~\eqref{intensityfun} yield
\begin{align}
\label{part1}
&\sum_{\bl \in \Delta_1}\bl \beta_{\bl}(\bx,\by,z_E)\nonumber\\
&=\sum_{i=1}^{\infty}\sum_{j=1}^{\infty} \frac{\beta i y_i j x_j}{\eta_E}(-\bei_i+\bei_{i-1}-\bes_j+\bei_{j-1}) \nonumber \\
&=\frac{\beta}{\eta_E}\left[x_E\sum_{i=1}^{\infty}i y_i (-\bei_i+\bei_{i-1})
+y_E\sum_{j=1}^{\infty} j x_j (-\bes_j+\bei_{j-1})\right] \nonumber \\
&=
\frac{\beta}{\eta_E}\sum_{i=0}^{\infty}\left\{  x_E\left[(i+1)y_{i+1}-iy_i\right] \bei_i +y_E \left[-i x_i \bes_i+(i+1)x_{i+1}\bei_i\right]\right\},
\end{align}
~\eqref{lij2} and~\eqref{intensityfun} yield
\begin{eqnarray}
\label{part2}
\sum_{\bl \in \Delta_2}\bl \beta_{\bl}(\bx,\by,z_E)
&=&\sum_{i=1}^{\infty}\sum_{j=1}^{\infty} \frac{(\beta+\omega)i y_i j y_j}{\eta_E}(-\bei_i+\bei_{i-1}-\bei_j+\bei_{j-1}) \nonumber \\
&=&2\frac{(\beta+\omega)}{\eta_E} \sum_{i=1}^{\infty}\sum_{j=1}^{\infty} i y_i j y_j (-\bei_i+\bei_{i-1}) \nonumber \\
&=&2\frac{(\beta+\omega)y_E}{\eta_E} \sum_{i=0}^{\infty}[-iy_i+(i+1)y_{i+1}] \bei_i,
\end{eqnarray}
and~\eqref{lij3} and~\eqref{intensityfun} yield
\begin{align}
\label{part3}
&\sum_{\bl \in \Delta_3}\bl \beta_{\bl}(\bx,\by,z_E)\nonumber\\
&=\sum_{i=1}^{\infty}\sum_{j=1}^{\infty} \frac{\omega i y_i j x_j}{\eta_E}(-\bei_i+\bei_{i-1}-\bes_j+\bes_{j-1}) \nonumber \\
&=\frac{\omega}{\eta_E}\left[x_E\sum_{i=1}^{\infty} i y_i (-\bei_i+\bei_{i-1})
+y_E\sum_{j=1}^{\infty} j x_j (\bes_j+\bes_{j-1})\right] \nonumber \\
&=\frac{\omega}{\eta_E}\sum_{i=0}^{\infty}\left\{ x_E [(i+1)y_{i+1}-iy_i] \bei_i+ y_E[(i+1)x_{i+1}-ix_i] \bes_i \right\}.
\end{align}
Similarly,~\eqref{li1} and~\eqref{intensityfun} yield
\begin{align}
\label{part4}
&\sum_{\bl \in \Delta_4}\bl \beta_{\bl}(\bx,\by,z_E)\nonumber\\
&= \sum_{i=1}^{\infty} \frac{(\beta+\omega)i y_i z_E}{\eta_E}(-\bei_i+\bei_{i-1}-\ber) \nonumber \\
&=-\frac{(\beta+\omega)y_Ez_E}{\eta_E}\ber+\frac{(\beta+\omega)z_E}{\eta_E} \sum_{i=0}^{\infty} [(i+1) y_{i+1}-i y_i]\bei_i,
\end{align}
and~\eqref{li2} and~\eqref{intensityfun} yield
\begin{eqnarray}
\label{part5}
\sum_{\bl \in \Delta_5}\bl \beta_{\bl}(\bx,\by,z_E)
&=& \sum_{i=0}^{\infty} \gamma y_i (-\bei_i+i\ber) \nonumber \\
&=& \gamma y_E \ber- \gamma \sum_{i=0}^{\infty} y_i \bei_i.
\end{eqnarray}
Adding~\eqref{part1} to~\eqref{part5} and recalling that $\eta_E=x_E+y_E+z_E$ gives~\eqref{driftF}.

\section{Application of theorems for density dependent population processes}
\label{app:kurtz}
In this appendix we show that the conditions of the Theorems 11.2.2 and 11.2.3 in Ethier and Kurtz~\cite{Ethier:1986},
Chapter 11, concerning density dependent population processes are satisfied when there is a maximum
degree, $\dmax$ say, and $\rho<1$.  
(Recall that $\rho$ is the fraction of the population that is ultimately infected by the limiting deterministic
model.)  Thus, for $t \ge 0$, 
\begin{equation*}
\bWN(t)=\left(X_0^N(t), X_1^N(t),\ldots,X_{\dmax}^N(t),Y_0^N(t), Y_1^N(t),\ldots,Y_{\dmax}^N(t),Z_E^N(t)\right),
\end{equation*} 
so
$\{\bWN(t)\}$ has dimension $d=2(\dmax+1)+1$. The limiting deterministic process is $\{\bw(t)\}$, where, for $t \ge 0$,
\begin{align*}
\bw(t)&=(x_0(t), x_1(t),\ldots,x_{\dmax}(t), y_0(t), y_1(t),\ldots,y_{\dmax}(t), z_E(t))\\
&=(w_1(t),w_2(t),\ldots,w_d(t)).
\end{align*}
The domain of the intensity functions $\beta_{\bl}(\bw)$ $(\bl \in \Delta)$ is 
\begin{equation*}
H_*=\left\{\bw:w_i \ge 0 \;(i=1,2,\ldots,d),\sum_{i=1}^d w_i \le 1\right\}.
\end{equation*}
The proofs of the theorems in Ethier and Kurtz~\cite{Ethier:1986}, Chapter 11, make it clear that the conditions
need only hold in some small neighbourhood of $\{\bw(t)\}$.  Thus, since $\rho<1$, there exists $\epsilon >0$, so that $H_*$ can be replaced
by $H_*(\epsilon)= \left\{\bw \in H_*: x_E \ge \epsilon\right\}$, where $x_E=\sum_{i=1}^{\dmax} ix_i$.  It follows that the density dependent
condition~\eqref{DDPPcond} is satisfied for all sample paths of $\{\bWN(t)\}$ such that $N^{-1}\bWN(t)$ remains within $H_*(\epsilon)$, which
is sufficient for the proofs in  Ethier and Kurtz~\cite{Ethier:1986}.

Considering first the LLN for $\{\bWN(t)\}$, the conditions of Ethier and Kurtz~\cite{Ethier:1986}, Theorem 11.2.1, are satisfied
if (i) $\sum_{\bl \in \Delta}|\bl|\sup_{\bw \in H_*(\epsilon)}\beta_{\bl}(\bw)<\infty$; (ii) the drift function $F$ is Lipschitz continuous on $H_*(\epsilon)$; and (iii) $\lim_{N\to \infty} N^{-1} \bWN(0)=\bw(0) \ne \bzero$.  It is easily seen from~\eqref{intensityfun} that (i) is satisfied, since $\Delta$ is finite
and $\eta_E \ge x_E \ge \epsilon>0$ for all $\bw \in H_*(\epsilon)$.  It follows from~\eqref{driftF} that the partial derivatives $\partial_j F_i(\bw)$ $(i,j=1,2,\ldots,d)$ are uniformly bounded on $H_*(\epsilon)$, since $\eta_E \ge \epsilon$ for all $\bw \in H_*(\epsilon)$, so (ii) is satisfied.
Finally, it is easily seen from the proof in Ethier and Kurtz~\cite{Ethier:1986} that the result still holds if the convergence in (iii) holds almost surely,
thus the LLN for $\{\bWN(t)\}$, stated in Section~\ref{sec:ED} holds for epidemics on both MR and NSW
random graphs.

Turning to the functional CLT~\eqref{FCLT}, where to be more explict $\Rightarrow$ denotes weak convergence
in the space of right-continuous functions $f:[0,\infty) \to \mathbb{R}^d$ 
having limits from the left (i.e.~c\`{a}dl\`{a}g functions), endowed with the Skorohod metric, the conditions of Ethier and Kurtz~\cite{Ethier:1986}, Theorem 11.2.3, are satisfied if, in addition to (i)-(iii),
(iv) $\sum{\bl \in \Delta}|\bl|^2\sup_{\bw \in H_*(\epsilon)}\beta_{\bl}(\bw)<\infty$; (v) the intensity functions $\beta_{\bl}(\bw)$ $(\bl \in \Delta)$
and the partial derivatives $\partial_j F_i(\bw)$ $(i,j=1,2,\ldots,d)$ are continuous on $H_*(\epsilon)$; and
(vi) $\lim_{N \to \infty} \sqrt{N}\left(N^{-1} \bWN(0)-\bw(0)\right)=\bV(0)$, where $\bV(0)$ is constant.  Now (iv) is satisfied, for similar reasons to (i).
It is easily seen from~\eqref{intensityfun} and~\eqref{driftF} that (v) is satisified, and (vi) follows from~\eqref{initialcond}.  Thus~\eqref{FCLT} is proved.

Consider now the random time-scale transformed process $\{\bWNt(t)\}$ introduced in Section~\ref{sec:final}. The limiting  determinstic
process is now $\{\bwt(t): t \ge 0\}$, where
\[
\bwt(t)=(\xt_0(t),\xt_1(t),\ldots,\xt_{\dmax}(t),\yt_0(t),\yt_1(t),\ldots,\yt_{\dmax}(t),\zEt(t)).
\]
For any $t_0 \in (0,\taut)$, there exists
$\epsilon'>0$ such that $\yEt(t)=\sum_{i=1}^{\dmax} i \yt_i(t) \ge \epsilon'$ for all $0 \le t \le t_0$.
Let $\tilde{H}_*(\epsilon')= \left\{\bwt \in H_*: \yEt \ge \epsilon\right\}$.  The proofs that the conditions of Ethier and Kurtz~\cite{Ethier:1986},Theorems
11.2.1 and 11.2.3, are satisfied for the process $\{\bWNt(t): 0 \le t \le t_0\}$ are analagous to those above, except $H_*(\epsilon)$ is replaced by
$\tilde{H}_*(\epsilon')$.  Note that the denominator in the intensity functions $\tilde{\beta}_{\bl}(\bw)$ $(\bl \in \Delta)$ given at~\eqref{intensityfun1} (and hence in the drift function $\tilde{F}$ given at~\eqref{driftF1}) is $\yEt$, where for the untransformed process it is $\eta_E$.

\section{Properties of $\tautd$}
\label{app:tautdelta}
In this appendix we prove that (i) $\taut<\infty$ and (ii)~\eqref{gradphidotF} holds for all $\delta \in [0,y_E(0))$. Recalling the definition of $\tautd$ at~\eqref{taudelta}, it follows that $\tautd$ is the smallest positive solution of $\tilde y_E(t)=\delta$ with $\tilde y_E(t)$ given by~\eqref{yet} (Clearly, $\tautd=0$ for $\delta>y_E(0)$.) Also, it follows from~\eqref{ataut1},~\eqref{etaEt}, and $\yEt(\tautd)=\delta$ that
\begin{align*}
&\nabla \varphi(\bwt(\tautd)) \cdot \tilde{F}(\bwt(\tautd))\\
&=-(\beta+\omega)\delta+
\re^{-2(\beta+\omega)\tautd}\left[\beta\fde''\left(\psi(\tautd)\right)
-(\beta+\omega+\gamma)\mud\right].
\end{align*}
Let $\zd=\re^{-(\beta+\omega)\tautd}$ and recall that $\psi(\tautd)=p_{\omega}+(1-p_{\omega})\re^{-(\beta+\omega)\tautd}=\psit(zd)$.  Then,
\begin{equation}\label{eq:nabla}
\nabla \varphi(\bwt(\tautd)) \cdot \tilde{F}(\bwt(\tautd))=-(\beta+\omega)\delta+\zd^2
\left[\beta \fde''\left(\psit(\zd)\right)-(\beta+\omega+\gamma)\mud\right].
\end{equation}
and from~\eqref{yet}, if follows that $\zd$ satisfies
\begin{equation}
\label{eq:zd}
\fde'\left(\psit(\zd)\right)-\frac{[(\beta+\omega+\gamma)\zd-\gamma]}{\beta+\omega}\mud=-\frac{\delta}{\zd}.
\end{equation}
For $z \in [0,1]$, let
\begin{equation}
\label{eq:A}
A(z)=\fde'\left(\psit(z)\right)-\frac{[(\beta+\omega+\gamma)z-\gamma]}{\beta+\omega}\mud,
\end{equation}
so $z_0=\re^{-(\beta+\omega)\taut}=\re^{-(\beta+\omega)\taut_0}$ satisfies $A(z_0)=0$. Now $A(0)= \fde'(p_{\omega})+\gamma/(\beta+\omega)$ and $A(1)=\fde'(1)-\mud=-y_E(0)<0$. (Recall the definition of $\fde$ at~\eqref{fdes}.) Further, unless $p_{\omega}=\gamma=\fde'(0)=0$, then $A(0)>0$, so since $A(z)$ is continuous, $z_0 \in (0,1)$ and $\taut$ (and hence also $\tautd$) is finite. For $\delta \in (0,y_E(0))$, note that $\zd$ satisfies $A(\zd)+\frac{\delta}{\zd}=0$. Thus, $A(1)+\frac{\delta}{1}=\delta-y_E(0)<0$ and $A(z)+\frac{\delta}{z} \to \infty$ as $z \downarrow 0$, so $\zd \in(0,1)$ and $\tautd< \infty$.  If $p_{\omega}=\gamma=\fde'(0)=0$ and $\delta=0$, then it is easily verified using the convexity of $\fde'$ that $z_0=0$, so
$\taut=\infty$.
%$\rho=1-\fde(0)=1$ and (see Remark~\ref{rmk:epsE=01}) such cases are excluded in our CLTs.

We show now that $\nabla \varphi(\bwt(\tautd)) \cdot \tilde{F}(\bwt(\tautd))<0$ for $\delta \in [0,y_E(0))$. Differentiating~\eqref{eq:A} and recalling that $p_{\omega}=\frac{\omega}{\beta+\omega}$ yields
\begin{equation*}
A'(z)=\frac{1}{\beta+\omega}\left[\beta \fde''\left(\psit(z)\right)-(\beta+\omega+\gamma)\mud\right].
\end{equation*}
Suppose, for contradiction, that $\nabla \varphi(\bwt(\tautd))\cdot \tilde{F}(\bwt(\tautd)) \ge 0$.  Then, recalling~\eqref{eq:nabla}, $A'(\zd) \ge \frac{\delta}{\zd^2}$, whence $A'(z) > \frac{\delta}{z^2}$ for $z \in [\zd,1]$,
since $A'$ is increasing on $[0,1]$.  It follows from~\eqref{eq:zd}
that $A(\zd)=-\frac{\delta}{\zd}$.  Thus,
\begin{equation*}
A(1) > A(\zd)+ \int_{\zd}^1 \frac{\delta}{z^2}\,{\rm d}z = -\delta.
\end{equation*}
But $A(1)=\fde'(1)-\mud=-y_E(0)$, since, using~\eqref{fdes}, $\fde'(1)=\sum_{k=1}^{\infty}k(p_k - \epsilon_k)=
\mud -y_E(0)$.  Thus, $y_E(0) < \delta$, which is a contradiction as $\delta \in [0,y_E(0))$.  Hence
$\nabla \varphi(\bwt(\tautd)) \cdot \tilde{F}(\bwt(\tautd))<0$, as required.

Finally, suppose that the epidemic is started by a trace of infection, so $y_E=0$, and that $\delta=0$.   Then, Proposition~\ref{prop:final}(b)
%and Remark~\ref{rmk:binding final}
shows that~\eqref{eq:zd} (with $D_{\epsilon}$ replaced by $D$ and $\delta=0$) has a (unique) solution, $z_0$, in $[0,1)$ if and only if $R_0>1$.  Moreover, $z_0>0$ unless $p_{\omega}=\gamma=f_D'(0)=0$. 
%which implies $\rho=1$.  Thus again $\taut<\infty$ in all cases covered by our CLTs.

Further, the above proof is easily modified to show that $\nabla \varphi(\bwt(\taut)) \cdot \tilde{F}(\bwt(\taut))<0$.

\section{Calculations pertaining to $\Phit(t,u)$}
\label{app:phitcalculations}
Expanding~\eqref{Phit} in partitioned form yields, using~\eqref{partialFtilde},
\begin{equation}
\label{PhitXX}
\dfrac{\partial}{\partial t}\Phit_{XX}(t,u)=\partial \tilde{F}_{XX}(\bwt(t))\Phit_{XX}(t,u),
\end{equation}
and, for $A=Y, Z$ and $B=X, Y, Z$,
\begin{align}
\label{PhitAB}
&\dfrac{\partial}{\partial t}\Phit_{AB}(t,u)\nonumber\\
&= \partial \tilde{F}_{AX}(\bwt(t))\Phit_{XB}(t,u)+
\partial \tilde{F}_{AY}(\bwt(t))\Phit_{YB}(t,u)+ \partial \tilde{F}_{AZ}(\bwt(t))\Phit_{ZB}(t,u),
\end{align}
where $\Phit_{XY}(t,u)=0$ and $\Phit_{XZ}(t,u)=\bzero^\top$.

It follows from~\eqref{driftF1} that
\[
\left(\partial \tilde{F}_{XX}(\bwt(t))\right)_{ij}=-\beta i \delta_{i,j}+\omega\left[-i \delta_{i,j}+(i+1)\delta_{i+1,j}\right].
\]
Thus, letting $\tilde{\phi}_{ij}(t,u)$ denote the $(i,j)$th
element of $\Phit_{XX}(t,u)$, it follows from~\eqref{PhitXX} that, for $t \ge u$,
\begin{equation}
\label{phiij}
\dfrac{\partial}{\partial t}\tilde{\phi}_{ij}=-(\beta+\omega)i \tilde{\phi}_{ij}+\omega(i+1)\tilde{\phi}_{i+1,j}\qquad (i=0,1,\ldots),
\end{equation}
with the initial condition $\tilde{\phi}_{ij}(u,u)=\delta_{i,j}$.  For fixed $j$, apart from the initial condition, $\tilde{\phi}_{ij}(t,u)$ $(i=0,1,\dots)$ satisfies the same system of ODEs given at~\eqref{difft1} for $\xt_i$ $(i=0,1,\dots)$, and it follows from~\eqref{pij} that, for $t \ge u$,
\begin{equation}
\label{Phixij}
\tilde{\phi}_{ij}(t,u)=\begin{cases} \binom{j}{i} \re^{-(\beta+\omega)i(t-u)}\left(1-\re^{-(\beta+\omega)(t-u)t}\right)^{j-i} p_{\omega}^{j-i} & \text{ for } j \ge i,\\
0& \text{ for } j < i,
\end{cases}
\end{equation}
so
\begin{eqnarray}
\left(\bone \Phit_{XX}(t,u)\right)_j&=&\sum_{i=0}^j \binom{j}{i} \re^{-(\beta+\omega)i(t-u)}\left(1-\re^{-(\beta+\omega)(t-u)t}\right)^{j-i} p_{\omega}^{j-i} \nonumber\\
&=&\left(p_{\omega}+(1-p_{\omega})\re^{-(\beta+\omega)(t-u)}\right)^j\nonumber\\
&=& \psi(t-u)^j \qquad (j=0,1,\ldots),\label{bonePhitxxj}
\end{eqnarray}
where $\psi(t)$ is defined at~\eqref{eq:psi}.

From~\eqref{driftF1}, the coefficient of $\beii$ in $\tilde{F}(\bx,\by,z_E)$ is
\[
(\beta+\omega)[-i y_i+(i+1) y_{i+1}]\left(1+\frac{\eta_E}{y_E}\right)+\beta (i+1)x_{i+1}-\gamma y_i\frac{\eta_E}{y_E},
\]
so
\[
\left(\partial \tilde{F}_{YX}(\bwt(t))\right)_{ij}=(\beta+\omega)[-i \yt_i(t)+(i+1) \yt_{i+1}(t)]\frac{j}{\yEt(t)}+\beta(i+1)\delta_{i+1,j}-\gamma
\frac{j \yt_i(t)}{\yEt(t)}.
\]
Hence
\[
\sum_{i=1}^{\infty} i \left(\partial \tilde{F}_{YX}(\bwt(t))\right)_{ij}=-(\beta+\omega+\gamma)j+\beta j(j-1) \qquad (j=0,1,\ldots),
\]
so
\begin{eqnarray}
\label{pfyx}
\bp \, \partial \tilde{F}_{YX}(\bwt(t))&=&-(\beta+\omega+\gamma)\bp +\beta \bptwo,
\end{eqnarray}
where $\bptwo=(p_{[2],0}, p_{[2],1}, \ldots)$ with $p_{[2],i}=i(i-1)$ $(i=0,1,\ldots)$.
Similar calculations show that
\begin{eqnarray}
\bp \, \partial \tilde{F}_{YY}(\bwt(t))&=&-[2(\beta+\omega)+\gamma]\bp,\label{pfyy}\\
\bp \, \partial \tilde{F}_{YZ}(\bwt(t))&=&-(\beta+\omega+\gamma),\label{pfyz}\\
\partial \tilde{F}_{ZX}(\bwt(t))&=&\gamma \bp,\label{pfzx}\\
\partial \tilde{F}_{ZY}(\bwt(t))&=&\gamma \bp,\label{pfzy}\\
\partial \tilde{F}_{ZZ}(\bwt(t))&=&\gamma-\beta-\omega.\label{pfzz}\\
\nonumber
\end{eqnarray}

Setting $A=Y$ in~\eqref{PhitAB} and using~\eqref{pfyx}-\eqref{pfyz} yields, for $B=X,Y,Z$,
\begin{align}
\label{pPhitYB}
\dfrac{\partial}{\partial t}\bp \,\Phit_{YB}(t,u)&= -(\beta+\omega+\gamma)\bp \, \Phit_{XB}(t,u)
+\beta \bptwo \, \Phit_{XB}(t,u)\nonumber\\
&\phantom{=\ } -[2(\beta+\omega)+\gamma]\bp \, \Phit_{YB}(t,u)-(\beta+\omega+\gamma) \Phit_{ZB}(t,u).
\end{align}
Setting $A=Z$ in~\eqref{PhitAB} and using~\eqref{pfzx}-\eqref{pfzz} yields, for $B=X,Y,Z$,
\begin{equation}
\label{PhitZB}
\dfrac{\partial}{\partial t} \Phit_{ZB}(t,u)= \gamma\bp \, \Phit_{XB}(t,u)
+\gamma\bp \, \Phit_{YB}(t,u)+(\gamma-\beta-\omega) \Phit_{ZB}(t,u).
\end{equation}

Setting $B=Z$ in~\eqref{pPhitYB} and~\eqref{PhitZB}, and recalling that
$\Phit_{XY}(t,u)$ and $\Phit_{XZ}(t,u)$ are both identically zero, yields
\begin{eqnarray*}
\dfrac{\partial}{\partial t}\bp \,\Phit_{YZ}(t,u)&=&-[2(\beta+\omega)+\gamma]\bp \, \Phit_{YZ}(t,u)
-(\beta+\omega+\gamma) \Phit_{ZZ}(t,u),\\
\dfrac{\partial}{\partial t} \Phit_{ZZ}(t,u)&=&\gamma\bp \, \Phit_{YZ}(t,u)+(\gamma-\beta-\omega) \Phit_{ZZ}(t,u),
\end{eqnarray*}
with initial condition
\[
\bp \,\Phit_{YZ}(u,u)=0 \qquad\mbox{ and }\qquad \Phit_{ZZ}(u,u)=1.
\]
This linear system of two ODEs has solution, for $t \ge u$,
\begin{eqnarray}
\label{pPhityz}
\bp \,\Phit_{YZ}(t,u)&=&-\frac{\beta+\omega+\gamma}{\beta+\omega}\re^{-(\beta+\omega)(t-u)}\left(1-\re^{-(\beta+\omega)(t-u)}\right),\\
\Phit_{ZZ}(t,u)&=&\frac{\beta+\omega+\gamma}{\beta+\omega}\re^{-(\beta+\omega)(t-u)}-\frac{\gamma}{\beta+\omega}\re^{-2(\beta+\omega)(t-u)}.
\nonumber
\end{eqnarray}

Similarly, setting $B=Y$ in~\eqref{pPhitYB} and~\eqref{PhitZB} yields
\begin{eqnarray*}
\dfrac{\partial}{\partial t}\bp \,\Phit_{YY}(t,u)&=&-[2(\beta+\omega)+\gamma]\bp \, \Phit_{YY}(t,u)
-(\beta+\omega+\gamma) \Phit_{ZY}(t,u),\\
\dfrac{\partial}{\partial t} \Phit_{ZY}(t,u)&=&\gamma\bp \, \Phit_{YY}(t,u)+(\gamma-\beta-\omega) \Phit_{ZY}(t,u),
\end{eqnarray*}
with initial condition
\[
\bp \,\Phit_{YY}(u,u)=\bp \qquad\mbox{ and }\qquad \Phit_{ZY}(u,u)=\bzero
\]
and solution, for $t \ge u$,
\begin{eqnarray}
\label{pPhityy}
\bp \,\Phit_{YY}(t,u)&=&\left(\frac{\beta+\omega+\gamma}{\beta+\omega}\re^{-2(\beta+\omega)(t-u)}-
\frac{\gamma}{\beta+\omega}\re^{-(\beta+\omega)(t-u)}\right) \bp,\\
\Phit_{ZY}(t,u)&=&\frac{\gamma}{\beta+\omega}\re^{-(\beta+\omega)(t-u)}\left(1-\re^{-(\beta+\omega)(t-u)}\right)\bp.
\nonumber
\end{eqnarray}

Setting $B=X$ in~\eqref{pPhitYB} and~\eqref{PhitZB} yields
\begin{align}
\dfrac{\partial}{\partial t}\bp \,\Phit_{YX}(t,u)&=-(\beta+\omega+\gamma)\bp \, \Phit_{XX}(t,u)+\beta \bptwo\, \Phit_{XX}(t,u)
\nonumber\\
&\qquad-[2(\beta+\omega)+\gamma]\bp \, \Phit_{YX}(t,u)
-(\beta+\omega+\gamma) \Phit_{ZX}(t,u),\label{dpPhityx}\\
\dfrac{\partial}{\partial t} \Phit_{ZX}(t,u)&=\gamma\bp \, \Phit_{XX}(t,u)+\gamma\bp \, \Phit_{YX}(t,u)+(\gamma-\beta-\omega) \Phit_{ZX}(t,u),
\label{dPhityz}
\end{align}
with initial condition
\begin{equation}
\label{pPhityinit}
\bp \,\Phit_{YX}(u,u)=\bzero \qquad\mbox{ and }\qquad \Phit_{ZX}(u,u)=\bzero.
\end{equation}
Further, using~\eqref{Phixij}, for $j=0,1,\ldots$,
\begin{eqnarray}
\left(\bp \Phit_{XX}(t,u)\right)_j&=&j \re^{-(\beta+\omega)(t-u)}\psi(t-u)^{j-1},\label{pPhixij}\\
\left(\bptwo \Phit_{XX}(t,u)\right)_j&=&j(j-1) \re^{-2(\beta+\omega)(t-u)}\psi(t-u)^{j-2}.\label{p2Phixij}
\end{eqnarray}
Note that~\eqref{dpPhityx}-\eqref{p2Phixij} imply that, for $0 \le u \le t$,
\begin{equation}
\label{pPhity}
\bp \,\Phit_{YX}(t,u)=\bp \,\Phit_{YX}(t-u,0) \qquad\mbox{ and }\qquad \Phit_{ZX}(t,u)=\Phit_{ZX}(t-u,0),
\end{equation}
so we consider the case when $u=0$.

Let
\begin{equation*}
D=\begin{bmatrix}
-2(\beta+\omega)-\gamma &\quad -(\beta+\omega+\gamma) \\
 \gamma &\quad \gamma-\beta-\omega
\end{bmatrix}.
\end{equation*}
Then,
%see, for example, Bellman~\cite{Bellman:1970}, page 173,
\begin{align}
\label{Phiint}
&\begin{pmatrix}
\bp \, \Phit_{YX}(t,0)\\
\Phit_{ZX}(t,0)
\end{pmatrix}\nonumber\\
&=
\int_0^t \re^{-D(t-s)}
\begin{pmatrix}
-(\beta+\omega+\gamma)\bp \, \Phit_{XX}(s,0)+\beta \bptwo\, \Phit_{XX}(s,0)\\
\gamma\bp \, \Phit_{XX}(s,0)
\end{pmatrix}
\,{\rm d}s,
\end{align}
%where $\re^{-Dt}=\sum_{k=0}^{\infty} \frac{t^k D^k}{k!}$ is the usual matrix exponential
%(see, for example, Bellman~\cite{Bellman:1970}, page 169).  Now $X(t)=\re^{-Dt}$ is the solution of
%\[
%\dfrac{dX}{dt}=DX, \qquad X(0)=I,
%\]
%so it is easily verified that
with
\begin{align}
\label{Dexp}
\re^{-Dt}&=\frac{1}{\beta+\omega}\re^{-2(\beta+\omega)t}
\begin{bmatrix}
\beta+\omega+\gamma &\quad \beta+\omega+\gamma\\
-\gamma &\quad -\gamma
\end{bmatrix}\nonumber\\
&\phantom{=\ }+\frac{1}{\beta+\omega}\re^{-(\beta+\omega)t}
\begin{bmatrix}
-\gamma &\quad -(\beta+\omega+\gamma)\\
\gamma &\quad \beta+\omega+\gamma
\end{bmatrix}.
\end{align}

Substituting~\eqref{Dexp} into~\eqref{Phiint} yields, after using~\eqref{pPhixij} and~\eqref{p2Phixij}, that, for $j=0,1,\ldots$,
\begin{equation}
\left(\bp \, \Phit_{YX}(t,0)\right)_j=
I_j^{(1)}(t)+I_j^{(2)}(t)+I_j^{(3)}(t), \label{bPhityxI}
\end{equation}
where
\begin{eqnarray*}
I_j^{(1)}(t)&=&-(\beta+\omega+\gamma)\re^{-2(\beta+\omega)t}\int_0^t j\re^{(\beta+\omega)s}\psi(s)^{j-1}\,{\rm d}s,\\
I_j^{(2)}(t)&=&\frac{\beta(\beta+\omega+\gamma)}{\beta+\omega}\re^{-2(\beta+\omega)t}\int_0^t
j(j-1)\psi(s)^{j-2}\,{\rm d}s,\\
I_j^{(3)}(t)&=&-\frac{\beta \gamma}{\beta+\omega}\re^{-(\beta+\omega)t}\int_0^t
j(j-1)\re^{-(\beta+\omega)s}\psi(s)^{j-2}\,{\rm d}s
\end{eqnarray*}
and, recalling~\eqref{eq:psi}, $\psi(s)= p_{\omega}+(1-p_{\omega})\re^{-(\beta+\omega)s}$.  Integrating by parts,
\begin{align*}
\int_0^t &
j(j-1)\psi(s)^{j-2}\,{\rm d}s\\
&=\left[-\frac{1}{\beta}\re^{(\beta+\omega)s}j\psi(s)^{j-1}\right]_0^t
+
\int_0^t \frac{\beta+\omega}{\beta} j \re^{(\beta+\omega)s}\psi(s)^{j-1}\,{\rm d}s,
\end{align*}
so
\[
I_j^{(2)}=\frac{\beta+\omega+\gamma}{\beta+\omega}\re^{-2(\beta+\omega)t}j\left[1-\re^{(\beta+\omega)t}\psi(t)^{j-1}\right]
-I_j^{(1)}.
\]
Also,
\begin{eqnarray*}
I_j^{(3)}&=&-\frac{\beta \gamma}{\beta+\omega}\re^{-(\beta+\omega)t}j \left[-\frac{1}{(\beta+\omega)(1-p_{\omega})}\psi(s)^{j-1}\right]_0^t\\
&=&\frac{\gamma}{\beta+\omega}\re^{-(\beta+\omega)t}j\left[\psi(t)^{j-1}-1\right].
\end{eqnarray*}
It then follows using~\eqref{bPhityxI} and~\eqref{pPhity} that, for $j=0,1,\ldots$,
\begin{align}
\left(\bp \,\Phit_{YX}(t,u)\right)_j&=I_j^{(1)}(t-u)+I_j^{(2)}(t-u)+I_j^{(3)}(t-u)\\
&=\re^{-(\beta+\omega)(t-u)}\frac{\left((\beta+\omega+\gamma)
\re^{-(\beta+\omega)(t-u)}-\gamma\right)}{\beta+\omega}j\nonumber\\
&\qquad-\re^{-(\beta+\omega)(t-u)}\psi(t-u)^{j-1}j.\label{bonePhityxj}
\end{align}

\section{Calculation of $\SigmaMR(\beta,\omega,\gamma)$}
\label{app:mrvariance}
Recall~\eqref{mrsigma2} for $\SigmaMR(\beta,\omega,\gamma)$,
where $\sigma^2_1,\sigma^2_2,\ldots,\sigma^2_5$ are given by~\eqref{mrsigma2a}.  We first obtain closed-form expressions for
the integrands in the definitions of $\sigma^2_1,\sigma^2_2,\ldots,\sigma^2_5$, then evaluate the integrals as a function
of $\taut$ and finally show that the expression for $\SigmaMR(\beta,\omega,\gamma)$ reduces to that given in Proposition~\ref{prop:mrVar}.

\subsection{Integrands}
We determine the integrands for $\sigma^2_1,\sigma^2_2,\ldots,\sigma^2_5$ in reverse order.
\subsubsection{Integral for $\sigma^2_5$}
For $i=0,1,\ldots$, it follows from~\eqref{li2}, \eqref{boldc} and~\eqref{hRtauu} that
\[
\bc(\taut,u) \bl_i^{(5)}=i[h_R(\taut,u)-h_I(\taut,u)]=-ib(\taut)\re^{-(\beta+\omega)(\taut-u)},
\]
so, using~\eqref{intensityfun1} and recalling~\eqref{etaEt} for $\etaEt(t)$,
\begin{eqnarray*}
\sum_{i=0}^{\infty} (\bc(\taut,u)\bl_i^{(5)})^2 \betat_{\bl}^{(5)}(\bwt(u))&=&
\sum_{i=0}^{\infty} i^2 b(\taut)^2\re^{-2(\beta+\omega)(\taut-u)}\gamma \yt_i(u)\frac{\etaEt(u)}{\yEt(u)}\\
&=&
\gamma \mud b(\taut)^2\re^{-2(\beta+\omega)\taut}\frac{\yEt^{(2)}(u)}{\yEt(u)},
\end{eqnarray*}
where $\yEt^{(2)}(u)=\sum_{i=1}^{\infty} i^2 \yt_i(u)$.  Thus, using~\eqref{mrsigma2a},
\begin{equation}
\label{sigmasq5}
\sigma^2_5=\gamma \mud b(\taut)^2\re^{-2(\beta+\omega)\taut} \int_0^{\taut} \frac{\yEt^{(2)}(u)}{\yEt(u)}\,{\rm d}u.
\end{equation}

\subsubsection{Integral for $\sigma^2_4$}
For $i=1,2,\ldots$, it follows from~\eqref{li1}, \eqref{boldc} and~\eqref{hRtauu} that
\[
\bc(\taut,u) \bl_i^{(4)}=-h_I(\taut,u)-h_R(\taut,u)=b(\taut)\re^{-(\beta+\omega)(\taut-u)}-2h_I(\taut,u),
\]
so, using~\eqref{intensityfun1},
\begin{align*}
&\sum_{i=1}^{\infty} (\bc(\taut,u)\bl_i^{(4)})^2 \betat_{\bl}^{(4)}(\bwt(u))\\
&=\sum_{i=1}^{\infty} \left(b(\taut)\re^{-(\beta+\omega)(\taut-u)}-2h_I(\taut,u)\right)^2(\beta+\omega)i\yt_i(u)\frac{\zEt(u)}{\yEt(u)}\\
&=
(\beta+\omega)\left(b(\taut)\re^{-(\beta+\omega)(\taut-u)}-2h_I(\taut,u)\right)^2 \zEt(u).
\end{align*}
Thus, using~\eqref{mrsigma2a},
\begin{equation}
\label{sigmasq4}
\sigma^2_4=\int_0^{\taut} (\beta+\omega)\left(b(\taut)\re^{-(\beta+\omega)(\taut-u)}-2h_I(\taut,u)\right)^2 \zEt(u) \,{\rm d}u.
\end{equation}

\subsubsection{Integral for $\sigma^2_3$}
For $i,j=1,2,\ldots$, it follows from~\eqref{lij3}, \eqref{boldc} and~\eqref{cStauu} that
\[
\bc(\taut,u) \bl_{ij}^{(3)}=-[2h_I(\taut,u)+\hat{c}_j(\taut,u)],
\]
where $\hat{c}_j(\taut,u)=\tilde{c}_j(\taut,u)-\tilde{c}_{j-1}(\taut,u)$.  Hence, using~\eqref{intensityfun1},
\[
\sum_{i=1}^{\infty}\sum_{j=1}^{\infty} (\bc(\taut,u)\bl_{ij}^{(3)})^2 \betat_{\bl}^{(3)}(\bwt(u))
= \omega \sum_{j=1}^{\infty} \left(2h_I(\taut,u)+\hat{c}_j(\taut,u)\right)^2 j \xt_j(u) ,
\]
and, using~\eqref{mrsigma2a},
\begin{equation}
\label{sigmasq3}
\sigma^2_3=\omega \int_0^{\taut} \sum_{j=1}^{\infty} \left(2h_I(\taut,u)+\hat{c}_j(\taut,u)\right)^2 j \xt_j(u)\,{\rm d}u.
\end{equation}

\subsubsection{Integral for $\sigma^2_2$}
For $i,j=1,2,\ldots$, it follows from~\eqref{lij2} and~\eqref{boldc} that
\[
\bc(\taut,u) \bl_{ij}^{(2)}=-2h_I(\taut,u),
\]
so, using~\eqref{intensityfun1},
\[
\sum_{i=1}^{\infty}\sum_{j=1}^{\infty} (\bc(\taut,u)\bl_{ij}^{(2)})^2 \betat_{\bl}^{(2)}(\bwt(u))
=4 h_I(\taut,u)^2 (\beta+\omega) \yEt(u),
\]
and, using~\eqref{mrsigma2a},
\begin{equation}
\label{sigmasq2}
\sigma^2_2=4(\beta+\omega) \int_0^{\taut}  h_I(\taut,u)^2  \yEt(u) \,{\rm d}u.
\end{equation}

\subsubsection{Integral for $\sigma^2_1$}

For $i,j=1,2,\ldots$, it follows from~\eqref{lij1}, \eqref{boldc} and~\eqref{cStauu} that
\[
\bc(\taut,u) \bl_{ij}^{(1)}=-[2h_I(\taut,u)+\tilde{c}_j(\taut,u)],
\]
so, using~\eqref{intensityfun1},
\[
\sum_{i=1}^{\infty}\sum_{j=1}^{\infty} (\bc(\taut,u)\bl_{ij}^{(1)})^2 \betat_{\bl}^{(1)}(\bwt(u))
= \beta \sum_{j=1}^{\infty} \left(2h_I(\taut,u)+\tilde{c}_j(\taut,u)\right)^2 j \xt_j(u) ,
\]
and, using~\eqref{mrsigma2a},
\begin{equation}
\label{sigmasq1}
\sigma^2_1=\beta \int_0^{\taut} \sum_{j=1}^{\infty} \left(2h_I(\taut,u)+\tilde{c}_j(\taut,u)\right)^2 j \xt_j(u)\,{\rm d}u.
\end{equation}

\subsection{Evaluation of integrals}
Recall that $\etaEt(u)=\xEt(u)+\yEt(u)+\zEt(u)$.  Then adding~\eqref{sigmasq4}-\eqref{sigmasq1} gives,
\begin{equation}
\label{sigmasum}
\sum_{i=1}^4 \sigma_i^2=\sum_{i=1}^7 I_i,
\end{equation}
where
\begin{eqnarray}
I_1&=& 4(\beta+\omega)\int_0^{\taut} h_I(\taut,u)^2 \etaEt(u) \,{\rm d}u,\label{integralI1}\\
I_2&=& -4(\beta+\omega)b(\taut)\int_0^{\taut} h_I(\taut,u) \re^{-(\beta+\omega)(\taut-u)} \zEt(u) \,{\rm d}u,\label{integralI2}\\
I_3&=& (\beta+\omega)b(\taut)^2\int_0^{\taut} \re^{-2(\beta+\omega)(\taut-u)} \zEt(u) \,{\rm d}u,\label{integralI3}\\
I_4&=& 4 \omega \int_0^{\taut} h_I(\taut,u)\sum_{j=1}^{\infty} \hat{c}_j(\taut,u)j \xt_j(u) \,{\rm d}u,\label{integralI4}\\
I_5&=& \omega \int_0^{\taut} \sum_{j=1}^{\infty} \hat{c}_j(\taut,u)^2j \xt_j(u) \,{\rm d}u,\label{integralI5}\\
I_6&=& 4 \beta \int_0^{\taut} h_I(\taut,u)\sum_{j=1}^{\infty} \tilde{c}_j(\taut,u)j \xt_j(u) \,{\rm d}u,\label{integralI6}\\
I_7&=& \beta \int_0^{\taut} \sum_{j=1}^{\infty} \tilde{c}_j(\taut,u)^2j \xt_j(u) \,{\rm d}u.\label{integralI7}
\end{eqnarray}

Recalling~\eqref{etaEt}, \eqref{zet} and~\eqref{hItauu}, allows us to evaluate immediately $I_1, I_2$ and $I_3$:
\begin{align}
I_1&=\frac{4\mud b(\taut)^2 \re^{-2(\beta+\omega)\taut}}{\beta+\omega}\left[\gamma^2 \taut -\frac{2\gamma(\beta+\omega+\gamma)}{\beta+\omega}\left(1-\re^{-(\beta+\omega)\taut}\right)\right.\nonumber\\
&\qquad \qquad\left.+\frac{(\beta+\omega+\gamma)^2}{2(\beta+\omega)}\left(1-\re^{-2(\beta+\omega)\taut}\right)
\right], \label{integralI1a}\\
I_2&=-4\frac{\gamma \mud b(\taut)^2\re^{-(\beta+\omega)\taut}}{\beta+\omega}
\left[\gamma \taut \re^{-(\beta+\omega)\taut}\right.\nonumber\\
&\qquad\qquad\left.-\frac{(\beta+\omega+\gamma)\re^{-(\beta+\omega)\taut}+\gamma}{\beta+\omega}
\left(1-\re^{-(\beta+\omega)\taut}\right)\right.\nonumber\\
&\qquad\qquad\left.+\frac{\beta+\omega+\gamma}{2(\beta+\omega)}\left(1-\re^{-2(\beta+\omega)\taut}\right)\right],\label{integralI2a}\\
I_3&=\frac{\gamma \mud b(\taut)^2\re^{-(\beta+\omega)\taut}}{\beta+\omega}
\left\{1-\re^{-(\beta+\omega)\taut}\left[1+(\beta+\omega)\taut\right]\right\}. \label{integralI3a}
\end{align}

For $j,k=0,1,\ldots$, let $j_{[k]}=j(j-1)\ldots(j-k+1)$ denote a falling factorial, with the convention
that $j_{[0]}=1$.  To calculate $I_4, I_5, I_6$ and $I_7$, observe first using~\eqref{xit} that, for $\theta \in [0,1]$ and $k=1,2,\ldots$,
\begin{eqnarray}
\label{expansion}
\sum_{j=1}^{\infty} j_{[k]} \xt_j(u) \theta^{j-k}
&=&\sum_{j=k}^{\infty} \frac{j!}{(j-k)!}\theta^{j-k}\frac{ \re^{-(\beta+\omega)ju}}{j!}\fde^{(j)}\left(p_{\omega}\left[1-\re^{-(\beta+\omega)u}\right]\right)\nonumber\\
&=&\re^{-(\beta+\omega)ku}\sum_{j=k}^{\infty}\frac{\left[\theta \re^{-(\beta+\omega)u}\right]^{j-k}}{(j-k)!}\fde^{(j)}\left(p_{\omega}\left[1-\re^{-(\beta+\omega)u}\right]\right)\nonumber\\
&=&\re^{-k(\beta+\omega)u}\fde^{(k)}\left(\theta\re^{-(\beta+\omega)u}+p_{\omega}\left[1-\re^{-(\beta+\omega)u}\right]\right),
\end{eqnarray}
and that
\[
\re^{-(\beta+\omega)u}\psi(\taut-u)+p_{\omega}\left[1-\re^{-(\beta+\omega)u}\right]
=\psi(\taut).
\]
Thus, using~\eqref{expansion} with $\theta=\psi(\taut-u)$ and $k=1,2$,
\begin{align*}
&\sum_{j=1}^{\infty}\tilde{c}_j(\taut,u)j \xt_j(u)\\
&= \psi(\taut-u)\re^{-(\beta+\omega)u}\fde'(\psi(\taut))-b(\taut)\re^{-(\beta+\omega)(\taut-u)}
\left[\re^{-(\beta+\omega)u}\fde'(\psi(\taut))\right.\\
&\qquad\left.+\psi(\taut-u)\re^{-2(\beta+\omega)u}\fde^{(2)}(\psi(\taut))\right]\\
&= \psi(\taut-u)\re^{-(\beta+\omega)u}\left[\fde'(\psi(\taut))-b(\taut)\re^{-(\beta+\omega)\taut}\fde^{(2)}(\psi(\taut))\right]\\
&\phantom{=\ }-b(\taut)\re^{-(\beta+\omega)\taut}\fde'(\psi(\taut))
\end{align*}
and
\begin{align*}
&\sum_{j=1}^{\infty} \tilde{c}_{j-1}(\taut,u)j \xt_j(u)\\
&=\sum_{j=1}^{\infty} j \xt_j(u)\psi(\taut-u)^{j-1}-b(\taut)\re^{-(\beta+\omega)(\taut-u)}\sum_{j=2}^{\infty}j(j-1)\xt_j(u)\psi(\taut-u)^{j-2}\\
&=\re^{-(\beta+\omega)u}\left[\fde'(\psi(\taut))-b(\taut)\re^{-(\beta+\omega)\taut}\fde^{(2)}(\psi(\taut))\right].
\end{align*}
Hence, recalling~\eqref{hItauu},
%\begin{align*}
%I_4+I_6&=4 (\beta+\omega) \left[\fde'(\psi(\taut))-b(\taut)\re^{-(\beta+\omega)\taut}\fde^{(2)}(\psi(\taut))\right] \int_0^{\taut}
%\psi(\taut-u)\re^{-(\beta+\omega)u}h_I(\taut,u) \,{\rm d}u\\
%&\qquad-4 (\beta+\omega) b(\taut)\re^{-(\beta+\omega)\taut}\fde'(\psi(\taut)) \int_0^{\taut} h_I(\taut,u) \,{\rm d}u\\
%&\qquad -4 \omega \left[\fde'(\psi(\taut))-b(\taut)\re^{-(\beta+\omega)\taut}\fde^{(2)}(\psi(\taut))\right]
%\int_0^{\taut} \re^{-(\beta+\omega)u}h_I(\taut,u) \,{\rm d}u.
%\end{align*}
%Further, $(\beta+\omega)\psi(\taut-u)\re^{-(\beta+\omega)u}=\omega \re^{-(\beta+\omega)u}+\beta \re^{-(\beta+\omega)\taut}$, so
\begin{align}
\label{integralI4p6}
I_4+I_6&=4\left\{ \beta \left[\fde'(\psi(\taut))-b(\taut)\re^{-(\beta+\omega)\taut}\fde^{(2)}(\psi(\taut))\right]\re^{-(\beta+\omega)\taut}\right.\nonumber\\
&\phantom{=\ }\left.-4 (\beta+\omega) b(\taut)\re^{-(\beta+\omega)\taut}\fde'(\psi(\taut))\right\} I_8,
\end{align}
where
\begin{align}
\label{integralI8}
I_8&=\int_0^{\taut} h_I(\taut,u) \,{\rm d}u \nonumber\\
&= -\frac{b(\taut)}{\beta+\omega}\left[\frac{\gamma\left(1-\re^{-(\beta+\omega)\taut}\right)}
{\beta+\omega}-\frac{(\beta+\omega+\gamma)\left(1-\re^{-2(\beta+\omega)\taut}\right)}{2(\beta+\omega)}\right].
\end{align}

Turning to $I_5$ and $I_7$, note that
\begin{equation}
\label{cjtilde}
\tilde{c}_j(\taut,u)=\psi(\taut-u)^{j-1}\left(\psi(\taut-u)-b(\taut)j\re^{-(\beta+\omega)(\taut-u)}\right),
\end{equation}
so
\begin{align*}
\sum_{j=1}^{\infty} \tilde{c}_j(\taut,u)^2j \xt_j(u)&=\psi(\taut-u)^2 S_1(\taut,u)
-2b(\taut)\psi(\taut-u)\re^{-(\beta+\omega)(\taut-u)}S_2(\taut,u)\\
&\qquad+b(\taut)^2 \re^{-2(\beta+\omega)(\taut-u)}S_3(\taut,u),
\end{align*}
where
\[
S_k(\taut,u)=\sum_{j=1}^{\infty}\psi(\taut-u)^{2(j-1)} j^k \xt_j(u) \qquad (k=1,2,3).
\]

Let
\begin{equation}
\label{psi2}
\psi_2(\taut,u)=\re^{-(\beta+\omega)u}\psi(\taut-u)^2 +p_{\omega}\left(1-\re^{-(\beta+\omega)u}\right).
\end{equation}
Then, since $j^2=j_{[2]}+j$ and $j^3=j_{[3]}+3j_{[2]}+j$, it follows using~\eqref{expansion} that
\begin{align*}
S_1(\taut,u)&= \re^{-(\beta+\omega)u}\fde'(\psi_2(\taut,u)),\\
S_2(\taut,u)&= \psi(\taut-u)^2 \re^{-2(\beta+\omega)u}\fde^{(2)}(\psi_2(\taut,u))+ \re^{-(\beta+\omega)u}\fde'(\psi_2(\taut,u)),\\
S_3(\taut,u)&= \psi(\taut-u)^4 \re^{-3(\beta+\omega)u}\fde^{(3)}(\psi_2(\taut,u))\\
&\qquad+3\psi(\taut-u)^2 \re^{-2(\beta+\omega)u}\fde^{(2)}(\psi_2(\taut,u))+\re^{-(\beta+\omega)u}\fde'(\psi_2(\taut,u)),
\end{align*}
whence
\begin{align}
\label{sumctilde2}
&\sum_{j=1}^{\infty} \tilde{c}_j(\taut,u)^2j \xt_j(u)\nonumber\\
&=\left[\psi(\taut-u)-b(\taut)\re^{-(\beta+\omega)(\taut-u)}\right]^2
\re^{-(\beta+\omega)u} \fde'(\psi_2(\taut,u))\nonumber\\
&\quad+b(\taut) \psi(\taut-u)^2 \re^{-(\beta+\omega)(\taut+u)}\left[3b(\taut)\re^{-(\beta+\omega)(\taut-u)}-2\psi(\taut-u)\right]
\fde^{(2)}(\psi_2(\taut,u))\nonumber\\
&\quad+b(\taut)^2 \psi(\taut-u)^4 \re^{-(\beta+\omega)(2\taut+u)} \fde^{(3)}(\psi_2(\taut,u)).
\end{align}

Further,~\eqref{cjtilde} implies
\begin{align*}
\hat{c}_j(\taut,u)&=\psi(\taut-u)^{j-2}\left\{\psi(\taut-u)\left[\psi(\taut-u)-1-b(\taut)\re^{-(\beta+\omega)(\taut-u)}\right]\right.\\
&\left. \qquad-(j-1)b(\taut)(\psi(\taut-u)-1)\re^{-(\beta+\omega)(\taut-u)}\right\},
\end{align*}
so
\begin{align}
\label{sumchat2}
&\sum_{j=1}^{\infty} \hat{c}_j(\taut,u)^2j \xt_j(u)\nonumber\\
&=\left[\psi(\taut-u)-1-b(\taut)\re^{-(\beta+\omega)(\taut-u)}\right]^2 \re^{-(\beta+\omega)u} \fde'(\psi_2(\taut,u))\nonumber\\
&\phantom{=\ }-2b(\taut)\psi(\taut-u)(\psi(\taut-u)-1)\Big[\psi(\taut-u)-1\nonumber\\
&\phantom{=\ }\qquad\left.-b(\taut)\re^{-(\beta+\omega)(\taut-u)}\right]
\re^{-(\beta+\omega)(\taut+u)}\fde^{(2)}(\psi_2(\taut,u))\nonumber\\
&\phantom{=\ }+b(\taut)^2(\psi(\taut-u)-1)^2\left[\psi(\taut-u)^2 \re^{-(\beta+\omega)(2\taut+u)}\fde^{(3)}(\psi_2(\taut,u))\right.\nonumber\\
&\phantom{=\ }\qquad\left.+\re^{-2(\beta+\omega)\taut}\fde^{(2)}(\psi_2(\taut,u))\right],
\end{align}

To calculate the integral in~\eqref{sigmasq5} for $\sigma^2_5$, let
\begin{equation}
\label{yet2}
\yEt^{[2]}(t)=\sum_{i=2}^{\infty} i(i-1)\yt_i(t)
\end{equation}
and note that
\[
\sum_{i=2}^{\infty} \left[(i+1)i(i-1)\yt_{i+1}(t)-i^2(i-1)\yt_i(t)\right]=-2\yEt^{[2]}(t).
\]
Multiplying~\eqref{difft2} by $i(i-1)$ and summing over $i=2,3,\ldots$ yields, after recalling
~\eqref{etaEt} and invoking~\eqref{expansion} with $\theta=1$ and $k=3$, that
 \[
 \dfrac{d\yEt^{[2]}}{dt}+2(\beta+\omega)\yEt^{[2]}=-\mud\re^{-2(\beta+\omega)t}[2(\beta+\omega)+\gamma]
 \frac{\yEt^{[2]}}{\yEt}+\beta\re^{-3(\beta+\omega)t}\fde^{(3)}\left(\psi(t)\right),
 \]
so
\begin{eqnarray*}
\dfrac{d}{dt}\left(\re^{2(\beta+\omega)t}\yEt^{[2]}(t)\right)&=&
-\mud[2(\beta+\omega)+\gamma]\frac{\yEt^{[2]}(t)}{\yEt(t)}+\beta\re^{-(\beta+\omega)t}
\fde^{(3)}\left(\psi(t)\right)\\
&=&-\mud[2(\beta+\omega)+\gamma]\frac{\yEt^{[2]}(t)}{\yEt(t)}
-\dfrac{d}{dt}\left[\fde^{(2)}\left(\psi(t)\right)\right],
\end{eqnarray*}
since $p_{\omega}=\frac{\omega}{\beta+\omega}$.
Thus,
\begin{eqnarray*}
\int_0^{\taut} \frac{\yEt^{[2]}(u)}{\yEt(u)}\,{\rm d}u&=&
\left[-\frac{1}{\mud[2(\beta+\omega)+\gamma]}\left(\re^{2(\beta+\omega)u}\yEt^{[2]}(u)+
\fde^{(2)}\left(\psi(u)\right)\right)\right]_0^{\taut}\\
&=&\frac{1}{\mud[2(\beta+\omega)+\gamma]}\left[\yEt^{[2]}(0)+\fde^{(2)}(1)-
\fde^{(2)}\left(\psi(\taut)\right)\right],
\end{eqnarray*}
as $\yt_i(\taut)=0$ $(i=1,2,\ldots)$.  Setting $t=0$ in~\eqref{yet2} gives $\yEt^{[2]}(0)=\sum_{i=2}^{\infty} i(i-1)\epsilon_i$ and differentiating~\eqref{fdes} twice yields $\fde^{(2)}(1)=\sum_{i=2}^{\infty}i(i-1)(p_i-\epsilon_i)$.  Thus, $\yEt^{[2]}(0)+\fde^{(2)}(1)=\sum_{i=2}^{\infty}i(i-1)p_i=f_D''(1)$.
Further, $\yEt^{(2)}(u)=\yEt^{[2]}(u)+\yEt(u)$, so $\int_0^{\taut} \frac{\yEt^{(2])}(u)}{\yEt(u)}\,{\rm d}u
=\taut+\int_0^{\taut} \frac{\yEt^{[2]}(u)}{\yEt(u)}\,{\rm d}u$ and using~\eqref{sigmasq5},
\begin{equation}
\label{sigmasq5b}
\sigma^2_5=\gamma \mud b(\taut)^2\re^{-2(\beta+\omega)\taut}\left\{\taut+
\frac{1}{\mud[2(\beta+\omega)+\gamma]}\left[f_D''(1)-
\fde^{(2)}\left(\psi(\taut)\right)\right]\right\}.
\end{equation}

\subsection{Expression for $\SigmaMR(\beta,\omega,\gamma)$}
We now use~\eqref{mrsigma2}, \eqref{sigmasum} and~\eqref{sigmasq5b} to obtain an expression for $\SigmaMR(\beta,\omega,\gamma)$.

Let $z=\re^{-(\beta+\omega)\taut}$. Then, since $\taut$ is the unique solution in $(0,\infty)$ of~\eqref{finaltau}, $z$ is the unique solution in $[0,1)$ of~\eqref{zdef}.  Recall that $\psit(z)=p_{\omega}+(1-p_{\omega})z$.  It then follows from~\eqref{ataut} that
\[
a(\taut)=z^2\left[\beta\fde''\left(\psit(z)\right)-(\beta+\omega+\gamma)\mud\right],
\]
whence, using $b(\taut)=a(\taut)^{-1}\beta \xEt(\taut)$ and~\eqref{xet}
\begin{eqnarray}
\label{bttaut}
b(\taut)&=&\frac{\beta z\fde'\left(\psit(z)\right)}{z^2 \left[\beta \fde''\left(\psit(z)\right)-(\beta+\omega+\gamma)\mud\right]}\nonumber\\
&=&\frac{\beta\left[\frac{(\beta+\omega+\gamma)z-\gamma}{\beta+\omega}\right] \mud}
{z\left[\beta \fde''\left(\psit(z)\right)-(\beta+\omega+\gamma)\mud\right]},
\end{eqnarray}
using~\eqref{zdef}, so $b(\taut)=\btz$ defined at~\eqref{btz}.  For future reference, note that~\eqref{bttaut} implies
\begin{eqnarray}
b(\taut)z\fde''\left(\psit(z)\right)&=&\fde'\left(\psit(z)\right)
+\frac{(\beta+\omega+\gamma)}{\beta}zb(\taut)\mud\label{btautzfd2a}\\
&=&\left[(\beta+\omega+\gamma)\left(\frac{1}{\beta+\omega}
+\frac{b(\taut)}{\beta}\right)z-\frac{\gamma}{\beta+\omega}\right]\mud\label{btautzfd2b}.
\end{eqnarray}

Adding~\eqref{integralI1a} and ~\eqref{integralI2a} yields, after substituting $z=\re^{-(\beta+\omega)\taut}$,
\begin{align}
\label{I1pI2}
I_1+I_2&=2\frac{\mud b(\taut)^2 z(1-z)}{(\beta+\omega)^2}
\Big[\gamma(\gamma-\beta-\omega)\nonumber\\
&\phantom{=\ }\qquad+(\beta+\omega+\gamma)(\beta+\omega-2\gamma)z+(\beta+\omega+\gamma)^2 z^2\Big].
\end{align}
Using~\eqref{integralI4p6} and \eqref{btautzfd2a},
\begin{eqnarray*}
I_4+I_6&=&-4\left[(\beta+\omega+\gamma)z^2 b(\taut)\mud-4(\beta+\omega)z b(\taut)\fde'\left(\psit(z)\right)\right]I_8\\
&=&-4z b(\taut)\mud\left[(2(\beta+\omega+\gamma)z-\gamma\right]I_8,
\end{eqnarray*}
using~\eqref{zdef}.  Substituting for $I_8$ from~\eqref{integralI8} and rearranging then gives
\begin{equation}
\label{I4pI6}
I_4+I_6=2\frac{b(\taut)^2z(1-z)}{(\beta+\omega)^2}\left[2(\beta+\omega+\gamma)z-\gamma\right]
\left[\gamma-\beta-\omega-(\beta+\omega+\gamma)z\right]\mud.
\end{equation}
Adding~\eqref{I1pI2} and~\eqref{I4pI6} yields after a little algebra that
\begin{align}
\label{I4pI6a}
&I_1+I_2+I_4+I_6\nonumber\\
&=2\frac{(\beta+\omega+\gamma)[\gamma-\beta-\omega-(\beta+\omega+\gamma)z]}{(\beta+\omega)^2}\mud b(\taut)^2 z^2(1-z).
\end{align}

Adding~\eqref{integralI3a} and~\eqref{sigmasq5b} yields
\[
I_3+\sigma^2_5=\frac{\gamma}{\beta+\omega}\mud b(\taut)^2z(1-z)+\frac{\gamma}{2(\beta+\omega)+\gamma}b(\taut)^2 z^2\left[f_D''(1)-\fde''\left(\psit(z)\right)\right].
\]
Substituting from~\eqref{btautzfd2b} and noting that
$f_D''(1)=\sigma_D^2+\mud^2-\mud$ yields, after a little algebra, that
\begin{align}
\label{I3psigma5}
I_3+\sigma^2_5&=\frac{\gamma}{\beta(\beta+\omega)}\mud b(\taut)^2 z\left[\beta-(2\beta+\omega)z\right]\nonumber\\
&\phantom{=\ }+ \frac{\gamma}{\beta[2(\beta+\omega)+\gamma]}b(\taut)^2 z^2\left[\beta(\sigma_D^2+\mud^2)+\omega\mud\right]\nonumber\\
&\phantom{=\ }
-\frac{\gamma[(\beta+\omega+\gamma)z-\gamma]z}{[2(\beta+\omega)+\gamma](\beta+\omega)} \mud b(\taut).
\end{align}

Adding~\eqref{I4pI6a} and~\eqref{I3psigma5}, and comparing with~\eqref{MRCLTvar}, shows that
\begin{equation}
\label{I1pI2etc}
I_1+I_2+I_3+I_4+I_6+\sigma^2_5=\SigmaMR(\beta,\omega,\gamma)-I_A-I_B-I_C-I_D,
\end{equation}
with $I_A, I_B, I_C$ and $I_D$ given by~\eqref{IA}-\eqref{ID}.
Making the substitution $v=\re^{-(\beta+\omega)u}$ in the integrals in~\eqref{integralI5} and~\eqref{integralI7},
using the expressions~\eqref{sumctilde2} and~\eqref{sumchat2} for the respective integrands, shows that
\begin{equation}
\label{I5pI7}
I_5+I_7=I_A+I_B+I_C+I_D.
\end{equation}
The expression~\eqref{MRCLTvar} for $=\SigmaMR(\beta,\omega,\gamma)$ then follows using~\eqref{mrsigma2} and~\eqref{sigmasum}.

\section{Proof of Lemma~\ref{lemma:PGFcomparison}}
\label{app:PGFcomparison}
In this appendix we prove Lemma~\ref{lemma:PGFcomparison}, assuming without loss of generality that $\gamma=1$.  (If $\gamma \ne 1$
then time can be scaled linearly so that $\gamma=1$.)
For $k=1,2,\ldots$, let $X_k^{(\gamma,\omega)}=k-Y_k^{(\gamma,\omega)}$ and $X_k^{(\gamma+\omega,0)}=k-Y_k^{(\gamma+\omega,0)}$. Thus, for example, $X_k^{(\gamma,\omega)}$ is the number of neighbours an infective, $i^*$ say, with $k$ susceptible neighbours fails to infect in the dropping model. For $k,r \in \mathbb{Z}_+$, let $k_{[r]}=k(k-1)\ldots(k-r+1)$ denote a falling factorial, with the convention that $k_{[0]}=1$. Further let $\mu_{k,[r]}^{(\gamma,\omega)}=\E\left[X_{k,[r]}^{(\gamma,\omega)}\right]$, where $X_{k,[r]}^{(\gamma,\omega)}=X_k^{(\gamma,\omega)}(X_k^{(\gamma,\omega)}-1)\ldots(X_k^{(\gamma,\omega)}-r+1)$, be the
$r$th factorial moment of $X_k^{(\gamma,\omega)}$ and define $\mu_{k,[r]}^{(\gamma+\omega,0)}$ analogously for the modified model. Note that $\mu_{k,[r]}^{(\gamma,\omega)}= \mu_{k,[r]}^{(\gamma+\omega,0)}=0$ for all $r>k$. We prove first that
\begin{equation}
\label{facmominequ}
\mu_{k,[r]}^{(\gamma,\omega)} \le \mu_{k,[r]}^{(\gamma+\omega,0)}\quad\mbox{for all } k,r,
\end{equation}
with strict inequality for $2 \le r \le k$, and then consider the Taylor expansions of $f_k^{(\gamma,\omega)}(s)$ and $f_k^{(\gamma+\omega,0)}(s)$ about $s=1$ to prove Lemma~\ref{lemma:PGFcomparison}.

To determine the factorial moment $\mu_{k,[r]}^{(\gamma,\omega)}$, fix $k \ge 1$, give the $k$ neighbours of $i^*$ the labels $1,2,\ldots,k$ and let $A_k^{(\gamma,\omega)}$ be
the set of neighbours that are not infected by $i^*$.  Then, for any $B \subseteq \{1,2,\ldots,k\}$, $\Pp\left(A_k^{(\gamma,\omega)}=B\right)$ depends on $B$ only through
its size $|B|$, so $A_k^{(\gamma,\omega)}$ is a symmetric sampling procedure (Martin-L{\"o}f~\cite{MLof:1986}).  It follows from Lemma 1 in that paper that
$\mu_{k,[r]}^{(\gamma,\omega)}= k_{[r]} P_{k,r}^{(\gamma,\omega)}$ $(r=0,1,\ldots,k)$, where $P_{k,r}^{(\gamma,\omega)}$ is the probability that no one in any fixed set of $r$ neighbours of
$i^*$ is infected by $i^*$, with $P_{k,0}^{(\gamma,\omega)}=1$.  Similarly, in an obvious notation, $\mu_{k,[r]}^{(\gamma+\omega,0)}= k_{[r]} P_{k,r}^{(\gamma+\omega,0)}$ $(r=0,1,\ldots,k)$.
Note that $P_{k,r}^{(\gamma,\omega)}= P_{r,r}^{(\gamma,\omega)}=\Pp\left(Y_r^{(\gamma,\omega)}=0\right)$, so using~\eqref{Mbindrop},
\begin{eqnarray*}
P_{k,r}^{(\gamma,\omega)}&=&\E\left[\left(1-\frac{\beta}{\beta+\omega}\left(1-\re^{-(\beta+\omega)I}\right) \right)^r\right]\\
&=&\E\left[\left(\frac{\omega}{\beta+\omega}+\frac{\beta}{\beta+\omega}\re^{-(\beta+\omega)I}\right)^r\right]\\
&=&\sum_{i=0}^r \binom{r}{i}\left(\frac{\omega}{\beta+\omega}\right)^{r-i} \left(\frac{\beta}{\beta+\omega}\right)^i \E\left[\re^{-i(\beta+\omega)I}\right]\\
&=&\sum_{i=0}^r \binom{r}{i}\left(\frac{\omega}{\beta+\omega}\right)^{r-i} \left(\frac{\beta}{\beta+\omega}\right)^i \frac{1}{1+i(\beta+\omega)},
\end{eqnarray*}
since $I \sim \Exp(1)$.  A similar but simpler argument using~\eqref{Mbinmod} yields
\begin{equation*}
P_{k,r}^{(\gamma+\omega,0)}=\frac{1+\omega}{1+\omega+r\beta}.
\end{equation*}
Thus $\mu_{k,[r]}^{(\gamma,\omega)} \le \mu_{k,[r]}^{(\gamma+\omega,0)}$ for all $k,r$ if and only if
\begin{equation}
\label{inequality1}
\sum_{i=0}^r \binom{r}{i}\left(\frac{\omega}{\beta+\omega}\right)^{r-i} \left(\frac{\beta}{\beta+\omega}\right)^i \frac{1}{1+i(\beta+\omega)}
\le \frac{1+\omega}{1+\omega+r\beta} ,
\end{equation}
$(r=0,1,\ldots)$, which we now show.

First note that both sides of~\eqref{inequality1} equal $1$ when $r=0$. Suppose $r>0$.  Then
\begin{align*}
\sum_{i=0}^r &\binom{r}{i}\left(\frac{\omega}{\beta+\omega}\right)^{r-i} \left(\frac{\beta}{\beta+\omega}\right)^i \frac{1}{1+i(\beta+\omega)}
\le \frac{1+\omega}{1+\omega+r\beta}\\
&\iff \sum_{i=0}^r \binom{r}{i}\left(\frac{\omega}{\beta+\omega}\right)^{r-i} \left(\frac{\beta}{\beta+\omega}\right)^i \left[1-\frac{i(\beta+\omega)}{1+i(\beta+\omega)}\right]\le 1-\frac{r\beta}{1+\omega+r\beta}\\
&\iff \sum_{i=0}^r \binom{r}{i}\left(\frac{\omega}{\beta+\omega}\right)^{r-i} \left(\frac{\beta}{\beta+\omega}\right)^i \frac{i(\beta+\omega)}{1+i(\beta+\omega)} \ge \frac{r\beta}{1+\omega+r\beta}\\
&\iff r \beta \sum_{i=1}^r \binom{r-1}{i-1}\left(\frac{\omega}{\beta+\omega}\right)^{r-i}\left(\frac{\beta}{\beta+\omega}\right)^{i-1}
\frac{1}{1+i(\beta+\omega)} \ge \frac{r\beta}{1+\omega+r\beta}\\
&\iff \sum_{i=0}^{r-1}\binom{r-1}{i}\left(\frac{\omega}{\beta+\omega}\right)^{r-1-i}\left(\frac{\beta}{\beta+\omega}\right)^i
\frac{1}{1+(i+1)(\beta+\omega)} \ge \frac{1}{1+\omega+r\beta}\\
&\iff H(r) \ge 0,
\end{align*}
where
\begin{align*}
&H(r)\\
&=\sum_{i=0}^{r-1} \binom{r-1}{i}\left(\frac{\omega}{\beta+\omega}\right)^{r-1-i}\left(\frac{\beta}{\beta+\omega}\right)^i
\left[\frac{1}{1+(i+1)(\beta+\omega)}-\frac{1}{1+\omega+r\beta}\right].
\end{align*}

Now $H(1)=0$, so $\mu_{k,[1]}^{(\gamma,\omega)} = \mu_{k,[1]}^{(\gamma+\omega,0)}$ $(k=0,1,\ldots)$, as noted (in a different notation) after~\eqref{Mbinmod}.
For $r \ge 2$,
\begin{eqnarray*}
H(r)&=&\sum_{i=0}^{r-1}\binom{r-1}{i}\left(\frac{\omega}{\beta+\omega}\right)^{r-1-i}\left(\frac{\beta}{\beta+\omega}\right)^i
\left[\frac{(r-1-i)\beta-i\omega}{[1+(i+1)(\beta+\omega)](1+\omega+r\beta)}\right]\\
&=&\frac{1}{1+\omega+r\beta}\left(\frac{1}{\beta+\omega}\right)^{r-1}\tilde{H}(r),
\end{eqnarray*}
where
\begin{align*}
\tilde{H}(r)&=\sum_{i=0}^{r-2}\binom{r-1}{i}\omega^{r-1-i}\beta^i\frac{(r-1-i)\beta}{1+(i+1)(\beta+\omega)}\\
&\phantom{=\ }-\sum_{i=1}^{r-1}\binom{r-1}{i}\omega^{r-1-i}\beta^i \frac{i \omega}{1+(i+1)(\beta+\omega)}\\
&=\sum_{i=0}^{r-2}\binom{r-1}{i}\omega^{r-1-i}\beta^{i+1}\frac{(r-1-i)}{1+(i+1)(\beta+\omega)}\\
&\phantom{=\ }-\sum_{i=0}^{r-2}\binom{r-1}{i+1}\omega^{r-1-i}\beta^{i+1}\frac{i+1}{1+(i+2)(\beta+\omega)}\\
&=(r-1)\sum_{i=0}^{r-2} \binom{r-2}{i}\omega^{r-1-i}\beta^{i+1}\left[\frac{1}{1+(i+1)(\beta+\omega)}-\frac{1}{1+(i+2)(\beta+\omega)}\right]\\
&>0.
\end{align*}
Thus, $H(r)>0$ for $r=2,3,\ldots$, proving~\eqref{facmominequ}.

Turning to Lemma~\ref{lemma:PGFcomparison} note that for $k=1,2,\ldots$ and $s \ne 0$, $f_k^{(\gamma,\omega)}(s)=s^k \hat{f}_k^{(\gamma,\omega)}(s^{-1})$, where
$\hat{f}_k^{(\gamma,\omega)}(s)=\E\left[s^{X_k^{(\gamma,\omega)}}\right]$ $(s \in \mathbb{R})$ is the PGF of $X_k^{(\gamma,\omega)}$. Similarly, in an obvious notation,
$f_k^{(\gamma+\omega,0)}(s)=s^k \hat{f}_k^{(\gamma+\omega,0)}(s^{-1})$.  Now, for $s<1$,
\begin{eqnarray*}
\hat{f}_k^{(\gamma,\omega)}(s^{-1}) = \sum_{r=0}^k \mu_{k,[r]}^{(\gamma,\omega)} (s^{-1}-1)^r \le \sum_{r=0}^k \mu_{k,[r]}^{(\gamma+\omega,0)} (s^{-1}-1)^r = \hat{f}_k^{(\gamma+\omega,0)}(s^{-1}),
\end{eqnarray*}
with strict inequality if $k \ge 2$.  Thus, $f_k^{(\gamma,\omega)}(s) \le f_k^{(\gamma+\omega,0)}(s)$ for all $s \in (0,1)$, again with strict inequality if $k \ge 2$,
proving Lemma~\ref{lemma:PGFcomparison} for $s \in (0,1)$.  The lemma holds trivially when $s=1$ since $f_k^{(\gamma,\omega)}(1)=f_k^{(\gamma+\omega,0)}(1)=1$.
Finally, note that $f_k^{(\gamma,\omega)}(0)=\Pp(Y_k^{(\gamma,\omega)}=0)=\Pp(X_k^{(\gamma,\omega)}=k)=\mu_{k,[k]}^{(\gamma,\omega)}/k_{[k]}$ and, similarly, $f_k^{(\gamma+\omega,0)}(0)= \mu_{k,[k]}^{(\gamma+\omega,0)}/k_{[k]}$, so~\eqref{facmominequ} implies the lemma holds also when $s=0$.

\section{Derivation of asymptotic variances in Conjecture~\ref{conj:nodroppingCLT}} 
\label{app:nodroppingCLT}
In this appendix we derive the expressions for $\SigmaMRE(\beta,\gamma)$ and $\SigmaNSW(\beta,\gamma)$ given in Conjecture~\ref{conj:nodroppingCLT} by setting $\omega=0$ in 
Conjectures~\ref{conj:mrCLT} and~\ref{conj:nswCLT}.
We consider first the epidemic on an MR random network.

From~\eqref{integralI5} and~\eqref{I5pI7}, $I_A+I_B+I_C+I_D=I_7$, since $\omega=0$.  We derive a closed-form expression for
$I_7$ when $\omega=0$.  Note that now $p_{\omega}=0$, so using~\eqref{eq:psi} and~\eqref{psi2}, $\psi(t)=\re^{-\beta t}$ and
$\psi_2(\taut,u)=\re^{-\beta(2\taut-u)}$.  Substituting these into~\eqref{sumctilde2} yields
\begin{align}
\label{sumctilde2zero}
\sum_{j=1}^{\infty} \tilde{c}_j(\taut,u)^2j \xt_j(u)&=
(1-b(\taut))^2 \re^{-\beta(2\taut-u)}\fde'\left(\re^{-\beta(2\taut-u)}\right)\nonumber\\
&\phantom{=\ } +b(\taut)(3b(\taut)-2)\re^{-2\beta(2\taut-u)}\fde^{(2)}\left(\re^{-\beta(2\taut-u)}\right)\nonumber\\
&\phantom{=\ }+b(\taut)^2 \re^{-3\beta(2\taut-u)}\fde^{(3)}\left(\re^{-\beta(2\taut-u)}\right).
\end{align}
For $k=0,1,\ldots$, let
\[
J_k=\int_0^{\taut}  \re^{-k\beta(2\taut-u)}\fde^{(k)}\left(\re^{-\beta(2\taut-u)}\right) \,{\rm d}u.
\]
Integrating by parts, for $k=1,2,\ldots$,
\begin{align}
\label{reduction}
J_k&=\left[\re^{-(k-1)\beta(2\taut-u)} \frac{1}{\beta} \fde^{(k-1)}\left(\re^{-\beta(2\taut-u)}\right)\right]_0^{\taut}\nonumber\\
&\phantom{=\ }-\int_0^{\taut} (k-1) \beta \re^{-(k-1)\beta(2\taut-u)}\frac{1}{\beta} \fde^{(k-1)}\left(\re^{-\beta(2\taut-u)}\right)\,{\rm d}u\nonumber\\
&=\frac{1}{\beta}\left[\re^{-(k-1)\beta\taut} \fde^{(k-1)}\left(\re^{-\beta \taut}\right)-
\re^{-2(k-1)\beta\taut}\fde^{(k-1)}\left(\re^{-2\beta \taut}\right)\right]-(k-1)J_{k-1},
\end{align}
so, setting $k=1$,
\begin{equation}
\label{J1integral}
J_1=\frac{1}{\beta}\left[\fde\left(\re^{-\beta \taut}\right)-\fde\left(\re^{-2\beta \taut}\right)\right].
\end{equation}
Substituting~\eqref{sumctilde2zero} into~\eqref{integralI7}, and  using~\eqref{J1integral} and
\eqref{reduction} with $k=2,3$ yields
\begin{align}
\label{I7integralzero}
I_7&=\fde\left(\re^{-\beta \taut}\right)-\fde\left(\re^{-2\beta \taut}\right)\nonumber\\
&\phantom{=\ }+b(\taut)(b(\taut)-2)\left[\re^{-\beta\taut} \fde'\left(\re^{-\beta \taut}\right)-
\re^{-2\beta\taut}\fde'\left(\re^{-2\beta \taut}\right)\right]\nonumber\\
&\phantom{=\ }+b(\taut)^2 \left[\re^{-2\beta\taut} \fde''\left(\re^{-\beta \taut}\right)-
\re^{-4\beta\taut}\fde''\left(\re^{-2\beta \taut}\right)\right].
\end{align}

Recall that $z=\re^{-\beta \taut}$ and $\btz=b(\taut)$.  Setting $\omega=0$ in~\eqref{btautzfd2b}
gives
\[
\btz z \fde''(z)=\left[\frac{(\beta+\gamma)(1+\btz)z-\gamma}{\beta}\right]\mud.
\]
Substituting these into~\eqref{I7integralzero} and using~\eqref{zdefND} yields
\begin{align*}
%\label{I7integralzeroa}
I_7&=\fde\left(z\right)-\fde\left(z^2\right)
-\btz(\btz-2)z^2\fde'\left(z^2\right)-\btz^2 z^4 \fde''\left(z^2\right)\nonumber\\
&\quad+\btz^2 z \left(\frac{2(\beta+\gamma)z-\gamma}{\beta}\right)\mud-\btz z \left(\frac{(\beta+\gamma)z-\gamma}{\beta}\right)\mud.
\end{align*}
Setting $\omega=0$ in~\eqref{MRCLTvar} and recalling that now $I_A+I_B+I_C+I_D=I_7$ then yields
\begin{align}
\label{MRCLTvarNDa}
\SigmaMRE(\beta,\gamma)&=\fde\left(z\right)-\fde\left(z^2\right)
-\btz(\btz-2)z^2\fde'\left(z^2\right)-\btz^2 z^4 \fde''\left(z^2\right)\nonumber\\
&\phantom{=\ }+\left(\frac{\gamma}{2\beta+\gamma}\right)\btz^2 z^2 (\sigma_D^2+\mud^2)\nonumber\\
&\phantom{=\ }+2\left(\frac{\gamma-(\beta+\gamma)z}{\beta}\right)\left(\frac{\beta+\gamma}{2\beta+\gamma}\right)z\btz \mud\nonumber\\
&\phantom{=\ }+2\left(\frac{\gamma-(\beta+\gamma)z}{\beta}\right)^2 z^2\btz^2\mud.
\end{align}
Setting $\omega=0$ in~\eqref{btaut1} shows that $h(\beta,\gamma,z)=z\btz$, where $h(\beta,\gamma,z)$
is defined at~\eqref{hdef}.  Further $\fde\left(z\right)=1-\rho$; see immediately after~\eqref{zdefND}.
The expression~\eqref{MRCLTvarND} for $\SigmaMRE(\beta,\gamma)$ then follows immediately from~\eqref{MRCLTvarNDa}.

Turning to the epidemic on an NSW random network, setting $\omega=0$ in~\eqref{sigma0} and noting that
then $\psit(z)=z$, yields
\begin{align}
\label{sigma0ND}
\sigma_0^2(\beta, 0, \gamma)&=f_D\left(z^2\right)-(1-\rho)^2+\btz^2 z^4 f_D''\left(z^2\right)+\btz(\btz-2)z^2\fde'\left(z^2\right)\nonumber\\
&\quad+\btz^2z^2\left(\frac{(\beta+\gamma)z-\gamma}{\beta}\right)^2\left(\sigma_D^2+\mud^2\right)\nonumber\\
&\quad-2 \left(\frac{(\beta+\gamma)z-\gamma}{\beta}\right)
\left(\frac{(\beta+\gamma)z-\gamma}{\beta}+\frac{(\beta+\gamma)}{\beta}z\right)z^2 \btz^2 \mud.
\end{align}
Setting $\omega=0$ in~\eqref{sigmatNSW1} shows that $\SigmaNSW(\beta,\gamma)$ is given by the sum of the right-hand sides of~\eqref{MRCLTvarNDa}, with $D_{\epsilon}$ replaced by $D$, and~\eqref{sigma0ND}.  The expression~\eqref{NSWCLTvarND} for $\SigmaNSWE(\beta,\gamma)$ now follows since $f_D(z)=1-\rho$ and $h(\beta,\gamma,z)=z\btz$.

\section{ODE initial conditions for the epidemic on an NSW graph}
\label{app:ODEinitConds}
In this appendix we derive the initial conditions $\Sigma_{\rm NSW}(0)$ that are given in Section~\ref{sec:implementation}. We assume that the number of initial infectives is $i^N_0 = [\epsilon N]$ (or that $i^N_0$ is any function of $N$ such that $\lim_{N\to\infty} N^{-1} i^N_0 = \epsilon$) and that these individuals are chosen uniformly from the population. Since there is nothing special about the labelling of the individuals $1,2,\dots,N$ in the population we can assume that individuals $1,2,\dots,i^N_0$ are initially infected.

First consider the term $\sigma_{x_i,x_i}(0) = \lim_{N\to\infty} N^{-1} \var(X_i^N(0))$. Writing $X_i^N(0)$ as a sum of indicator variables, we use the independence of different individuals' degrees to find that
\begin{align*}
\var(X_i^N(0)) & = \var(\sum_{k=1}^N \ind{\textrm{indiv $k$ is deg $i$ \& susc}}) \\
& = \var(\sum_{k=i_0^N+1}^N \ind{\textrm{indiv $k$ is deg $i$}}) \\
  & = \sum_{k=i_0^N+1}^N \var(\ind{\textrm{indiv $k$ is deg $i$}}) \\
 & = (N-i_0^N) p_i(1-p_i),
\end{align*}
so
\begin{equation*}
\sigma_{x_i,x_i}(0) = \lim_{N\to\infty} N^{-1} \var(X_i^N(0)) = (1-\epsilon) p_i (1-p_i) %\label{eq:Sigmaxixi0}
\end{equation*}
for all $i$. Considering infectives instead, essentially the same arguments establish that
\begin{equation*}
\sigma_{y_i,y_i}(0) = \lim_{N\to\infty} N^{-1} \var(Y_i^N(0)) = \epsilon p_i (1-p_i).
\end{equation*}

For the covariances we use the same independence and $\cov(X,Y)=E[XY]-E[X]E[Y]$ to find that, for $i\neq j$,
\begin{align*}
\cov(X_i^N(0),X_j^N(0)) & = \cov(\sum_{k=1}^N \ind{\textrm{indiv $k$ is deg $i$ \& susc}},\sum_{l=1}^N \ind{\textrm{indiv $l$ is deg $j$ \& susc}}) \\
& = \cov(\sum_{k=i_0^N+1}^N \ind{\textrm{indiv $k$ is deg $i$}},\sum_{l=i_0^N+1}^N \ind{\textrm{indiv $l$ is deg $j$}}) \\
 & = \sum_{k=i_0^N+1}^N \sum_{l=i_0^N+1}^N \cov(\ind{\textrm{indiv $k$ is deg $i$}},\ind{\textrm{indiv $l$ is deg $j$}}) \\
 & = \sum_{k=i_0^N+1}^N \cov(\ind{\textrm{indiv $k$ is deg $i$}},\ind{\textrm{indiv $k$ is deg $j$}}) \\
 & = (N-i_0^N) (0 - p_ip_j) = -(N-i_0^N) p_ip_j,
\end{align*}
so that we have
\begin{equation*}
\sigma_{x_i,x_j}(0) = \lim_{N\to\infty} N^{-1} \cov(X_i^N(0),X_j^N(0)) = -(1-\epsilon) p_ip_j.
\end{equation*}
The same calculations for $\cov(Y_i^N(0),Y_j^N(0))$ yield
\begin{equation*}
\sigma_{y_i,y_j}(0) = \lim_{N\to\infty} N^{-1} \cov(Y_i^N(0),Y_j^N(0)) = -\epsilon p_ip_j
\end{equation*}
for $i\neq j$. Next, for all $i,j$,
\begin{align*}
\cov(X_i^N(0),Y_j^N(0)) & = \cov(\sum_{k=1}^N \ind{\textrm{indiv $k$ is deg $i$ \& susc}},\sum_{l=1}^N \ind{\textrm{indiv $l$ is deg $j$ \& inf}}) \\
& = \cov(\sum_{k=i_0^N+1}^N \ind{\textrm{indiv $k$ is deg $i$}},\sum_{l=1}^{i_0^N} \ind{\textrm{indiv $l$ is deg $j$}}) \\
 & = \sum_{k=i_0^N+1}^N \sum_{l=1}^{i_0^N} \cov(\ind{\textrm{indiv $k$ is deg $i$}},\ind{\textrm{indiv $l$ is deg $j$}}) \\
 & = 0
\end{align*}
by independence of individuals (there are no terms with $k=l$ since the indices take values in disjoint sets). Thus
\begin{equation*}
\sigma_{x_i,y_j}(0) = \lim_{N\to\infty} N^{-1} \cov(X_i^N(0),Y_j^N(0)) = 0.
\end{equation*}
Finally, we have $Z_E^N(0) = 0$, so all (co)variances involving it are zero and remain so when divided by $N$, whence for all $i$ we have
\begin{equation*}
\sigma_{x_i,z_E}(0) = \sigma_{y_i,z_E}(0) = \sigma_{z_E,z_E}(0) = 0.
\end{equation*}

\paragraph{Acknowledgements}
This work was partially supported by a grant from the Simons Foundation and was carried out as a result of the authors' visit to the Isaac Newton Institute for Mathematical Sciences during the programme Theoretical Foundations for Statistical Network Analysis in 2016 (EPSRC Grant Number EP/K032208/1). KYL is supported by the Swedish Research Council (VR) Grant Number 2015-05015. This work was also supported by a grant from the Knut and Alice Wallenberg Foundation, which enabled FB to be a guest professor at the Department of Mathematics, Stockholm University.  We thank Phil Pollett for some helpful discussions relating to Theorem~\ref{KurtzFCLTrandinit}
and the reviewers for their constructive comments which have improved the presentation of the paper.

\bibliographystyle{spmpsci}
\bibliography{mybib3a}

\end{document}